\documentclass[reqno,11pt]{amsart}

\usepackage[utf8]{inputenc}
\usepackage[T1]{fontenc} 

\usepackage{amssymb}
\usepackage{amsmath}
\usepackage{amsfonts}
\usepackage{amsthm}
\usepackage{mathtools}
\usepackage{mathabx}
\usepackage{fullpage}
\usepackage{comment}
\usepackage[dvipsnames]{xcolor}
\usepackage{enumitem}

\usepackage{url} 
\usepackage{hyperref}
\usepackage{bbold}
\usepackage{cite}

\usepackage{caption}
\usepackage{subcaption} 

\usepackage{pgfplots,tikz}
\pgfplotsset{compat=1.17}
\usetikzlibrary{decorations.markings}
\usetikzlibrary{calc} 
\usetikzlibrary{positioning}

\usepackage{fancyhdr}
\usepackage{graphicx}

\theoremstyle{plain}%
\newtheorem{theorem}{Theorem}[section]
\newtheorem{lemma}[theorem]{Lemma}
\newtheorem{proposition}[theorem]{Proposition}

\theoremstyle{definition}
\newtheorem{definition}[theorem]{Definition}
\newtheorem{example}[theorem]{Example}

\theoremstyle{remark}
\newtheorem{remark}[theorem]{Remark}

\let \leq \leqslant
\let \geq \geqslant

\DeclareMathOperator{\sgn}{sgn}
\DeclareMathOperator{\cone}{Cone}
\DeclareMathOperator{\support}{supp}

\DeclareMathOperator*{\argmax}{arg\,max}
\DeclareMathOperator*{\argmin}{arg\,min}

\definecolor{detailcolor00}{rgb}{0.4405, 0.204, 0.343}
\definecolor{detailcolor01}{rgb}{0.546, 0.215, 0.352}
\definecolor{detailcolor02}{rgb}{0.675, 0.247, 0.387} 
\definecolor{detailcolor03}{rgb}{0.775, 0.317, 0.455}
\definecolor{detailcolor04}{rgb}{0.830, 0.421, 0.553} 
\definecolor{detailcolor05}{rgb}{0.831, 0.533, 0.663}
\definecolor{detailcolor06}{rgb}{0.779, 0.619, 0.775}
\definecolor{detailcolor07}{rgb}{0.724, 0.694, 0.827}
\definecolor{detailcolor08}{rgb}{0.687, 0.770, 0.880}
\definecolor{detailcolor09}{rgb}{0.671, 0.839, 0.904}
\definecolor{detailcolor10}{rgb}{0.659, 0.872, 0.882}

\title{Wilson loops in finite abelian lattice gauge theories}

\author{Malin P. Forsstr\"om}
\address[Malin P. Forsstr\"om]{Department of Mathematics, KTH Royal Institute of Technology, 100 44 Stockholm, Sweden.}
\email{malinpf@kth.se}

\author{Jonatan Lenells}
\address[Jonatan Lenells]{Department of Mathematics, KTH Royal Institute of Technology, 100 44 Stockholm, Sweden.}
\email{jlenells@kth.se}

\author{Fredrik Viklund}
\address[Fredrik Viklund]{Department of Mathematics, KTH Royal Institute of Technology, 100 44 Stockholm, Sweden.}
\email{fredrik.viklund@math.kth.se}

\begin{document}

\maketitle
\begin{abstract}
    We consider lattice gauge theories on $\mathbb{Z}^4$ with Wilson action and structure group $\mathbb{Z}_n$. We compute the expectation of Wilson loop observables to leading order in the weak coupling regime, extending and refining a recent result of Chatterjee. Our proofs use neither duality relations nor cluster expansion techniques.
\end{abstract}

\section{Introduction}
\subsection{Background}
An important challenge in mathematical physics is to give precise mathematical meaning to quantum field theories (QFTs) that appear in the Standard Model. These QFTs describe fundamental forces in nature and their interaction with elementary particles. One way to proceed is to perform a Wick rotation,  discretize space-time by a four-dimensional Euclidean lattice, and then try to approximate the continuum QFT by a well-defined probabilistic theory defined on the lattice. Such lattice models are called lattice gauge theories and have been studied since at least the 1970s (see, e.g.,~\cite{wilson}). The hope is that one can take a scaling limit of the lattice model and in this way somehow obtain a rigorously defined continuum field theory. This has proved difficult except in simpler cases of limited physical relevance. However, it is also natural to study lattice gauge theories in their own right as statistical mechanics models, and this is the point of view we take here. We refer the reader to \cite{gl2010} for further background and references.

\subsection{Lattice gauge theories with Wilson action}
In order to state our main results, we need to give a few definitions. The lattice \( \mathbb{Z}^4\) has a vertex at each point in \( \mathbb{R}^4\) with integer coordinates, and an edge between nearest neighbors, oriented in the positive direction, so that there are exactly four positively oriented edges emerging from each vertex \( a \), denoted by \( \frac{\partial}{\partial x^i}\big|_a , \,i=1,\ldots,4 \). We will let \( -\frac{\partial}{\partial x^i}\big|_a \) denote the edge with the same end points as \( \frac{\partial}{\partial x^i}\big|_a \) but with opposite orientation.
For \( i<j \), the closed set of points forming the lattice square determined by a given pair \( \frac{\partial}{\partial x^i}\big|_a \) and \( \frac{\partial}{\partial x^j}\big|_a \) is called a $2$-cell. This $2$-cell in turn determines a \emph{positively oriented plaquette} \( \frac{\partial}{\partial x^i}\big|_a \wedge \frac{\partial}{\partial x^j}\big|_a \), and a \emph{negatively oriented plaquette} \( -\frac{\partial}{\partial x^i}\big|_a \wedge \frac{\partial}{\partial x^j}\big|_a .\)

Given an integer $N \geq 1$, let $B_N =  [-N,N]^4 \cap \mathbb{Z}^4$. Write $C_1(B_N)$ for the set of (positively and negatively) oriented edges whose endpoints are  contained in $B_N$, and $C_2(B_N)$ for the set of oriented plaquettes whose boundary edges (defined in the obvious way) are all contained in $C_1(B_N)$. We will often write $e$ and $p$ for elements of $C_1(B_N)$ and $C_2(B_N)$, respectively.

In this paper we will take as structure group $G$ the finite additive group $\mathbb{Z}_n$. We will only consider faithful and one-dimensional representations of $G$. If $\rho$ is such a representation, then 
\begin{align}\label{rhoexplicit}
\rho(k) = e^{k \cdot 2 \pi i m/n}, \qquad k \in \mathbb{Z}_n,
\end{align}
for some \(m \in \{1, \dots, n-1\}\) relatively prime to \( n \). In particular, \(\rho\) is unitary and \(\rho(G) =  \{ e^{k \cdot 2 \pi i/n} \}_{k \in \{ 1,2, \ldots, n\}}\) consists of the $n^{\textrm{th}}$ roots of unity. (See Remark \ref{generalgroupremark} for a discussion of the generalization of our results to finite abelian structure groups and to representations of arbitrary finite dimension.) Next, let \(\Omega^1(B_N,G)\) be the set of \( G \)-valued 1-forms on  \(C_1(B_N) \), which can be thought of as the set of functions \( \sigma \colon C_1(B_N) \to G \) with the property that \( \sigma(e) = -\sigma(-e) \). (See Section~\ref{sec: ddf} for the precise definition.)
Each element \(\sigma  \in \Omega^1(B_N,G) \) induces a spin configuration \( d\sigma \) defined on plaquettes according to
\begin{equation}\label{p-e}
d\sigma(p) \coloneqq \sigma(e_1)+\sigma(e_2)+\sigma(e_3) +\sigma(e_4), \quad p \in  C_2(B_N),
\end{equation}
where \(e_1,e_2, e_3, e_4\) are the edges in the boundary of \(p\), directed according to the orientation of the plaquette \( p \), see Section~\ref{sec: oriented cells}. The set of all spin configurations on \( C_2(B_N) \) which arise in this way will be denoted by \( \Omega_0^2(B_N,G) \). %
With \( N \), \( G \), and \( \rho \) fixed, we define the \emph{Wilson action}, corresponding to the structure group \(G\) and representation \( \rho \), by 
\[
S(\sigma) 
\coloneqq - \sum_{p \in C_2(B_N)}  \Re   \rho\bigl(d\sigma(p)\bigl), \quad \sigma \in \Omega^1(B_N,G).
\]
Letting \( \mu_H \) denote the uniform measure on \( \Omega^1(B_N,G)\), we obtain an associated probability measure \(\mu_{N,\beta}\) on \( \Omega^1(B_N,G) \) by weighting \( \mu_H \) by the Wilson action:
\begin{align}\label{mubetaNdef}
\mu_{N,\beta}(\sigma) \coloneqq Z^{-1}_{N,\beta} e^{-\beta S(\sigma)} \,  \mu_H(\sigma), \quad \sigma \in \Omega^1(B_N,G).
\end{align}
Here \( Z_{N,\beta} \) is a normalizing constant that ensures that \( \mu_{N,\beta} \) is a probability measure.

While lattice gauge theories with a finite abelian structure group are not (as far as we are aware) of known direct physical significance in the context of the Standard Model, they (including the choice $\mathbb{Z}_n$) have been studied in the physics literature, see, e.g., \cite{frolich-spencer82, b1984} and the references therein.

\subsection{Main result}
In a seminal paper \cite{wilson}, Wilson introduced a particular observable in order to study quark confinement in lattice gauge theories, which we now define (see also e.g., \cite{gl2010}). 
A \( 1 \)-chain is a formal sum of positively oriented edges with integer coefficients, see Section~\ref{sec: chains} below. The support of a \(1\)-chain \(\gamma\), written \(\support \gamma\), is the set of directed edges with non-zero coefficient. 
We say that a \(1\)-chain with finite support is a \emph{generalized loop} if it has coefficients in \(\{-1,0,1\}\) and empty boundary, see Definition~\ref{def: generalized loop}. Roughly speaking, this means that a generalized loop is a disjoint union of a finite number of closed loops, where each closed loop is a nearest-neighbor path in the graph \( \mathbb{Z}^4  \) starting and ending at the same vertex. For example, any rectangular loop as well as any finite disjoint union of such loops is a generalized loop.
The length \(\ell = |\support \gamma|\) of a loop \(\gamma\) is the number of edges in \( \support \gamma\). An edge \( e \in \support \gamma \) is said to be a corner edge in \( \gamma \) if there is another edge \( e' \in \pm \support \gamma \) such that \( e \) and \( e' \) are both in the boundary of some common plaquette.
Given a generalized loop $\gamma $, the Wilson loop observable \( W_\gamma \) is defined by
\[
W_\gamma \coloneqq W_\gamma(\sigma) \coloneqq  \rho \bigl( \sigma(\gamma) \bigr) \coloneqq  \rho \Bigl(\sum_{e \in \support \gamma} \sigma(e) \Bigr), \quad \sigma\in \Omega^1(B_N,G). 
\]
When \( f \colon \Omega^1(B_N,G)\to \mathbb{C} \), we let  
\begin{equation*}
    \mathbb{E}_{N,\beta}[f] \coloneqq \sum_{\sigma \in \Omega^1(B_N,G)} f(\sigma)\, \mu_{N,\beta} (\sigma)
\end{equation*}
denote the expected value of \( f(\sigma) \). 
We will apply Ginibre's inequality in Section~\ref{sec: ginibre} to prove that, for a fixed $\gamma$, the limit
\begin{align}\label{wilsonlimit}
\langle W_\gamma \rangle_\beta \coloneqq \lim_{N \to \infty } \mathbb{E}_{N,\beta}[W_\gamma]
\end{align}
exists and is translation invariant.  

We can now state our main result which extends Theorem~1.1 of \cite{c2019} to the case when the structure group is given by \( \mathbb{Z}_n \) for some \( n \geq 2 \), and the representation is unitary, faithful, and one-dimensional. In the statement,  for \( \beta \geq 0 \) we let
\begin{equation}\label{Thetadef}
    \theta(\beta) \coloneqq \frac{\sum_{g \in G}\rho(g) e^{12\beta \Re\rho(g)}}{\sum_{g \in G} e^{12\beta \Re \rho(g)}}
\end{equation}
and
\begin{equation}\label{lambdadef}
    \lambda(\beta) \coloneqq \max_{g \in G \smallsetminus \{ 0 \}} \frac{e^{\beta \Re \rho(g)} }{ e^{\beta \Re  \rho(0)}} = \max_{g \in G \smallsetminus \{ 0 \}} e^{\beta (\Re \rho(g)-1)}  
    =  e^{-\beta (1-\cos(2\pi/n))}.
\end{equation}
Note that \( \lim_{\beta \to \infty} \theta(\beta) = 1.\)

\begin{theorem}\label{theorem: Chatterjee's main theorem}
   Let $n \geq 2$ be an integer. Consider lattice gauge theory with structure group $G = \mathbb{Z}_n$ and a faithful one-dimensional representation \( \rho \) of \( G \). 
    Then, for all sufficiently large \( \beta_0 >0\) there are constants \(K' = K'(\beta_0)\) and \(K'' = K''(\beta_0) \) such that for any \( \beta \geq \beta_0 \) and any generalized loop \( \gamma \) in \( \mathbb{Z}^4 \), if $\ell = |\support \gamma|$ and \( \ell_c \) is the number of corner edges in \( \gamma \), then
    \begin{equation} \label{eq: Chatterjee's main equation}
        \biggl| \langle W_\gamma \rangle_\beta-     e^{-\ell (1 - \theta(\beta))} \biggr|  
        \leq 
        K' \Bigl[ \sqrt\frac{\ell_c}{\ell} + \lambda(\beta)^2 \Bigr]^{\mathrlap{K''}}.
    \end{equation}
\end{theorem}
     In essence, Theorem~\ref{theorem: Chatterjee's main theorem} says that if \( \ell  (1-\theta(\beta)) \) is very large, then \( \langle W_\gamma \rangle_\beta \approx 0 \), if \( \ell  (1-\theta(\beta))  \) is very small, then \( \langle W_\gamma \rangle_\beta \approx 1 \), and otherwise \( \langle W_\gamma \rangle_\beta \) has a non-trivial behavior.
     If \( \beta \) is large compared to \( \ell \) (in the sense that \( \ell  (1-\theta(\beta)) \) is very small), then it is very likely that no plaquette \( p \) close to \( \gamma \)  satisfies \( d\sigma(p) \neq 0\). This will ensure that \( \langle  W_\gamma \rangle_\beta \approx 1 \).
     On the other hand, if  \( \beta \) is small compared to \( \ell \) (in the sense that \( \ell  (1-\theta(\beta))  \) is very large), then with high probability there are many plaquettes \( p \) close to \( \gamma \) with \( d\sigma(p) \neq 0\). In fact, there will be sufficiently many such plaquettes to introduce independence in the model, which will ensure that \( \langle  W_\gamma \rangle_\beta \approx 0 \).
     Finally, if \( \ell  (1-\theta(\beta)) \) is neither very small nor very large, then the model turns out to behave as a certain Poisson process, and this will allow us to use a resampling trick at each edge in \( \gamma \) to approximate $\langle W_\gamma \rangle_\beta$ with \( \theta(\beta)^{\ell} \approx e^{\ell(1-\theta(\beta))} \). The constant  \( \theta(\beta) \) appears naturally in this context since it describes the expected spin at a plaquette adjacent to a given edge, conditioned on the event that all plaquettes adjacent to this edge have the same spin. Recalling that \( n \geq 2 \) and defining \( \xi \coloneqq 1 - \cos(2 \pi/n) > 0\), we have \( \lambda(\beta) = e^{-\beta \xi} \) and \( 1 - \theta(\beta ) \sim  (1 + \mathbb{1}_{n \geq 3})\xi  e^{-12\beta\xi } \) as \( \beta \to \infty \). We see that both $\lambda(\beta)$ and $1 - \theta(\beta)$ are strictly positive and decay exponentially to zero as $\beta \to \infty$. In particular, this implies that the right-hand side of~\eqref{eq: Chatterjee's main equation} is small when \( \beta \) is large and \( \gamma \) is a long loop with relatively few corners.
 \begin{remark}
    In Section~\ref{section: proof of main result}, where the proof of Theorem~\ref{theorem: Chatterjee's main theorem} is completed, we give explicit expressions for the constants \( K' \) and \( K'' \) (see Equation \eqref{CpCppexpressions}). It follows from these expressions that the constant \( K'' \) can be taken arbitrarily close to \( 1/5 \) by  choosing \( \beta_0 \)  sufficiently large. 
\end{remark}

\begin{remark}
    For Theorem~\ref{theorem: Chatterjee's main theorem} to hold, \( \beta_0>0 \) needs to be sufficiently large. In fact, it follows from the proof of Theorem~\ref{theorem: Chatterjee's main theorem} that \( \beta_0 \) is sufficiently large if \( 5 (|G|-1) \lambda(\beta_0)^2 < 1 \) and the following inequality is satisfied for all \( \beta \geq \beta_0 \).
    \begin{equation} 
        \max_{g_1,\ldots,g_6 \in G} \; \biggl[\,   \frac{\sum_{g\in  G}
         e^{-2\beta \sum_{k=1}^6  \Re \rho(g+ g_k)} }{\max_{g \in G}    e^{-2\beta \sum_{k=1}^6 \Re \rho(g+ g_k)}  }
         -
         \Bigl| \argmax_{g \in G}    e^{-2\beta \sum_{k=1}^6 \Re \rho(g+ g_k)} \Bigr|
         \biggr] < \frac{1-\cos(2 \pi/n)}{8}.
    \end{equation}
\end{remark}
 
\begin{remark}
    Our proofs can easily be modified to work for closely related actions as well, such as the Villain action. Also, with very small modifications, they work for any lattice \( \mathbb{Z}^m \) with \( m \geq 3 \).
\end{remark}

\begin{remark}\label{generalgroupremark}
    The proof of the main result can, with some work, be generalized to the case of a general finite abelian structure group with an arbitrary unitary, faithful, and irreducible representation. In fact, most of the results in this paper are not hard to generalize to this case. However, these adaptations are of technical nature and, we feel, only of limited interest. Since at the time of posting this paper on ArXiv, the paper~\cite{cao20} contained an alternative generalization in this direction, we choose to state and prove a weaker result in order to keep the paper shorter and more transparent.
\end{remark}

\subsection{Outline of the proof} Let us discuss the main ideas of the proof of Theorem~\ref{theorem: Chatterjee's main theorem}. 

First, we introduce a notion of oriented surface which is used to rewrite the Wilson loop expectation as a function of plaquette configurations on a surface bounded by the loop rather than as a function of spin configurations along the loop. One reason why it is  convenient to work with surfaces and plaquette configurations is that if \( \beta \) is large and \( \sigma \sim \mu_{N,\beta} \), then the probability that \( d\sigma(p) \neq 0 \) for any given plaquette \( p \) is very small, while the same does not hold for the probability that \( \sigma(e) \neq 0 \) for an edge \( e \).
Second, we introduce a notion of \emph{irreducible} spin configuration. Using this notion, one can think of \( \mu_{N,\beta} \) as a measure on \emph{vortex} configurations, that is, on sums of irreducible spin configurations. (The term vortex is used in many related but slightly different senses in the literature; see Section~\ref{sec: vortices} for the precise definition that we use here.)
A novel and central feature of our approach is a result (see Lemma~\ref{lemma: vortex flip}) which estimates how likely it is that a given plaquette is in the support of a vortex. This result allows us to understand how likely it is that a vortex of a given size influences the value of \( W_\gamma \). In particular, it allows us to deduce that when \(  \ell \lambda(\beta)^{12}\) is small, then with high probability, only very small vortices close to $\gamma$ will influence \( W_\gamma \). Using a resampling argument, this essentially completes the proof in this case. In the case when \( \ell \lambda(\beta)^{12}\) is large, then, again using a resampling argument, we show that \( W_\gamma \) is noisy enough for its expectation to be very close to zero.

\subsection{Relation to other work} 
This paper was inspired by, and builds upon, a recent paper of Chatterjee \cite{c2019}. There, pure lattice gauge theory with Wilson action and structure group $G=\mathbb{Z}_2$ was considered. Chatterjee obtained an expression for the leading term of the expected value of a  Wilson loop in the limit as $\beta$ and $\ell$ tend to infinity simultaneously. 
This paper extends and refines Chatterjee's result to the case of a general finite cyclic structure group and a faithful one-dimensional representation. 
Apart from the more general setting of our paper, our argument replaces an argument based on the duality of the model on the \( \mathbb{Z}^4 \)-lattice and on Dobrushin's criterion with a simpler Peierl's type argument, thus simplifying the proof even in the case \( G = \mathbb{Z}_2 \) (see Lemma~\ref{lemma: vortex flip}). This also allows for direct generalizations to the lattices \( \mathbb{Z}^m \) for \( m \geq 3 \). Also, we extend and refine several of the ideas presented in~\cite{c2019} to a more general setting.

While we were in the final stages of preparing this paper, Cao's preprint \cite{cao20} appeared. Cao obtains a more general result than our main theorem, which also allows for non-abelian finite groups. However, we believe the present paper, while less general, is still of interest because our approach is quite different from Cao's. Whereas the approach of \cite{cao20} relies on a cluster expansion and Stein's method even in the case considered in this paper, ours is based on a rather elementary estimate of the probability for a plaquette to be part of a vortex of a given size. We also note that our error term is significantly smaller than the one in~\cite{cao20} for large $\beta$ if e.g.\ \( \sqrt{\ell_c/\ell} \leq O \bigl( \lambda(\beta)^2 \bigr)\) or \( \ell_c \leq O(\ell^{5/6})\) (see also Remark~\ref{remark: other bounds}).

\subsection{Organization of the paper} 
In Section~\ref{sec: preliminaries}, we give the necessary background on discrete exterior calculus which will be needed for the rest of the paper. In particular, in Section~\ref{sec: oriented surfaces}, we define oriented surfaces, and describe how these naturally correspond to oriented loops in the lattice.
In Section~\ref{sec: mubetaNsubsec}, we describe how \( \mu_{N,\beta} \) can be viewed as a measure on \( \Omega^2_0(B_N,G) \) rather than as a measure on \( \Omega^1(B_N,G) \) when considering \( W_\gamma. \) Next, in Section~\ref{sec: ginibre}, we recall Ginibre's inequality, and show how it implies the existence of the infinite volume limit of \( \mathbb{E}_{N,\beta}[W_\gamma]. \)
In Section~\ref{sec: vortices}, we give our definition of vortex and establish related terminology as well as some useful results.  Finally, in Section~\ref{section: proof of main result}, we give a proof of our main result.

\subsection*{Acknowledgements}  
MPF acknowledges support from the European Research Council, Grant Agreement No. 682537, and the Swedish Research Council, Grant No. 2015-05430.
JL is grateful for support from the G\"oran Gustafsson Foundation, the Ruth and Nils-Erik Stenb\"ack Foundation, the Swedish Research Council, Grant No. 2015-05430, and the European Research Council, Grant Agreement No. 682537.
FV acknowledges support from the Knut and Alice Wallenberg Foundation, the Swedish Research Council, and the Ruth and Nils-Erik Stenb\"ack Foundation. Thanks go to Juhan Aru for comments on an earlier version of our paper. We are grateful for the many very helpful comments provided by an anonymous referee.

\section{Preliminaries}\label{sec: preliminaries}

This section collects definitions and known results needed in the proof of our main result. In some cases, if we could not find a clean reference, we have included derivations of results even though they could be considered well-known.

\subsection{The cell complex}
In this section, we introduce notation for the cell complexes of the lattices \( \mathbb{Z}^m \) and \( B_N \coloneqq [-N,N]^m \cap  \mathbb{Z}^m \) for \( m,N \geq 1 \).

Any set \( B \) of the form \( \bigl( [a_1,b_1] \times \cdots \times [a_m,b_m] \bigr) \cap \mathbb{Z}^m\) where, for each \( j \in \{ 1,2, \ldots, m \} \), \( \{a_j, b_j\}  \subset \mathbb{Z} \) satisfies \(a_j < b_j\), will be referred to as a  \emph{box}. If all the intervals \( [a_j,b_j]\), \(1 \leq j \leq m\), have the same length, then the set \( \bigl( [a_1,b_1] \times \cdots \times [a_m,b_m] \bigr) \cap \mathbb{Z}^m\) will be referred to as a {\it cube}.

To simplify notation, we define \( e_1 \coloneqq (1,0,\dots,0) \), \( e_2 \coloneqq (0,1,0,\dots,0) \), \ldots, \( e_m \coloneqq (0,\dots,0,1) \).

\subsubsection{Non-oriented cells}

When \( a \in \mathbb{Z}^m \), \( k \in \{ 0,1, \dots, m \} \), and \( \{ j_1,\dots, j_k \} \subseteq \{ 1,2, \dots, m \} \), we say that the set
\begin{equation*}
    (a; e_{j_1}, \dots, e_{j_k}) \coloneqq \bigl\{ x \in \mathbb{R}^m \colon \exists b_1, \dots, b_k \in [0,1] \text{ such that } x = a + \sum_{i=1}^k b_i e_{j_i}  \bigr\}
\end{equation*}
is a \emph{non-oriented \( k \)-cell}. Note that if \( \sigma \) is a permutation, then \( (a; e_{j_1}, \dots, e_{j_k}) \) and \( (a; \sigma(e_{j_1}, \dots, e_{j_k})) \) represent the same non-oriented \( k \)-cell.

\subsubsection{Oriented cells}\label{sec: oriented cells}
To each non-oriented $k$-cell \( (a; e_{j_1}, \dots, e_{j_k}) \) with \( a \in \mathbb{Z}^m \), \( k \geq 1 \), and \( 1\leq j_1 < \dots < j_k\leq m \), we associate two \emph{oriented \( k \)-cells}, denoted \(  \frac{\partial}{\partial x^{j_1}}\big|_a \wedge \dots \wedge \frac{\partial}{\partial x^{j_k}}\big|_a\) and \( -\frac{\partial}{\partial x^{j_1}}\big|_a \wedge \dots \wedge \frac{\partial}{\partial x^{j_k}}\big|_a \), with opposite orientation.  
When \( a \in \mathbb{Z}^m \), \( 1\leq j_1 < \dots < j_k\leq m \), and \( \sigma  \) is a permutation of \( \{ 1,2, \dots, k \} \), we define
\begin{equation*}
    \frac{\partial}{\partial x^{j_{\sigma(1)}}}\bigg|_a \wedge \dots \wedge \frac{\partial}{\partial x^{j_{\sigma(k)}}}\bigg|_a 
    \coloneqq 
    \sgn(\sigma) \, 
    \frac{\partial}{\partial x^{j_1}}\bigg|_a \wedge \dots \wedge \frac{\partial}{\partial x^{j_k}}\bigg|_a
\end{equation*}
If \( \sgn(\sigma)=1 \), then  \( \frac{\partial}{\partial x^{j_{\sigma(1)}}}\big|_a \wedge \dots \wedge \frac{\partial}{\partial x^{j_{\sigma(k)}}}\big|_a \) is said to be \emph{positively oriented}, and if \( \sgn(\sigma)=-1 \), then \( \frac{\partial}{\partial x^{j_{\sigma(1)}}}\big|_a \wedge \dots \wedge \frac{\partial}{\partial x^{j_{\sigma(k)}}}\big|_a  \) is said to be \emph{negatively oriented}. 
Analogously, we define \begin{equation*}
    -\frac{\partial}{\partial x^{j_{\sigma(1)}}}\bigg|_a \wedge \dots \wedge \frac{\partial}{\partial x^{j_{\sigma(k)}}}\bigg|_a 
    \coloneqq 
    -\sgn(\sigma) \, 
    \frac{\partial}{\partial x^{j_1}}\bigg|_a \wedge \dots \wedge \frac{\partial}{\partial x^{j_k}}\bigg|_a,
\end{equation*}
and say that \( -\frac{\partial}{\partial x^{j_{\sigma(1)}}}\bigg|_a \wedge \dots \wedge \frac{\partial}{\partial x^{j_{\sigma(k)}}}\bigg|_a  \) is positively oriented if \( -\sgn(\sigma) = 1 \), and negatively oriented if \( -\sgn(\sigma) = -1. \)

Let $\mathcal{L} = \mathbb{Z}^m$ or $\mathcal{L} = B_N \subseteq \mathbb{Z}^m$. An oriented cell \( \frac{\partial}{\partial x^{j_1}}\big|_a \wedge \dots \wedge \frac{\partial}{\partial x^{j_k}}\big|_a \) is said to be in \( \mathcal{L} \) if all corners of \( (a;e_{j_1},\dots, e_{j_k})\) belong to \( \mathcal{L} \); otherwise it is said to be {\it outside} $\mathcal{L}$. 
The set of all oriented \( k \)-cells in \( \mathcal{L} \) will be denoted by \( C_k(\mathcal{L}). \) The set of all positively and negatively oriented cells in \( C_k(\mathcal{L}) \) will be denoted by \( C_k^+(\mathcal{L})\) and \( C_k^-(\mathcal{L})\), respectively.
A set \( C \subseteq C_k(\mathcal{L}) \) is said to be \emph{symmetric} if for each \( c \in C \) we have \( -c \in C \).

A non-oriented 0-cell \( a\in \mathbb{Z}^m \) is simply a point, and to each point we associate two oriented \(0\)-cells \( a^+ \) and \( a^-  \) with opposite orientation. We let \( C_0(\mathcal{L}) \) denote the set of all oriented \( 0 \)-cells.

Oriented 1-cells will be referred to as~\emph{edges}, and oriented 2-cells will be referred to as~\emph{plaquettes}.

\subsubsection{\( k \)-chains}\label{sec: chains}

The space of finite formal sums of positively oriented \( k \)-cells with integer coefficients will be denoted by \( C_k(\mathcal{L},\mathbb{Z}) \). 
Elements of \( C_k(\mathcal{L},\mathbb{Z}) \) will be referred to as \emph{\( k \)-chains}. 
If \( q \in C_k(\mathcal{L},\mathbb{Z}) \) and \( c \in C^+_k(\mathcal{L}) \), we let \( q[c] \) denote the coefficient of \( c \) in \( q \).
If \( c \in C^-_k(\mathcal{L}) \), we let \( q[c]\coloneqq -q[-c]. \)
For \(q,q' \in  C_k(\mathcal{L},\mathbb{Z}) \), we define
\begin{equation*}
    q+q' \coloneqq \sum_{c \in C_k^+(\mathcal{L})} \bigl(q[c] + q'[c] \bigr) c.
\end{equation*}
Using this operation, \( C_k(\mathcal{L},\mathbb{Z}) \) becomes a group.

When \( q \in C_k(\mathcal{L},G) \), we let the \emph{support} of \( q \) be defined by
\begin{equation*}
    \support q \coloneqq \bigl\{ c \in C_k^+(\mathcal{L}) \colon q[c] \neq 0 \bigr\}.
\end{equation*}

To simplify notation, when \( q \in C_k(\mathcal{L},G) \) and \( c\in C_k(\mathcal{L}) \), we write \( c \in q \) if either
\begin{enumerate}
    \item \( c \in C_k^+(\mathcal{L}) \) and \( q[c]>0\), or
    \item \( c \in C_k^-(\mathcal{L}) \) and \( q[-c]<0. \)
\end{enumerate}

\subsubsection{The boundary of a cell}

When \( k \geq 2 \), we define the \emph{boundary} \(\partial c \in C_{k-1}(\mathcal{L}, \mathbb{Z})\) of  \( c = \frac{\partial}{\partial x^{j_1}}\big|_a \wedge \dots \wedge \frac{\partial}{\partial x^{j_k}}\big|_a \in C_k(\mathcal{L})\) by
\begin{equation}\label{eq: boundary chain}
    \begin{split}
        \partial c \coloneqq \sum_{k' \in \{ 1,\dots, k \}}  \biggl(&
        (-1)^{k'}  \frac{\partial}{\partial x^{j_1}}\bigg|_a \wedge \dots \wedge \frac{\partial}{\partial x^{j_{k'-1}}}\bigg|_a \wedge  \frac{\partial}{\partial x^{j_{k'+1}}}\bigg|_a \wedge \dots \wedge \frac{\partial}{\partial x^{j_k}}\bigg|_a
        \\
        & + (-1)^{k'+1}   
        \frac{\partial}{\partial x^{j_1}}\bigg|_{a + e_{j_{k'}}} \wedge \dots \wedge \frac{\partial}{\partial x^{j_{k'-1}}}\bigg|_{a + e_{j_{k'}}} \wedge  \frac{\partial}{\partial x^{j_{k'+1}}}\bigg|_{a + e_{j_{k'}}} \wedge \dots \wedge \frac{\partial}{\partial x^{j_k}}\bigg|_{a + e_{j_{k'}}} 
        \biggr).
    \end{split}
\end{equation}
When \( c \coloneqq \frac{\partial}{\partial x^{j_1}}\big|_a \in C_1(\mathcal{L})\) we define the boundary \( \partial c \in C_0(\mathcal{L},\mathbb{Z}) \) by
\begin{equation*}
    \partial c = (-1)^1 a^+ + (-1)^{1+1} 
    (a+e_{j_1})^+ = (a+e_{j_1})^+ - a^+.
\end{equation*}

We extend the definition of \( \partial \) to \( k \)-chains \( q \in C_k(\mathcal{L},\mathbb{Z}) \) by linearity.
One verifies, as an immediate consequence of this definition, that if \( k \in \{ 2,3, \dots, m \} \), then \( \partial \partial c = 0 \) for any \( c \in \Omega_k(\mathcal{L}). \)

For an illustration of the boundary of an oriented 2-cell, see Figure~\ref{fig: boundary}.

\begin{figure*}[htp]
    \centering
	\begin{subfigure}[t]{0.24\textwidth}\centering
		\begin{tikzpicture}[scale=1]  
            	\fill[fill=detailcolor07, fill opacity=0.24] (0.02,0.02) -- (0.98,0.02) -- (0.98,0.98) -- (0.02,0.98) -- (0.02,0.02); 
            	
            	\draw (0.5,0.5) node[align=center] {\large $\circlearrowleft$};
            	
            	\draw[white] (0.5,-0.3) circle (2pt);
		\end{tikzpicture}
		\caption{An oriented 2-cell \( p.\)}
	\end{subfigure}
	\hfil
	\begin{subfigure}[t]{0.24\textwidth}\centering
		\begin{tikzpicture}[scale=1]  
            \fill[fill=detailcolor07, fill opacity=0.24] (0.02,0.02) -- (0.98,0.02) -- (0.98,0.98) -- (0.02,0.98) -- (0.02,0.02); 
            \draw (0.5,0.5) node[align=center] {\large $\circlearrowleft$};
            \draw (0.85,0.2) node[align=center] {\footnotesize $1$};
    
            \draw[white] (0.5,-0.3) circle (2pt);
		\end{tikzpicture}
		\caption{The 2-chain \( 1 \cdot p.\)}
	\end{subfigure}
	\hfil
	\begin{subfigure}[t]{0.24\textwidth}\centering
		\begin{tikzpicture}[scale=1]  
            
            \draw[detailcolor08!60!black,semithick] (0.02,0) -- (0.98,0);
            
            \draw[detailcolor08!60!black,semithick] (1,0.02)   -- (1,0.98);
            
            \draw[detailcolor08!60!black,semithick] (0.98,1) -- (0.02,1);
            
            \draw[detailcolor08!60!black,semithick] (0,0.98) -- (0,0.02);
            	
            \begin{scope}[decoration={markings, mark=at position 0.92 with {\arrow{>}}}] 
    		    \draw[semithick,draw=detailcolor01!10!black, postaction={decorate}] (0.55,0)  node[anchor=north] {\scriptsize $1$} -- (0.6,0);
    		    \draw[semithick,draw=detailcolor01!10!black, postaction={decorate}] (0.55,1)  node[anchor=south] {\scriptsize $\mathllap{-}1$} -- (0.6,1);
    		    \draw[semithick,draw=detailcolor01!10!black, postaction={decorate}] (1,0.55)  node[anchor=west] {\scriptsize $1$} -- (1,0.6);
    		    \draw[semithick,draw=detailcolor01!10!black, postaction={decorate}] (0,0.55) node[anchor=east] {\scriptsize $\mathllap{-}1$} -- (0,0.6);
    		 \end{scope}
            	
                \draw[white] (0.5,-0.3) circle (2pt);
    		 
		\end{tikzpicture}
		\caption{The 1-chain \( \partial p. \)}
	\end{subfigure}
	\hfil
	\begin{subfigure}[t]{0.24\textwidth}\centering
		\begin{tikzpicture}[scale=1]  
            
            \draw[detailcolor08!60!black,semithick] (0.02,0) -- (0.98,0);
            
            \draw[detailcolor08!60!black,semithick] (1,0.02)   -- (1,0.98);
            
            \draw[detailcolor08!60!black,semithick] (0.98,1) -- (0.02,1);
            
            \draw[detailcolor08!60!black,semithick] (0,0.98) -- (0,0.02);
            	
            \begin{scope}[decoration={markings, mark=at position 0.92 with {\arrow{>}}}] 
    		    \draw[semithick,draw=detailcolor01!10!black, postaction={decorate}] (0.55,0) -- (0.6,0);
    		    \draw[semithick,draw=detailcolor01!10!black, postaction={decorate}] (0.45,1) -- (0.4,1);
    		    \draw[semithick,draw=detailcolor01!10!black, postaction={decorate}] (1,0.55) -- (1,0.6);
    		    \draw[semithick,draw=detailcolor01!10!black, postaction={decorate}] (0,0.45)  -- (0,0.4);
    		 \end{scope}
            	
                \draw[white] (0.5,-0.3) circle (2pt);
    		 
		\end{tikzpicture}
		\caption{The edges \( e\in\partial p. \)}
	\end{subfigure}
	\caption{An illustration of the boundary of an oriented 2-cell.}\label{fig: boundary}
\end{figure*}
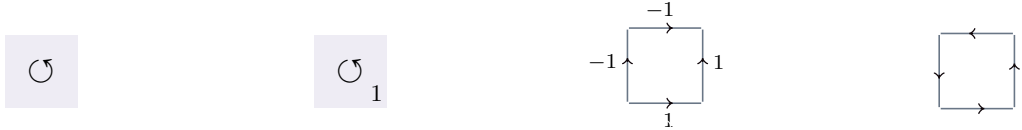

\subsubsection{The coboundary of an oriented cell}\label{sec: coboundary}

If \( k \in \{ 0,1, \ldots, n-1 \} \) and \( c \in C_k(\mathcal{L})\) is an oriented \( k \)-cell, we define the \emph{coboundary} \( \hat \partial c \in C_{k+1}(\mathcal{L})\) of \( c \) as the \( (k+1) \)-chain 
\begin{equation*}
	\hat \partial c \coloneqq \sum_{c' \in C_{k+1}(\mathcal{L})} \bigl(\partial c'[c] \bigr) c'.
\end{equation*}
Note in particular that if \( c' \in C_{k+1}(\mathcal{L}),\) then
\begin{equation*}
    \hat \partial c[c'] = \partial c'[c].
\end{equation*}

For an illustration of the coboundary of an oriented 1-cell, see Figure~\ref{fig: coboundary}.

\begin{figure*}[htp]
    \centering
	\begin{subfigure}[t]{0.24\textwidth}\centering
		\begin{tikzpicture}[scale=1]  
            	\draw[white] (0,-1) circle (1pt);
				\draw (0,0) -- (1,0);
            		\begin{scope}[decoration={markings, mark=at position 0.92 with {\arrow{>}}}] 
    		    \draw[semithick,draw=detailcolor01!10!black, postaction={decorate}] (0.55,0) -- (0.6,0);
    		 \end{scope} 
		\end{tikzpicture}
		\caption{An oriented 1-cell \( e. \)}
	\end{subfigure}
	\hfil
	\begin{subfigure}[t]{0.24\textwidth}\centering
		\begin{tikzpicture}[scale=1]  
            	\draw[white] (0,-1) circle (1pt);
				\draw (0,0) -- (1,0);
            		\begin{scope}[decoration={markings, mark=at position 0.92 with {\arrow{>}}}] 
    		    \draw[semithick,draw=detailcolor01!10!black, postaction={decorate}] (0.55,0)  node[anchor=north] {\scriptsize $1$} -- (0.6,0);
    		 \end{scope} 
		\end{tikzpicture}
		\caption{The 1-chain \( 1\cdot e .\)}
	\end{subfigure}
	\hfil
	\begin{subfigure}[t]{0.24\textwidth}\centering
		\begin{tikzpicture}[scale=1]   
            	\draw[white] (0,-1) circle (1pt);
		
            	\fill[fill=detailcolor07, fill opacity=0.24] (0.02,0.02) -- (0.98,0.02) -- (0.98,0.98) -- (0.02,0.98) -- (0.02,0.02);  
            	\draw (0.5,0.5) node[align=center] {\large $\circlearrowleft$};
            	\draw (0.85,0.2) node[align=center] {\footnotesize $1$};

            	\fill[fill=detailcolor07, fill opacity=0.24] (0.02,-0.98) -- (0.98,-0.98) -- (0.98,-0.02) -- (0.02,-0.02) -- (0.02,-0.98);  
            	\draw (0.5,-0.5) node[align=center] {\large $\circlearrowleft$};
            	\draw (0.85,-0.8) node[align=center] {\footnotesize $\mathllap{-}1$};
		\end{tikzpicture}
		\caption{The 2-chain  \( \hat \partial e. \)}
	\end{subfigure}\hfil
	\begin{subfigure}[t]{0.24\textwidth}\centering
		\begin{tikzpicture}[scale=1]   
            	\draw[white] (0,-1) circle (1pt);
		
            	\fill[fill=detailcolor07, fill opacity=0.24] (0.02,0.02) -- (0.98,0.02) -- (0.98,0.98) -- (0.02,0.98) -- (0.02,0.02);  
            	\draw (0.5,0.5) node[align=center] {\large $\circlearrowleft$};
            	
            	\fill[fill=detailcolor07, fill opacity=0.24] (0.02,-0.98) -- (0.98,-0.98) -- (0.98,-0.02) -- (0.02,-0.02) -- (0.02,-0.98);  
            	\draw (0.5,-0.5) node[align=center] {\large $\circlearrowright$};
		\end{tikzpicture}
		\caption{The plaquettes {\( {p \in \hat \partial e.} \)}}
	\end{subfigure}
	\caption{An illustration of the coboundary of an oriented 1-cell.}\label{fig: coboundary}
\end{figure*}
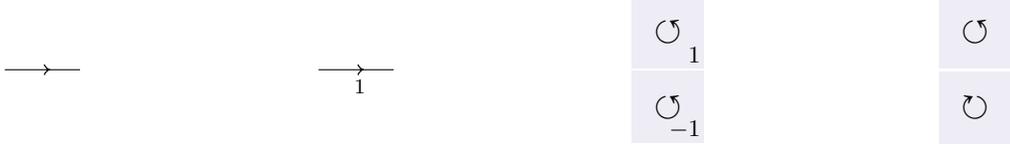

We extend the definition of \( \hat{\partial} \) to \( k \)-chains \( q \in C_k(\mathcal{L},\mathbb{Z}) \) by linearity.

\subsubsection{The boundary of a box}

An oriented \( k \)-cell \( c = \frac{\partial}{\partial x^{j_1}}\big|_a \wedge \dots \wedge \frac{\partial}{\partial x^{j_k}}\big|_a \in C_k(B_N)\) is said to be a \emph{boundary cell} of a box \( B = \bigl( [a_1,b_1]\times \dots \times [a_m,b_m] \bigr) \cap \mathbb{Z}^m \subseteq B_N\), or equivalently to be in \emph{the boundary} of \( B \), if the non-oriented cell \( (a;e_{j_1}, \dots, e_{j_k}) \) intersects the boundary of 
\(  [a_1,b_1]\times \dots \times [a_m,b_m]. \)
Note that if \( k \in \{ 0,1, \dots, m-1 \} \) and \( c \) is in the boundary of \( B \), then there is a \( (k+1) \)-cell \(  c' \in C_{k+1}(\mathbb{Z}^m)\smallsetminus C_{k+1}(B) \) such that \( \hat \partial c[c'] \neq 0 \).

\subsection{The dual cell complex}

The lattice \( \mathbb{Z}^m \) has a natural dual, called the \( \emph{dual lattice} \) and denoted by \( (\mathbb{Z}^m)^* \). In this context, the lattice \( \mathbb{Z}^m \) is called the \emph{primal lattice.} The vertices of the dual lattice \( (\mathbb{Z}^m)^* \) are placed at the centers of the non-oriented \( m \)-cells of the primal lattice. For the dual lattice we replace the unit vectors \( e_1, e_2, \dots, e_m \) in the definition of cells, with the vectors \( \hat e_1 \coloneqq -e_1,\dots, \hat e_m \coloneqq -e_m \) (see Figure~\ref{figure: dual lattice}). 
For $b \in (\mathbb{Z}^m)^*$ and $j = 1, \dots, m$, we let \( \frac{\partial}{\partial y^j} \big|_b \) denote the oriented \( 1 \)-cell in the dual lattice that starts at \( b \) and ends at $b + \hat{e}_j = b - e_j$.

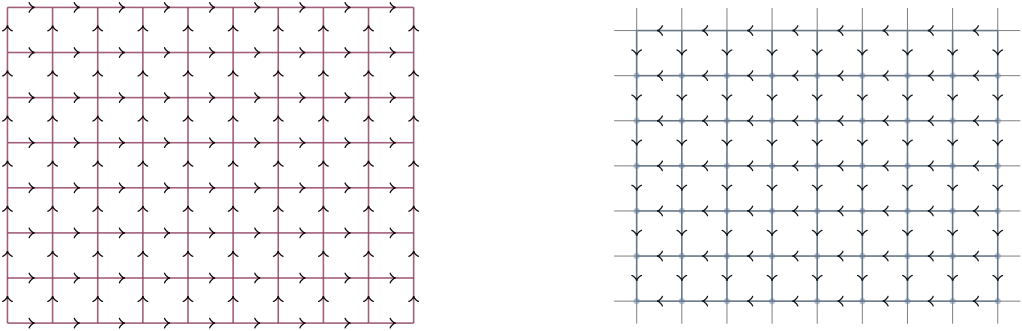
\begin{figure}[htp]
    \centering
        \begin{subfigure}[t]{0.45\textwidth}\centering
            \begin{tikzpicture}[scale=0.6]  
        
        \begin{scope}[decoration={markings, mark=at position 0.92 with {\arrow{>}}}] 
    		    \foreach \x in {-1,...,7} 
    		        \foreach \y in {-2,...,5}  { 
    		        
    		            \draw[detailcolor01,semithick, opacity=0.8] (\x,\y) -- (\x+1,\y); \draw[semithick,draw=detailcolor01!10!black, postaction={decorate}] (\x+0.55,\y) -- (\x+0.6,\y);
    		      };
    		  
    		  \foreach \x in {-1,...,8} 
    		        \foreach \y in {-2,...,4}  { 
    		        
    		            \draw[detailcolor01,semithick, opacity=0.8] (\x,\y) -- (\x,\y+1); \draw[semithick,draw=detailcolor01!10!black, postaction={decorate}] (\x,\y+0.55) -- (\x,\y+0.6);
    		      };
    		    
		\end{scope}

        \end{tikzpicture}
        \caption{The positively oriented 1-cells in the primal lattice.}
        \end{subfigure}
        \hfil
        \begin{subfigure}[t]{0.45\textwidth}\centering
            \begin{tikzpicture}[scale=0.6]  
        
                \draw[white] (0,0) circle (1pt); 
                    \foreach \x in {-1,...,7} 
    		            \foreach \y in {-2,...,3}  { 
            		        \fill[detailcolor08] (\x,\y) circle (2pt);
    		        };
    		      
    		      \foreach \x in {-1,...,7} {
    		        \draw[help lines] (\x,-2.5) -- (\x,-2);
    		        \draw[help lines] (\x,4) -- (\x,4.5);
    		        };
    		      
    		      \foreach \y in {-2,...,4} {
    		        \draw[help lines] (-1.5,\y) -- (-1,\y);
    		        \draw[help lines] (7,\y) -- (7.5,\y);
    		        };

                \begin{scope}[decoration={markings, mark=at position 0.92 with {\arrow{>}}}] 
    		    \foreach \x in {-1,...,6} 
    		        \foreach \y in {-2,...,4}  { 
    		            \draw[detailcolor08!60!black,semithick] (\x,\y) -- (\x+1,\y); \draw[semithick,draw=detailcolor08!10!black, postaction={decorate}] (\x+0.5,\y) -- (\x+0.45,\y);
    		      };
    		  
    		  \foreach \x in {-1,...,7} 
    		        \foreach \y in {-2,...,3}  { 
    		            \draw[detailcolor08!60!black,semithick] (\x,\y) -- (\x,\y+1); \draw[semithick,draw=detailcolor08!10!black, postaction={decorate}] (\x,\y+0.5) -- (\x,\y+0.45);
    		      };
    		    
		\end{scope}

        \end{tikzpicture}
        
        \caption{The positively oriented 1-cells in the dual lattice.}
    \end{subfigure}
    \caption{In the figures above, we draw the positively oriented 1-cells in the primal and dual lattices respectively.}
    \label{figure: dual lattice}
\end{figure}

More generally, for \( k \in \{ 0,1,\ldots, m \} \), there is a bijection between the set of oriented \( k \)-cells of \( \mathbb{Z}^m \) and the set of oriented \( (m-k) \)-cells of \((\mathbb{Z}^m)^* \) defined as follows. 
For each \( a \in \mathbb{Z}^m \), let \(   b \coloneqq *(a;e_1,\dots,e_m)\in (\mathbb{Z}^m)^*\) be the point at the center of the primal lattice non-oriented \( m \)-cell \( (a;e_1,\dots,e_m) \).
With this in mind, we define
\begin{equation*}
    *(a^+) \coloneqq \frac{\partial}{\partial y^{1}} \bigg|_b \wedge \dots \wedge  \frac{\partial}{\partial y^{m} }\bigg|_b, \qquad
    * \bigg(\frac{\partial}{\partial x^{1}} \bigg|_a \wedge \dots \wedge  \frac{\partial}{\partial x^{m} }\bigg|_a \bigg) \coloneqq b^+,
\end{equation*}
and
\begin{equation*}
    * \bigg(\frac{\partial}{\partial y^{1}} \bigg|_b \wedge \dots \wedge  \frac{\partial}{\partial y^{m} }\bigg|_b \bigg) \coloneqq a^+,
    \qquad 
    *(b^+) \coloneqq  \frac{\partial}{\partial x^{1}} \bigg|_a \wedge \dots \wedge  \frac{\partial}{\partial x^{m} }\bigg|_a.
\end{equation*}
Next,  let \( k \in \{ 1, \ldots, m-1 \} \) and assume that \( 1 \leq j_1 < \cdots < j_k \leq m \) and  \( a \in \mathbb{Z}^m\) are given. Then \( 
\frac{\partial}{\partial x^{j_1}} \big|_a \wedge \dots \wedge \frac{\partial}{\partial x^{j_k}} \big|_a\) is a positively oriented \( k \)-cell in \( \mathbb{Z}^m \). 
Let \( {i_1}, \ldots, i_{m-k} \) be any enumeration of \( \{ 1,2, \ldots, m \} \smallsetminus \{ j_1, \ldots, j_k \} \), and let \(  \sgn (j_1,\ldots, j_k, i_{1}, \ldots, i_{m-k} )\) denote the sign of the permutation that maps $(1,2,\ldots, m)$ to \( (j_1,\ldots, j_k, i_{1}, \ldots, i_{m-k} ) \). 
Define 
\begin{equation*}
    *\bigg({\frac{\partial}{\partial x^{j_1}} \bigg|_a \wedge \dots \wedge \frac{\partial}{\partial x^{j_k}}} \bigg|_a\bigg)
    \coloneqq 
    \sgn (j_1,\ldots, j_k, i_{1}, \ldots, i_{m-k} )\, 
    \frac{\partial}{\partial y^{i_1}}\bigg|_b \wedge \dots \wedge  \frac{\partial}{\partial y^{i_{m-k}} } \bigg|_b.
\end{equation*}
Analogously, we define
\begin{equation*}
    \begin{split}
         &{*}  \bigg(\frac{\partial}{\partial y^{i_1}} \bigg|_b \wedge \dots \wedge  \frac{\partial}{\partial y^{i_{m-k}} }\bigg|_b \bigg) 
         \coloneqq
        \sgn (i_{1}, \ldots, i_{m-k},j_1,\ldots, j_k)\, \frac{\partial}{\partial x^{j_1}} \bigg|_a \wedge 
        \dots \wedge \frac{\partial}{\partial x^{j_k}} \bigg|_a.
    \end{split}
\end{equation*}
We extend these definitions from \( C_k^+(B_N) \) to \( C_k(B_N) \) and \( C_k(B_N,\mathbb{Z}) \) by linearity.

If \( B \) is a box in \( \mathbb{Z}^m \), then we define
\begin{equation*}
    B^* \coloneqq \bigl\{ y \in (\mathbb{Z}^m)^* \colon \exists x \in C_0(B) \text{ such that } y \text{ is a corner in } {*x} \bigr\}.
\end{equation*}
Note that with this definition, \( B \subsetneq (B^*)^* \) (see also Figure~\ref{figure: dual box}).

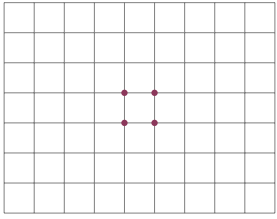
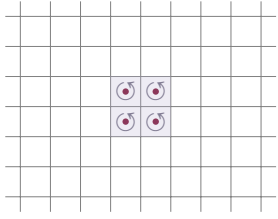
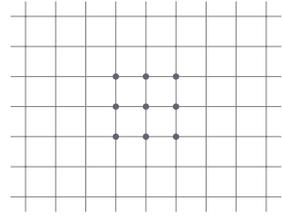
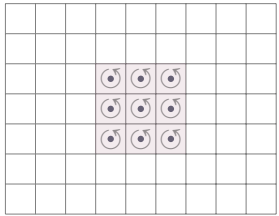
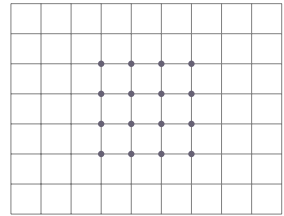
\begin{figure}[htp]
    \centering
    \begin{subfigure}[t]{0.3\textwidth}\centering 
        \begin{tikzpicture}[scale=0.4] 
        \draw[help lines] (0,0) grid (9,7); 
            
            \fill[fill=detailcolor01] (4,3) circle (3pt);
            \fill[fill=detailcolor01] (5,3) circle (3pt);
            \fill[fill=detailcolor01] (5,4) circle (3pt);
            \fill[fill=detailcolor01] (4,4) circle (3pt);
            
        \end{tikzpicture}
        \caption{A box \( B \) containing four points in \( \mathbb{Z}^2 \) (purple dots).}
    \end{subfigure}
        \hfil 
    \begin{subfigure}[t]{0.3\textwidth}\centering 
        \begin{tikzpicture}[scale=0.4] 
        
            \draw[white] (0,0) circle (1pt);
            \begin{scope}[shift={(0.5,0.5)}]
                \draw[help lines] (-0.5,-0.5) grid (8.5,6.5); 
            \end{scope} 
        
            \fill[fill=detailcolor07, fill opacity=0.24] (4-0.5,3-0.5) -- (5+0.5,3-0.5) -- (5+0.5,4+0.5) -- (4-0.5,4+0.5) -- (4-0.5,3-0.5);  
            
            \fill[fill=detailcolor01] (4,3) circle (3pt);
            \fill[fill=detailcolor01] (5,3) circle (3pt);
            \fill[fill=detailcolor01] (5,4) circle (3pt);
            \fill[fill=detailcolor01] (4,4) circle (3pt);
            
            \draw[detailcolor07!70!black] (4,3) node {\footnotesize $\circlearrowleft$};
            \draw[detailcolor07!70!black] (5,3) node {\footnotesize $\circlearrowleft$};
            \draw[detailcolor07!70!black] (5,4) node {\footnotesize $\circlearrowleft$};
            \draw[detailcolor07!70!black] (4,4) node {\footnotesize $\circlearrowleft$};
            
        \end{tikzpicture}
        \caption{The set \( B \) (purple), together with the positively oriented 2-cells in the dual lattice which are equal to \( *a \) for some \( a \in C_0^+(B) \) (blue).}
    \end{subfigure}
    \hfil
    \begin{subfigure}[t]{0.3\textwidth}\centering 
        \begin{tikzpicture}[scale=0.4] 
        
            \draw[white] (0,0) circle (1pt);
            \begin{scope}[shift={(0.5,0.5)}]
                \draw[help lines] (-0.5,-0.5) grid (8.5,6.5); 
            \end{scope}
            
            \foreach \x in {4,5,6} 
    		        \foreach \y in {3,4,5}  { 
    		        	\fill[fill=detailcolor07!55!black] (\x-0.5,\y-0.5) circle (3pt);
    		      }; 
    		      
        \end{tikzpicture}
        \caption{The set \( B^* \) (blue),  consisting of nine points in \( (\mathbb{Z}^2)^* .\)}
    \end{subfigure}
    
    \vspace{2ex}
    \begin{subfigure}[t]{0.3\textwidth}\centering 
        \begin{tikzpicture}[scale=0.4] 
        
            \draw[help lines] (0,0) grid (9,7);  
            
    		\fill[fill=detailcolor01, fill opacity=0.1] (3,2) -- (6,2) -- (6,5) -- (3,5) -- (3,2); 
            
            \foreach \x in {4,5,6} 
    		        \foreach \y in {3,4,5}  { 
    		        	\fill[fill=detailcolor07!55!black] (\x-0.5,\y-0.5) circle (3pt);
    		        	
    		        	\draw[fill=detailcolor01,opacity=0.4] (\x-0.5,\y-0.5) node {\small $\circlearrowleft$};
    		      };

        \end{tikzpicture}
        \caption{The set \( B^* \) (blue) together with the positively oriented 2-cells in the primal lattice which are equal to \( *b \) for some \( b \in C_0^+(B^*) \) (purple).}
    \end{subfigure}
    \hfil
    \begin{subfigure}[t]{0.3\textwidth}\centering 
        \begin{tikzpicture}[scale=0.4] 
        
            \draw[help lines] (0,0) grid (9,7);  
             
            \foreach \x in {3,4,5,6} 
    		        \foreach \y in {2,3,4,5}  { 
    		        	\fill[fill=detailcolor07!55!black] (\x,\y) circle (3pt);
    		      };

        \end{tikzpicture}
        \caption{The set \( (B^*)^* \) in the primal lattice, consisting of 16 points in \( \mathbb{Z}^2.\)}
    \end{subfigure}
    
    \caption{In the figures above, we illustrate the sets \( B^* \) and \( (B^*)^* \) for a box \( B \) consisting of four points in the primal lattice.}
        \label{figure: dual box}
\end{figure}

We end this subsection with the following lemma.
\begin{lemma}[Lemma 2.4 in~\cite{c2019}]\label{lemma: lemma 2.4}
    Let \( B \)  be any box in \( \mathbb{Z}^m \). Then an oriented \( k \)-cell \(c\) is outside \(B\) if and only if \(*c\) is either outside \(B^*\) or in the boundary of \(B^* \). Moreover, if \( c\) is an oriented \(k\)-cell outside \(B\) that contains a \( (k-1)\)-cell of \(B\), then \( *c \) belongs to the boundary of \(B^* \).
\end{lemma}

\subsection{Discrete exterior calculus}
In what follows, we give a brief overview of discrete exterior calculus on the cell complexes of \( \mathbb{Z}^m \) and \( B = [a_1,b_1] \times \dots \times [a_m,b_m]  \cap  \mathbb{Z}^m \) for \( m \geq 1 \).  

All of the results in this subsection are obtained under the assumption that an abelian group \( G \), which is not necessarily finite, has been given. In particular, they all hold for both \( G=\mathbb{Z}_n \) and \( G=\mathbb{Z} \).

\subsubsection{Discrete differential forms}\label{sec: ddf}

A homomorphism from the group \( C_k(\mathcal{L},\mathbb{Z}) \) to the group \( G \) is called a \emph{\( k \)-form}. The set of all such \( k\)-forms will be denoted by \( \Omega^k(\mathcal{L},G) \). This set becomes an abelian group if we add two homomorphisms by adding their values in \( G \).

The set $C_k^+(\mathcal{L})$ of positively oriented $k$-cells is naturally embedded in  $C_k(\mathcal{L},\mathbb{Z})$ via the map  $c \mapsto 1 \cdot c$, and we will  frequently identify $c \in C_k^+(\mathcal{L})$ with the $k$-chain $1 \cdot c$ using this embedding. Similarly, we will identify a negatively oriented $k$-cell $c \in C_k^-(\mathcal{L})$ with the $k$-chain $(-1) \cdot (-c)$. 
In this way, a $k$-form $\omega$ can be viewed as a \( G \)-valued function on \( C_k(\mathcal{L}) \) with the property that \( \omega(c) = -\omega(-c) \) for all \( c \in C_k(\mathcal{L}) \). Indeed, if \( \omega \in \Omega^k(\mathcal{L},G) \) and \( q = \sum a_i c_i \in C_k(\mathcal{L},\mathbb{Z}) \), we have
\begin{equation*}
    \omega(q) = \omega \bigl(\sum a_i c_i \bigr) = \sum a_i \omega(c_i),
\end{equation*}
and hence a \( k \)-form is uniquely determined by its values on positively oriented \( k \)-cells.  

If \( \omega \) is a \( k \)-form, it is useful to represent it by the formal expression 
\begin{equation*}
    \sum_{1 \leq j_1 < \dots < j_k \leq m} \omega_{j_1\dots j_k} dx^{j_1} \wedge \cdots \wedge dx^{j_k}.
\end{equation*}
where \( \omega_{j_1\dots j_k} \) is a \( G \)-valued function on the set of all \( a \in \mathbb{Z}^m \) such that \( \frac{\partial}{\partial x^{j_1}} \big|_a \wedge \dots \wedge \frac{\partial}{\partial x^{j_k}}\big|_a \in C_k(\mathcal{L})\), defined by
\begin{equation*}
    \omega_{j_1 \dots j_k}(a) = \omega \biggl( \frac{\partial}{\partial x^{j_1}} \bigg|_a \wedge \dots \wedge \frac{\partial}{\partial x^{j_k}} \bigg|_a \biggr).
\end{equation*} 

If \( 1\leq j_1 < \dots < j_k\leq m \) and \( \sigma  \) is a permutation of \( \{ 1,2, \dots, k \} \), we define
\begin{equation*}
    dx^{j_{\sigma(1)}}  \wedge \dots \wedge d x^{j_{\sigma(k)}}
    \coloneqq 
    \sgn(\sigma) \, 
    d  x^{j_1} \wedge \dots \wedge d x^{j_k},
\end{equation*} 
and if \( 1 \leq j_1,\dots, j_k \leq n \) are such that \( j_i = j_{i'} \) for some \( 1 \leq i < i' \leq k \), then we let
\begin{equation*}
    d  x^{j_1} \wedge \dots \wedge d x^{j_k} \coloneqq 0.
\end{equation*}

Given a \( k \)-form \( \omega \), we let \( \support \omega \) denote the support of \( \omega \), i.e., the set of all oriented \( k \)-cells \( c \) such that \( \omega(c) \neq 0 \). Note that $\support \omega$ always contains an even number of elements.

\subsubsection{The exterior derivative}\label{sec: derivative}
Given \( h \colon \mathbb{Z}^m \to G \), \( a \in \mathbb{Z}^m \), and \( i \in \{1,2, \ldots, m \} \), we let 
\begin{equation*}
    \partial_i h(a) \coloneqq h(a+e_i) - h(a) .
\end{equation*}
If \( k \in \{ 0,1,2, \ldots, m \} \) and \( \omega \in \Omega^k(\mathcal{L},G) \), we define the \( (k+1) \)-form \( d\omega \in \Omega^{k+1}(\mathcal{L},G) \) by
\begin{equation*}
    d\omega = \sum_{1 \leq j_1 < \dots < j_k \leq m} \sum_{i=1}^m \partial_i \omega_{j_1,\dots,j_k} \,  dx^i \wedge (dx^{j_1} \wedge \dots \wedge dx^{j_k}).
\end{equation*} 
The operator \( d \) is called the \emph{exterior derivative.}

If \( a_0 \in \mathbb{Z}^m \), \( 1 \leq j_1 < \dots < j_k \leq m \), \( c = {\frac{\partial}{\partial x^{j_1}}\big|_{a_0} \wedge \dots \wedge \frac{\partial}{\partial x^{j_k}}}\big|_{a_0} \), and \( \omega = \omega_{j_1 \dots j_k} d x^{j_1} \wedge \dots \wedge d x^{j_k}, \) where $\omega_{j_1 \dots j_k}(a) = 1$ if $a = a_0$ and $\omega_{j_1 \dots j_k}(a) = 0$ if $a \neq a_0$, then
    \begin{equation*}
    	d\omega
    	= 
        \sum_{i = 1}^m \partial_i \omega_{j_1 \dots j_k} \, dx^i \wedge ( dx^{j_1} \wedge \cdots \wedge dx^{j_k}),
    \end{equation*}
    where
    \begin{equation*}
        (\partial_i \omega_{j_1 \dots j_k})(a) 
        = 
        \begin{cases} 
            1, & a = a_0 - e_i,  \\
            -1, & a = a_0,  \\
            0, & \text{otherwise}.
        \end{cases}
    \end{equation*}
    Using \eqref{eq: boundary chain}, it follows in this special case that whenever \( c' \in C_{k+1}(\mathcal{L}) \), we have \( d\omega(c') = 1 \) if \( c \in \support \partial c' \), \( d\omega(c')=-1 \) if \( -c' \in \support \partial c \), and \( d\omega(c') = 0 \) else.
    As a consequence, for any \( \omega\in \Omega^k(\mathcal{L},G) \) and \( c \in  C_k(\mathcal{L},\mathbb{Z}) \), we have
    \begin{equation}\label{eq: stokes}
	    d\omega(c) = \omega(\partial c).
    \end{equation}
    This equality is known as the \emph{discrete Stokes' theorem.}
    In particular, if \( \omega \in \Omega^2(\mathcal{L},G) \) and \( d\omega = 0 \), then for any \( c \in C_3(\mathcal{L},\mathbb{Z}) \) we have
    \begin{equation}\label{eq: Bianchi}
        \omega(\partial c) = d\omega(c) = 0.
    \end{equation}
    This conclusion is sometimes called \emph{Bianchi's lemma}, and will be useful to us later.

Before ending this section, we note that 
whenever \( k \in \{ 2,3,\dots, m-2\}, \) \( \omega \in \Omega^{k}(\mathcal{L},G) \) and \( c \in C_{k+2}(\mathcal{L}) \), we have \( \partial \partial c = 0 \), and hence by the discrete Stokes theorem, it follows that
\begin{equation*}
    dd\omega(c) = d\omega(\partial c) = \omega(\partial\partial c) = 0.
\end{equation*}
Consequently, we have \( dd\omega = 0 \) for any \( \omega \in \Omega^k(\mathcal{L},G). \)

\subsubsection{Closed forms and the Poincar\'e lemma}

For \( k \in \{ 0,\ldots, m \} \), we say that a \( k \)-form \( \omega \in \Omega^k(\mathcal{L},G) \) is \emph{closed} if \( d\omega(c) = 0 \) for all \( c \in C_{k+1}(\mathcal{L}). \)
The set of all closed forms in \( \Omega^k(\mathcal{L},G) \) will be denoted by \( \Omega^k_0(\mathcal{L},G). \)

\begin{lemma}[The Poincar\'e lemma, Lemma 2.2 in~\cite{c2019}]\label{lemma: poincare}
    Let \( k \in \{ 1, \ldots, m\} \) and let \( B \) be a box in \( \mathbb{Z}^m \). Then the exterior derivative \( d \) is a surjective map from the set \( \Omega^{k-1}(B \cap \mathbb{Z}^m, G) \) to \( \Omega^k_0(B \cap \mathbb{Z}^m, G) \).
    Moreover, if \(G\) is finite, then this map is an \( \bigl| \Omega^{k-1}_0(B \cap \mathbb{Z}^m, G)\bigr|\)-to-\(1\) correspondence.
    Lastly, if \( k \in \{ 1,2, \ldots, m-1 \} \) and \(\omega \in \Omega^k_0(B \cap \mathbb{Z}^m, G)\) vanishes on the boundary of \(B\), then there is a \((k-1)\)-form \( \omega' \in \Omega^{k-1}(B \cap \mathbb{Z}^m, G)\) that also vanishes on the boundary of \(B\) and satisfies \(d\omega' = \omega\). 
\end{lemma}

\subsubsection{The coderivative}\label{sec: coderivative}

Given \( h \colon \mathbb{Z}^m \to G \), \( a \in \mathbb{Z}^m \), and \( i \in \{ 1,2, \ldots, m \} \), we let 
\begin{equation*}
    \bar \partial_i h(a) \coloneqq h(a) - h(a-e_i).
\end{equation*}
When \( k \in \{ 1,2, \ldots, m \} \) and \( \omega \in \Omega^k(\mathcal{L},G) \), we define the \( (k-1) \)-form \( \delta \omega \in \Omega^{k-1}(\mathcal{L},G) \)  by
\begin{equation*}
    \delta \omega \coloneqq 
    \sum_{1 \leq j_1 < \dots < j_k \leq m}
    \sum_{i = 1}^k  
    (-1)^{i} \, \bar \partial_{j_i} \omega 
    \, dx^{j_1} \wedge \cdots \wedge dx^{j_{i-1}}\wedge dx^{j_{i+1}} \wedge \cdots \wedge  dx^{j_{k}}.
\end{equation*}
The operator \( \delta \) is called the \emph{coderivative}.

\begin{example}
    If \( a_0 \in \mathbb{Z}^m \), \( 1 \leq j_1 < \dots < j_k \leq m \), \( c = {\frac{\partial}{\partial x^{j_1}}\big|_{a_0} \wedge \dots \wedge \frac{\partial}{\partial x^{j_k}}}\big|_{a_0}  \), and \( \omega = \omega_{j_1 \dots j_k} d x^{j_1} \wedge \dots \wedge d x^{j_k}, \) where $\omega_{j_1 \dots j_k}(a) = 1$ if $a = a_0$ and $\omega_{j_1 \dots j_k}(a) = 0$ if $a \neq a_0$, then
    \begin{equation*}
        \delta \omega = 
        \sum_{i = 1}^k  
        (-1)^{i} \, \bar \partial_{j_i} \omega 
        \, dx^{j_1} \wedge \cdots \wedge dx^{j_{i-1}}\wedge dx^{j_{i+1}} \wedge \cdots \wedge  dx^{j_{k}},
    \end{equation*}
    where
    \begin{equation*}
        (\bar \partial_i \omega_{j_1 \dots j_k})(a) 
        = 
        \begin{cases} 
            1, & a = a_0,  \\
            -1, & a = a_0 + e_i,  \\
            0, & \text{otherwise}.
        \end{cases}
    \end{equation*}
    Recalling the definition of the coboundary of an oriented cell from Section~\ref{sec: coboundary}, we see that, in this special case, for \( c' \in C_{k-1}(\mathcal{L}) \), we have \( d\omega(c') = 1 \) if \( c \in \support \hat \partial c' \), \( d\omega(c')=-1 \) if \( -c' \in \support \hat \partial c \), and \( d\omega(c') = 0 \) otherwise.
    As a consequence, for any \( \omega\in \Omega^k(\mathcal{L},G) \) and \( c \in C_k (\mathcal{L},\mathbb{Z}) \), we have
    \begin{equation*}
	    \delta\omega(c) = \omega(\hat \partial c).
    \end{equation*}
\end{example}

\begin{lemma}[The Poincar\'e lemma for the coderivative, Lemma 2.7 in~\cite{c2019}]\label{lemma: lemma 2.7}
    Let \( G \) be an abelian group, and let \( k \in \{ 1,2, \ldots, m-1 \} \). Let \( \omega \in \Omega^k(\mathbb{Z}^m,G) \) be equal to zero outside a finite region and satisfy \( \delta f  = 0 \). Then there is \( \omega' \in \Omega^{k+1}(\mathbb{Z}^m,G) \) such that \( \omega = \delta \omega' \). Moreover, if \( \omega \) is equal to zero outside a box \( B \), then there is a choice of \( \omega' \) that is equal to zero outside \( B \). 
\end{lemma}

\subsubsection{The Hodge dual}
 
Given  \( \omega \in \Omega^k(\mathbb{Z}^m,G)\), we define the \emph{Hodge dual} \( *\omega \) of \( \omega \)  as

\begin{equation*}
    *\omega \coloneqq \sum_{1 \leq i_1 < \cdots < i_k \leq m} \omega_{j_1 \dots j_k}(*(\cdot; e_{i_1},\dots, e_{i_k})) \sgn (i_1,\ldots, i_k,j_{1}, \ldots, j_{m-k}) \, dy^{j_1} \wedge \cdots \wedge d^{y_{j_{m-k}}},
\end{equation*}
where, in each term, the sequence \( j_1, \ldots, j_{m-k} \) depends on the sequence \( i_1,\ldots, i_k \).

One verifies that with these definitions, for any \( \omega \in \Omega^k(\mathbb{Z}^m,G) \) and \( c \in C_k(\mathbb{Z}) \), we have \( *\omega(*c) = \omega(c) \), and 
\begin{equation}\label{hodgehodge}
    {*}{*\omega}(c) = (-1)^{k(m-k)}\omega(c).
\end{equation}
The exterior derivative on the dual cell lattice is defined by
\begin{equation*}
    d\omega \coloneqq \sum_{1 \leq j_1 < \cdots < j_k\leq m} \sum_{i=1}^m \bar \partial_i \omega_{j_1\dots j_k} \, dy_i \wedge (dy_{j_1} \wedge \cdots \wedge dy_{j_k}).
\end{equation*}
Note that the definitions of the exterior derivative on the primal and dual lattices are not identical; in the dual lattice the partial derivative \( \partial_i \) has been replaced by \( \bar \partial_i \). The reason for this difference is the opposite orientation of the edges in the dual lattice.

The next lemma describe a relationship between the Hodge dual, the exterior derivative on the dual lattice, and the coderivative on the primal lattice.
 
\begin{lemma}[Lemma 2.3 in~\cite{c2019}]\label{lemma: lemma 2.3}
    For any \( \omega \in \Omega^k(\mathbb{Z}^m,G) \), and any \( c \in C_{k-1}(\mathbb{Z})  \), 
    \begin{equation*}
        \delta \omega(c)= (-1)^{m(k+1)+1} {*}(d (*\omega(*c))).
    \end{equation*}
\end{lemma}

\subsubsection{Non-trivial forms}

We say that a \( k \)-form \( \omega \in \Omega^k(\mathcal{L},G) \) is \emph{non-trivial} if there is at least one \( k \)-cell \( c \in C_k(\mathcal{L}) \) such that \( \omega(c) \neq 0 \).

\subsubsection{Restrictions of forms}

If \( \omega \in \Omega^k(\mathcal{L},G) \), \( C \subseteq C_k(\mathcal{L}) \) is symmetric, and \( c \in C \), we define 
\begin{equation*}
    \omega|_C(c) \coloneqq 
    \begin{cases}
        \omega(c) &\text{if } c \in C, \cr 
        0 &\text{else.}
    \end{cases}
\end{equation*}

\subsection{Generalized loops and oriented surfaces}\label{sec: oriented surfaces}

In this section, we introduce the concepts of generalized loops and oriented surface.

\begin{definition}\label{def: generalized loop}
    A 1-chain \( \gamma \in C_1(\mathcal{L},\mathbb{Z}) \) with finite support is a \emph{generalized loop} if 
    \begin{enumerate}
        \item for all \( e \in \Omega^1(\mathcal{L}) \), we have \( \gamma[e] \in \{ -1,0,1 \} \), and
        \item  \( \partial \gamma = 0. \)
    \end{enumerate}
\end{definition}

\begin{definition}
    Let \( \gamma \in C_1(B_N,\mathbb{Z}) \) be a generalized loop. A \( 2 \)-chain \( q \in C_2(B_N,\mathbb{Z}) \) is an \emph{oriented surface} with \emph{boundary} \( \gamma \) if \( \partial q = \gamma. \)
\end{definition}

In Figure~\ref{fig: oriented surface}, we give an example of an oriented surface and the corresponding boundary.

\begin{figure}[htp]
    \centering 
    \begin{subfigure}[t]{0.45\textwidth}
        \centering
        \begin{tikzpicture}[scale=0.9,every node/.style={minimum size=1cm},on grid] 
            
            \begin{scope}[every node/.append     
            style={yslant=-0.5}, yslant=-0.5, decoration={markings, mark=at position 0.5 with {\arrow{>}}}]    
               
                \foreach \y in {0,1,2} {   \draw[semithick,draw=detailcolor07!5!black, postaction={decorate}] (0,\y+0.02) -- (0,\y+0.98) node[midway, anchor=east, xshift=1ex]{\footnotesize \( -1 \)};};  
                
                \foreach \x in {0,1,2} {   \draw[semithick,draw=detailcolor07!1!black, postaction={decorate}] (\x+0.98,0) -- (\x+0.02,0) node[midway, anchor=south, yshift=-1ex]{\footnotesize \( -1 \)};};   
                
            \end{scope}

            \begin{scope}[every node/.append style={yslant=0.5},yslant=0.5, decoration={markings, mark=at position 0.5 with {\arrow{>}}}]    
                
                \foreach \x in {3,4} {   \draw[semithick,draw=detailcolor07!1!black, postaction={decorate}] (\x+0.02,-3) -- (\x+0.98,-3) node[midway, anchor=south, yshift=-1ex]{\footnotesize \( 1 \)};};  
                
                \foreach \x in {4} {   \draw[semithick,draw=detailcolor07!20!black, postaction={decorate}] (\x+0.02,0) -- (\x+0.98,0) node[midway, anchor=south, yshift=-1ex]{\footnotesize \( \mathllap{-}1 \)};}; 
                
                \foreach \y in {-3,-2,-1} {   \draw[semithick,draw=detailcolor07!5!black, postaction={decorate}] (5,\y+0.02) -- (5,\y+0.98) node[midway, anchor=west, xshift=-1.5ex]{\footnotesize \( 1 \)};};  
             
        \end{scope}

        \begin{scope}[every node/.append style={yslant=0.5,xslant=-1},yslant=0.5,xslant=-1, decoration={markings, mark=at position 0.5 with {\arrow{>}}}]    
               
                \foreach \y in {0,1,2} {   \draw[semithick,draw=detailcolor07!20!black, postaction={decorate}] (4,\y+0.02) -- (4,\y+0.98);};
                
                \foreach \x in {3} {   \draw[semithick,draw=detailcolor07!20!black, postaction={decorate}] (\x+0.02,3) -- (\x+0.98,3);};
             
        \end{scope}

        \draw (0.2,3.6) node[] {\footnotesize \( -1 \)}; 
        
        \draw (1.6,3.6) node[] {\footnotesize \( 1 \)};
        \draw (2.6,3.1) node[] {\footnotesize \( 1 \)};
        \draw (3.6,2.6) node[] {\footnotesize \( 1 \)};
        
    \end{tikzpicture}
    \caption{A generalized loop \( \gamma \in C_1(B_N,\mathbb{Z}) \).}
    \end{subfigure}
          \hfil
    \begin{subfigure}[t]{0.45\textwidth}
        \centering
        \begin{tikzpicture}[scale=0.9, every node/.style={minimum size=1cm},on grid]
            \begin{scope}[every node/.append style={yslant=-0.5}, yslant=-0.5]
            
                \foreach \x in {0,1,2}
                    \foreach \y in {0,1,2}
                    {
                        \fill[color=detailcolor07, fill opacity=0.6] (\x+0.02,\y+0.02) rectangle +(0.96,0.96);
                        \node at (\x+0.5,\y+0.5) {$\circlearrowright$ };
                        \draw (\x+0.85,\y+0.2) node[align=center] {\footnotesize $\mathllap{-}1$};
                    };
        \end{scope}
        
        \begin{scope}[every node/.append style={yslant=0.5},yslant=0.5]
            
            \foreach \x in {3,4}
                \foreach \y in {-3,-2,-1}
                    {
                        \fill[color=detailcolor07, fill opacity=0.3] (\x+0.02,\y+0.02) rectangle +(0.96,0.96);
                        \node at (\x+0.5,\y+0.5) {$\circlearrowleft$ };
                        \draw (\x+0.85,\y+0.2) node[align=center] {\footnotesize $1$};
                    };

\end{scope}
\begin{scope}[every node/.append style={yslant=0.5,xslant=-1},yslant=0.5,xslant=-1]  
            
            \foreach \x in {3}
                \foreach \y in {0,1,2}
                    {
                        \fill[color=detailcolor07, fill opacity=0.2] (\x+0.02,\y+0.02) rectangle +(0.96,0.96);
                        \node at (\x+0.5,\y+0.5) {$\circlearrowleft$ };
                        \draw (\x+0.85,\y+0.2) node[align=center] {\footnotesize $1$};
                    }; 
        \end{scope}
    \end{tikzpicture}
    \caption{An oriented surface \( q \in C_2(B_N,\mathbb{Z})\) whose boundary equals   \( \gamma \).}
    \end{subfigure}
  
    \caption{In the two figures above, we draw illustrations of an oriented surface and its boundary.}
    \label{fig: oriented surface}
\end{figure}
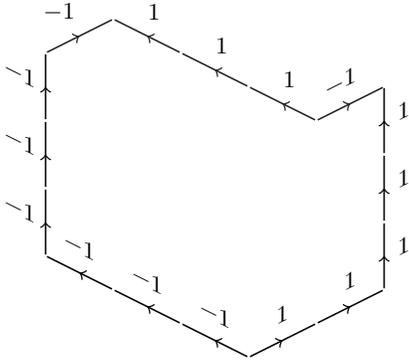
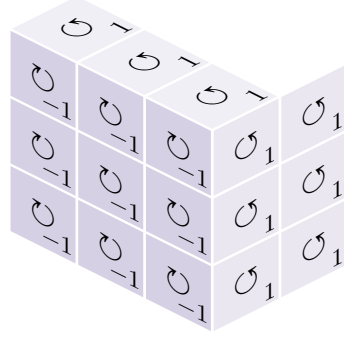

We recall that by Stokes' theorem (see Section~\ref{sec: derivative}), for any \( q \in C_2(B_N,G) \) and any \( \sigma \in \Omega^1(B_N,G) \), we have
\begin{equation*} 
      \sigma(\partial q) = d\sigma(q).
  \end{equation*}

The motivation for introducing oriented surfaces is the following lemma, which gives a connection between generalized loops and oriented surfaces.

\begin{lemma}\label{lemma: oriented loops}
    Let \( \gamma \in C_1(B_N, \mathbb{Z}) \) be a generalized loop, and let \( B \subseteq B_N \) be a box containing the support of \( \gamma \). Then there is an oriented surface \( q \in C_2(B_N, \mathbb{Z}) \) with support contained in \( B \) such that \( \gamma \) is the boundary of \( q \).
\end{lemma}

\begin{proof}
     Define \( \sigma_\gamma \in \Omega^1(B_N,\mathbb{Z}) \) by letting, for each \( e \in C^+_1(B_N) \),
    \begin{equation*}
        \sigma_\gamma(e) \coloneqq \gamma[e].
    \end{equation*}
    We now show that \( \delta \sigma_\gamma = 0. \) To this end, note first that for any \( a \in C_0^+(B_N) \), we have
    \begin{equation*}
        \delta \sigma_\gamma(a)
        =  
        \sigma_\gamma(\hat \partial a)
        =
        \sum_{e \in C_1^+(B_N)} \hat \partial a[e]\sigma_\gamma(e)
        =
        \sum_{e \in C_1^+(B_N)} \hat \partial a[e] \gamma[e]
        =
        \sum_{e \in C_1^+(B_N)} \partial e[a] \gamma[e]
        =
        \partial \gamma[a].
    \end{equation*} 
    Since \( \gamma \) is a generalized loop, we have \( \partial \gamma = 0 \), and hence it follows that \( \delta \sigma_\gamma = 0. \)
    Next, since the support of \( \gamma \) is contained in \( B \), the support of  \( \sigma_\gamma \) is contained in \( B \). 
    Consequently, Lemma~\ref{lemma: lemma 2.7} implies that there is a $2$-form \( \omega_q \in \Omega^2(B_N,\mathbb{Z})\), whose support is contained in \( B \), which is such that \( \delta \omega_q = \sigma_\gamma \).
    Since for any \( e \in C_1(B_N) \), we have \( \delta \omega_q(e) =  \omega_q(\hat \partial e) \), it follows that
    \begin{equation*}
        \sigma_\gamma(e) = \omega_q(\hat \partial e).
    \end{equation*}
    
    Let \( q \in C_2(B_N,\mathbb{Z}) \) be defined, for \( p \in C_2^+(B_N) \), by
    \begin{equation*}
        q[p] = \omega_q(p).
    \end{equation*}
    If \( \sigma \in \Omega^1(B_N,G) \), then
    \begin{equation*}
        \begin{split}
            &\sigma(\partial q)
            =
            \sum_{e \in C_1^+(B_N)} \partial q[e] \sigma(e)
            =
            \sum_{e \in C_1^+(B_N)} \biggl( \sum_{p \in C_2^+(B_N)} \partial p[e] q[p]  \biggr) \sigma(e)
            \\&\qquad=
            \sum_{e \in C_1^+(B_N)} \biggl( \sum_{p \in C_2^+(B_N)} \partial p[e] \omega_q(p)  \biggr) \sigma(e)
            = 
            \sum_{e \in C_1^+(B_N)} \biggl( \sum_{p \in C_2^+(B_N)} \hat \partial  e[p] \omega_q(p)  \biggr) \sigma(e)
            \\&\qquad= 
            \sum_{e \in C_1^+(B_N)} \biggl( \omega_q(\hat \partial e)  \biggr) \sigma(e)
            = 
            \sum_{e \in C_1^+(B_N)}   \sigma_\gamma(e)  \sigma(e)
            = 
            \sum_{e \in C_1^+(B_N)}   \gamma[e]  \sigma(e) = \sigma(\gamma).
        \end{split}
    \end{equation*}
    This concludes the proof.
\end{proof}

Now suppose that a generalized loop \( \gamma \) is given. An edge \( e \in \gamma \) is said to be a \emph{corner edge} in \( \gamma \) if there is another edge \( e' \in \gamma \) and a plaquette \( p \in \hat \partial e \) such that \( p \in \pm \hat \partial e' \) (see Figure~\ref{fig: generalized loop and corner edges}). We define \( \gamma_c \in C_1(B_N,\mathbb{Z}) \) for \( c' \in C_1(B_N) \) by 
\begin{equation}\label{def: gammac}
    \gamma_c[c'] \coloneqq \begin{cases}
        \gamma_c[c'] &\text{if } c' \text{ is a corner edge of } \gamma, \cr
        0 &\text{else.}
    \end{cases}
\end{equation}
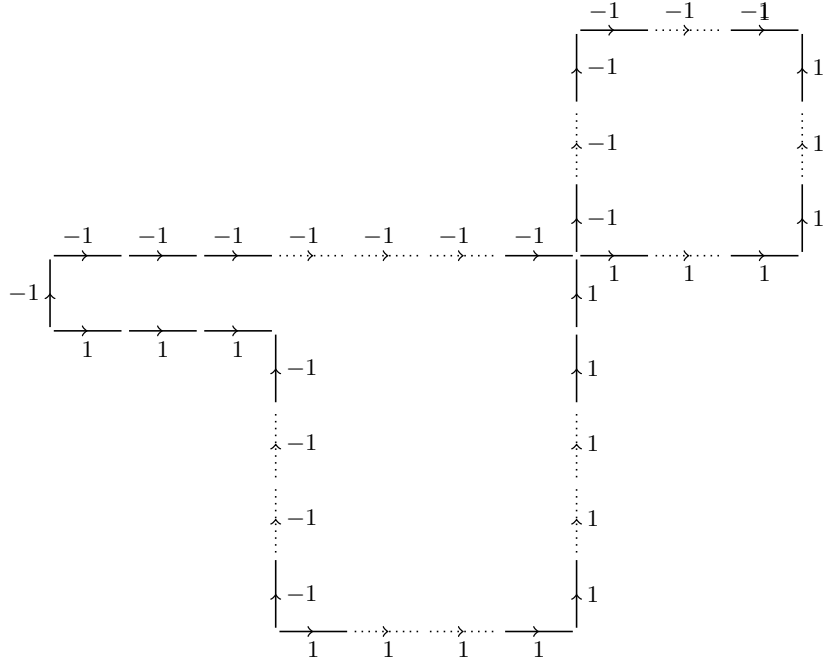
\begin{figure}
    \centering
    \begin{tikzpicture}
        \begin{scope}[semithick, decoration={markings, mark=at position 0.5 with {\arrow{>}}},shorten <=.5mm,shorten >=.5mm]   
        
            \draw[black, postaction={decorate}] (0,0) --(1,0) node[midway, anchor=north]{\footnotesize \( 1 \)};
            \draw[dotted, postaction={decorate}] (1,0)--(2,0) node[midway, anchor=north]{\footnotesize \( 1 \)};
            \draw[dotted, postaction={decorate}] (2,0)--(3,0) node[midway, anchor=north]{\footnotesize \( 1 \)}; 
            \draw[black, postaction={decorate}] (3,0)--(4,0) node[midway, anchor=north]{\footnotesize \( 1 \)};
            
            \draw[black, postaction={decorate}] (4,0)--(4,1) node[midway, anchor=west]{\footnotesize \( 1 \)}; 
            \draw[dotted, postaction={decorate}] (4,1)--(4,2) node[midway, anchor=west]{\footnotesize \( 1 \)}; 
            \draw[dotted, postaction={decorate}] (4,2)--(4,3) node[midway, anchor=west]{\footnotesize \( 1 \)}; 
            \draw[black, postaction={decorate}] (4,3)--(4,4) node[midway, anchor=west]{\footnotesize \( 1 \)}; 
            \draw[black, postaction={decorate}] (4,4)--(4,5) node[midway, anchor=west]{\footnotesize \( 1 \)};

            \draw[black, postaction={decorate}] (4,5)--(5,5) node[midway, anchor=north]{\footnotesize \( 1 \)}; 
            \draw[dotted, postaction={decorate}] (5,5)--(6,5) node[midway, anchor=north]{\footnotesize \( 1 \)}; 
            \draw[black, postaction={decorate}] (6,5)--(7,5) node[midway, anchor=north]{\footnotesize \( 1 \)}; 
            
            \draw[black, postaction={decorate}] (7,5)--(7,6) node[midway, anchor=west]{\footnotesize \( 1 \)}; 
            \draw[dotted, postaction={decorate}] (7,6)--(7,7) node[midway, anchor=west]{\footnotesize \( 1 \)}; 
            \draw[black, postaction={decorate}] (7,7)--(7,8) node[midway, anchor=west]{\footnotesize \( 1 \)};

            \draw[black, postaction={decorate}] (6,8)--(7,8) node[midway, anchor=south]{\footnotesize \( 1 \)} node[midway, anchor=south]{\footnotesize \( \mathllap{-}1 \)};  
            \draw[dotted, postaction={decorate}] (5,8)--(6,8) node[midway, anchor=south]{\footnotesize \( \mathllap{-}1 \)}; 
            \draw[black, postaction={decorate}] (4,8)--(5,8) node[midway, anchor=south]{\footnotesize \( \mathllap{-}1 \)}; 
            
            \draw[black, postaction={decorate}] (4,7)--(4,8) node[midway, anchor=west]{\footnotesize \( -1 \)};  
            \draw[dotted, postaction={decorate}] (4,6)--(4,7) node[midway, anchor=west]{\footnotesize \( -1 \)};  
            \draw[black, postaction={decorate}] (4,5)--(4,6) node[midway, anchor=west]{\footnotesize \( -1 \)};

            \draw[black, postaction={decorate}] (3,5)--(4,5) node[midway, anchor=south]{\footnotesize \( \mathllap{-}1 \)}; 
            \draw[dotted, postaction={decorate}] (2,5)--(3,5) node[midway, anchor=south]{\footnotesize \( \mathllap{-}1 \)}; 
            \draw[dotted, postaction={decorate}] (1,5)--(2,5) node[midway, anchor=south]{\footnotesize \( \mathllap{-}1 \)}; 
            \draw[dotted, postaction={decorate}] (0,5)--(1,5) node[midway, anchor=south]{\footnotesize \( \mathllap{-}1 \)}; 
            \draw[black, postaction={decorate}] (-1,5)--(0,5) node[midway, anchor=south]{\footnotesize \( \mathllap{-}1 \)}; 
            \draw[black, postaction={decorate}] (-2,5)--(-1,5) node[midway, anchor=south]{\footnotesize \( \mathllap{-}1 \)}; 
            \draw[black, postaction={decorate}] (-3,5)--(-2,5) node[midway, anchor=south]{\footnotesize \( \mathllap{-}1 \)}; 
            
            \draw[black, postaction={decorate}] (-3,4)--(-3,5) node[midway, anchor=east]{\footnotesize \( \mathllap{-}1 \)}; 
            
            \draw[black, postaction={decorate}] (-3,4)--(-2,4) node[midway, anchor=north]{\footnotesize \( 1 \)}; 
            \draw[black, postaction={decorate}] (-2,4)--(-1,4) node[midway, anchor=north]{\footnotesize \( 1 \)}; 
            \draw[black, postaction={decorate}] (-1,4)--(-0,4) node[midway, anchor=north]{\footnotesize \( 1 \)}; 
            
            \draw[black, postaction={decorate}] (0,3)--(0,4) node[midway, anchor=west]{\footnotesize \( -1 \)}; 
            \draw[dotted, postaction={decorate}] (0,2)--(0,3) node[midway, anchor=west]{\footnotesize \( -1 \)}; 
            \draw[dotted, postaction={decorate}] (0,1)--(0,2) node[midway, anchor=west]{\footnotesize \( -1 \)}; 
            \draw[black, postaction={decorate}] (0,0)--(0,1) node[midway, anchor=west]{\footnotesize \( -1 \)}; 
             
        \end{scope}
    \end{tikzpicture}
    \caption{In the figure above, we illustrate a generalized loop \( \gamma \) (dotted and solid edges), with corner edges drawn as solid edges.}
    \label{fig: generalized loop and corner edges}
\end{figure}

If \( q \) is an oriented surface with boundary \( \gamma \) and \( p \in C_2^+(B_N) \), then we say that \( p \) is an \emph{internal plaquette} of \( q \) if 
\begin{enumerate}[label=(\roman*)]
    \item \( q[p] \neq 0 \), and
    \item for each \( e \in C_1^+(B_N) \), we have either \(  \partial p [e]  = 0 \) or \( \gamma[e]  = 0. \)
\end{enumerate}
An edge \( e \in C_1^+(B_N) \) is said to be an \emph{internal edge} of \( q \) if 
\begin{enumerate}[label=(\roman*)]
    \item there is \( p \in \support q \) with \( \partial p[e] \neq 0 ,\) and
    \item \( \partial q[e] = 0. \)
\end{enumerate}

\subsection{Additional notation and standing assumptions}\label{section:standing}

In the rest of the paper, we will only consider the lattice \( \mathbb{Z}^4 \). In addition, we will assume that a finite cyclic group $G = \mathbb{Z}_n$ with \( n \geq 2 \), and a faithful one-dimensional representation \( \rho \) of \( G \) have been fixed.
We recall that any such representation $\rho$ is unitary and has the form \eqref{rhoexplicit}.

It will be useful to define the function \( \phi_\beta \colon G \to \mathbb{R} \), by 
\begin{equation}\label{wilson-action-piece}
    \phi_\beta(g) \coloneqq e^{\beta \Re  \rho(g)}, \quad g \in  G.
\end{equation}

Elements in \( \Omega^2_0(B_N,G) \) will be referred to as \emph{plaquette configurations}, and elements in $\Omega^1(B_N,G)$ will be referred to as a \emph{spin configurations}.

\section{\texorpdfstring{\( \mu_{N,\beta}\)}{LGT} as a measure on plaquette configurations}\label{sec: mubetaNsubsec}
While the probability measure $\mu_{N,\beta}$ is defined as a measure on spin configurations, we are primarily interested in the corresponding measure on plaquette configurations. 
To describe the latter measure, we first note that the exterior derivative operator \( d \) maps spin configurations \( \sigma \in  \Omega^1(B_N,G)\)  to plaquette configurations \( \omega \in \Omega^2_0(B_N,G).\)
By the Poincar\'e lemma (Lemma~\ref{lemma: poincare}), whenever \( \omega \in \Omega^2_0(B_N,G)\), there is a \(\sigma \in  \Omega^1(B_N,G) \) such that \( d \sigma = \omega \),  and each spin configuration \( \omega \in  \Omega^2_0(B_N,G) \) corresponds to the same number of spin configurations \( \sigma \in \Omega^1(B_N,G) \). Combining there observations, it follows that for any \( \omega \in \Omega^2_0(B_N,G) \) we have
\begin{align*}
   &\mu_{N,\beta}\bigl(\{ \sigma \in \Omega^1(B_N,G) \colon d\sigma = \omega \}\bigr) = 
    \frac{\sum_{\sigma \in \Omega^1(B_N,G)\colon d\sigma = \omega}  e^{\beta \sum_{p \in C_2(B_N)} \Re \rho(d\sigma(p))} }{\sum_{\sigma \in \Omega^1(B_N,G)} e^{\beta \sum_{p \in C_2(B_N)} \Re \rho(d\sigma(p))}}
    \\&\qquad= 
    \frac{   e^{\beta \sum_{p \in C_2(B_N)} \Re \rho(\omega(p))} }{\sum_{\omega' \in \Omega^2_0(B_N,G)} e^{\beta \sum_{p \in C_2(B_N)} \Re \rho(\omega'(p))}} .
\end{align*} 
We will abuse notation slightly and write \( \mu_{N,\beta} \bigl( \{ \omega \} \bigr) \) to denote the probability on the right-hand side of the previous equation, and thus use \( \mu_{N,\beta} \) also as a measure on plaquette configurations.

\section{Ginibre's inequality and the existence of \texorpdfstring{\( \langle W_\gamma \rangle_\beta \)}{the infinite limit}}\label{sec: ginibre}

In this section, we state and prove a result (Theorem~\ref{theorem: limit exists} below) which shows existence and translation invariance of the infinite volume limit \eqref{wilsonlimit} of the expectation of Wilson loop observables. This result is well known, and is often mentioned in the literature as a direct consequence of the Ginibre inequalities. However, since we have not found a clean reference which includes a proof of this claim, we give such a proof below. 
In this theorem, we will be interested in real-valued functions \( f \colon \Omega^1(\mathbb{Z}^4, G) \to \mathbb{R}\) which only depend on the spins of edges in \( C_1(B_N) \) for some \( N \geq 1 \), i.e., functions such that \( f(\sigma) = f(\sigma|_{C_1(B_N)}) \) for all \( \sigma \in \Omega^1(\mathbb{Z}^4, G) \).
When this is the case and \( N \geq M \), we abuse notation and let \( f \) denote also the natural restriction of \( f \) to \( \Omega^1(B_N) \).
We now state the main result of this section.
\begin{theorem}\label{theorem: limit exists}
    Let \( G = \mathbb{Z}_n \), and let \( f \colon \Omega^1(\mathbb{Z}^4,G) \to \mathbb{R}\) be a real-valued function, which only depends on the spins of edges in \( C_1(B_M) \) for some integer \( M \geq 1 \).
    Let \( \beta \geq 0 \). Then the following hold.
    \begin{enumerate}[label=(\roman*)]
        \item The limit \( \lim_{N \to \infty} \mathbb{E}_{N,\beta}\bigl[ f(\sigma) \bigr] \) exists.
        \item For any translation \( \tau \) of \( \mathbb{Z}^n \), we have \( \lim_{N \to \infty} \mathbb{E}_{N,\beta} \bigl[f \circ \tau(\sigma)\bigr] =  \lim_{N \to \infty} \mathbb{E}_{N,\beta} \bigl[ f(\sigma) \bigr] \).
    \end{enumerate} 
\end{theorem} 

In the proof of Theorem~\ref{theorem: limit exists}, we will use a well-known estimate known as one of the Ginibre inequalities. For easy reference, we state it here using the terminology of~\cite{g1970}, but remark that we only use it in a very special case as explained in the proof of Theorem~\ref{theorem: limit exists}.

\begin{lemma}[Ginibre's inequality, Proposition~3 in~\cite{g1970}]\label{lemma: Ginibre}
    Let \( K \) be a compact space, let \( \mu \) be a probability measure on \( K \), and let \( \mathcal{C}(K) \) be the  the algebra of complex-valued continuous functions on \( K \).  
    Let \( S \) be a subset of \( \mathcal{C}(K) \) which is invariant under complex conjugation, and which is such that for any \( f_1,\ldots, f_m \in S \) and any choice of signs \( s_1, s_2, \ldots, s_m \in \{ -1,1 \} \), we have
    \begin{equation}\label{equation: Q3}
        \iint \mu(x) \mu(y) \prod_{i=1}^m (f_i(x) + s_i f_i(y)) \geq 0.
    \end{equation}
    Let \( \cone(S)  \) denote the norm closure of the set of polynomials of elements of \( S \cup \{ 1 \} \) with positive coefficients (with respect to the supremum norm  \( \|f \|_\infty = \sup_{x \in K} |f(x)| \)).
    For \( f,h \in \mathcal{C}(K) \),  define
    \begin{equation*}
        \langle f \rangle_h = \int f(x) e^{-h(x)}\, \mu(x) \bigg/ \int e^{-{h(x)}} \, \mu(x).
    \end{equation*}
    Then, for any \( f,g,-h \) in  \( \cone(S) \), the quantities \( \langle f \rangle_h \), \( \langle g \rangle_h \), and \( \langle fg \rangle_h \) are real,   and
\begin{equation}\label{Ginibresinequality}
    \langle fg \rangle_h \geq \langle f \rangle_h \langle g \rangle_h.
\end{equation}
\end{lemma}

\begin{proof}[Proof of Theorem~\ref{theorem: limit exists}]
    For \( N \geq 1 \), consider the group \(\Gamma^{(N)} =  (\Omega^1(B_N,G),+)\) which is isomorphic to the direct product \((\mathbb{Z}_n^{|C_1^+(B_N)|}, +)\).
    By Example~4 in~\cite{g1970}, the following statements hold.
    \begin{enumerate}
        \item The multiplicative characters of \( \Gamma^{(N)} \) (i.e., the group homomorphisms $\Gamma^{(N)} \to \mathbb{C}\smallsetminus \{0\}$) are given by the functions \( \sigma \mapsto \prod_{e \in C_1(B_N)} e^{2\pi i \rho(\sigma(e)) \rho(\sigma'(e))/n} \), where \( \sigma, \sigma' \in \Omega^1(B_N,G)\) with \(\sigma'\) fixed and we note that the value of \(e^{2\pi i \rho(\sigma(e)) \rho(\sigma'(e))/n}\) is independent of the choice of representatives of \( \sigma(e), \sigma'(e) \in \mathbb{Z}_n \cong \mathbb{Z}/n\mathbb{Z}\).
        \item Let \( S^{(N)} \) be the set of real parts of the characters of  \( \Gamma^{(N)} \). Then \( \cone(S^{(N)}) \) is the set of real-valued positive definite functions on \( \Omega^1(B_N,G) \).
        \item For any \( f_1,\ldots, f_n \in S^{(N)}\) and signs \( s_1, \ldots, s_m \in \{ -1,1 \} \), if we let \( \mu \) be the uniform measure on \( \Omega^1(B_N,G) \), then~\eqref{equation: Q3}  holds.
    \end{enumerate} 
    In particular this shows that for each \( N \geq 1 \), the assumptions of Lemma~\ref{lemma: Ginibre} hold with \( K = \Omega^1(B_N,G) \) equipped with the discrete topology, and \( S = S^{(N)} \), 
     
    Now fix \( N' \geq N \geq M\). 
    For \( \beta,\beta' \geq 0 \), consider the function \( h \), defined for \( \sigma \in \Omega^1(B_{N'},G) \) by
    \begin{equation*}
        -h(\sigma) \coloneqq -h_{N,N',\beta,\beta'}(\sigma) \coloneqq \beta \sum_{p \in C_2(B_N)} \Re \rho\bigl(d\sigma(p)\bigr)   + \beta' \sum_{p \in C_2(B_{N'})\smallsetminus C_2(B_N)} \Re  \rho\bigl(d\sigma(p)\bigr).
    \end{equation*}  
    Since for any \( p \in C_2(B_{N'}) \), the function
    \begin{equation*}
        \sigma \mapsto \rho\bigl(d\sigma(p)\bigr) ,\quad \sigma \in \Omega^1(B_{N'},G),
    \end{equation*}
    is a  character of \( \Gamma^{(N')} \), and since \( \beta,\beta' \geq 0 \), we have \( -h_{N,N',\beta,\beta'} \in \cone(S^{({N'})}) \).
     
    Since \( \Omega^1(B_M,G)\) is finite, the characters of \( \Gamma^{(M)} \) span the set of all real-valued functions on \( \Omega^1(B_M,G) \). In particular, this implies that there are functions \( g_1, g_2, \ldots, g_m \in S^{(M)} \) and real numbers \( a_1,a_2, \ldots, a_m \in \mathbb{R} \) such that \( f = a_1g_1 + \ldots + a_mg_m \).
    Since \( S^{(M)} \subseteq \cone (S^{(M)}) \), we have \( g_1,g_2, \ldots, g_m \in \cone (S^{(M)}) \). For any \( M' \geq M \) and \( j \in \{1,2, \ldots, m \} \), we abuse notation and let \( g_j \) denote the natural extension of \( g_j \) to \( \Omega^1(B_{M'},G) \).
    By definition, for any \( j \in \{1, \ldots, m \} \),  we have
    \begin{equation*}
        \begin{split}
            &\frac{d}{d\beta'} \langle g_j \rangle_h = 
            \Bigl\langle g_j(\sigma) \sum_{p \in C_2(B_{N'})\smallsetminus C_2(B_N)} \Re \rho\bigl (d\sigma(p)\bigr) \Bigr\rangle_{h} -   \langle g_j \rangle_h  \, \Bigl\langle \sum_{p \in C_2(B_{N'})  \smallsetminus C_2(B_N)} \Re\rho\bigl(d\sigma(p)\bigr) \Bigr\rangle_h.
        \end{split}
    \end{equation*}
    Thus, using that the functions \(g_j\) and \( \sigma \mapsto \sum_{p \in C_2(B_{N'})\smallsetminus C_2(B_N)}  \Re\rho\bigl(d\sigma(p)\bigr) \) belong to \( \cone(S^{(N')}) \), Ginibre's inequality \eqref{Ginibresinequality} yields
    \begin{equation*}
        \frac{d}{d\beta'} \langle g_j \rangle_{h_{N,N',\beta,\beta'}}  \geq 0,
    \end{equation*}
    and hence \( \langle g_j \rangle_{h_{N,N',\beta,\beta'}} \) is real and increasing in \( \beta' \) for \( \beta' \geq 0\). Since
    \begin{equation*}
        \mathbb{E}_{N,\beta} \bigl[g_j(\sigma)\bigr] = \langle g_j \rangle_{h_{N,N',\beta,0}}
    \end{equation*}
    this implies that 
    \begin{equation}\label{muNleqmuNprime}
        \mathbb{E}_{N,\beta} \bigl[ g_j(\sigma) \bigr] \leq \mathbb{E}_{N',\beta} \bigl[ g_j(\sigma) \bigr].
    \end{equation}
    Since \( g_j \) depends only on a finite set of edges, we have \( \|g_j \|_\infty < \infty \). Consequently, the sequence \( \mathbb{E}_{N,\beta} \bigl[ g_j(\sigma) \bigr] \) is both monotone and bounded, and hence \( \lim_{N \to \infty} \mu_{N,\beta} (g_j) \) exists. Finally, since \( f = \sum_{j=1}^m a_j g_j \) is a finite sum, we have 
    \begin{equation*}
        \sum_{j=1}^m a_j \lim_{N \to \infty} \mathbb{E}_{N,\beta} \bigl[ g_j(\sigma) \bigr]
        =
         \lim_{N \to \infty} \sum_{j=1}^m a_j\mu_{N,\beta} \bigl[ g_j(\sigma) \bigr]
        =
         \lim_{N \to \infty}  \mathbb{E}_{N,\beta} \biggl[\sum_{j=1}^m a_j g_j(\sigma)\biggr]
        =
         \lim_{N \to \infty}  \mathbb{E}_{N,\beta} \bigl[f(\sigma)\bigr].
    \end{equation*}  
    This concludes the proof of (i).

    We now show that the limit \( \lim_{N \to \infty} \mathbb{E}_{N,\beta} (f) \) is translation invariant. To see this, let \( \tau  \) be a translation of  \( \mathbb{Z}^4 \), and let \( N' \geq N \geq M \), where \( N' \) is so large that  \( B_N \), \( \tau B_N \), and  \( \tau^{-1} B_N \) are all contained in \( B_{N'} \).
    Define \( h_N \coloneqq h_{N,N',\beta, 0} \) and \( h_{N'} \coloneqq h_{N,N',\beta, \beta} \), and note that
    \begin{equation*}
        h_N \circ \tau =  \beta \sum_{p \in C_2(B_N)} \Re  \rho\bigl(d\sigma(\tau(p)) \bigr) =  \beta \sum_{p \in \tau(C_2(B_N))} \Re \rho\bigl(d\sigma(p)\bigr) .
    \end{equation*}
    Then the same kind of argument that led to \eqref{muNleqmuNprime} gives 
    \begin{equation*}
        \langle g_j \rangle_{h_N \circ \tau} \leq \langle g_j \rangle_{h_{N'}} \quad \text{and} \quad 
        \langle g_j \rangle_{h_N \circ \tau^{-1}} \leq \langle g_j \rangle_{h_{N'}}.
    \end{equation*} 
    From this it follows that 
    \begin{equation*}
        \mathbb{E}_{N,\beta} \bigl[  g_j \circ \tau(\sigma) \bigr]
        =
        \langle  g_j \circ \tau   \rangle_{h_N} 
        =
        \langle g_j \rangle_{h_N \circ \tau^{-1}} \leq     \langle  g_j  \rangle_{h_{N'}} =  \mathbb{E}_{N',\beta} \bigl[  g_j(\sigma  ) \bigr].
    \end{equation*}
    By letting first \( N' \) and then \( N \) tend to infinity, we obtain
    \begin{equation}\label{eq: translation invariance 1}
        \lim_{N \to \infty} \mathbb{E}_{N,\beta} \bigl[  g_j \circ \tau(\sigma) \bigr] \leq \lim_{N' \to \infty} \mathbb{E}_{N',\beta} \bigl[  g_j(\sigma)  \bigr].
    \end{equation}
    Similarly, we have
    \begin{equation*}
        \mathbb{E}_{N,\beta} \bigl[  g_j(\sigma  ) \bigr]
        =
        \langle  g_j    \rangle_{h_N} 
        =
        \langle  g_j \circ \tau \circ \tau^{-1}   \rangle_{h_N} 
        =  
        \langle  g_j \circ \tau  \rangle_{h_N \circ \tau} 
        \leq
        \langle  g_j \circ \tau  \rangle_{h_{N'}} 
        =
        \mathbb{E}_{N',\beta} \bigl[ g_j \circ \tau (\sigma) \bigr].
    \end{equation*}
    Again letting first \( N' \) and then \( N \) tend to infinity, we obtain
    \begin{equation}\label{eq: translation invariance 2}
        \lim_{N \to \infty} \mathbb{E}_{N,\beta} \bigl[  g_j(\sigma ) \bigr] \leq \lim_{N' \to \infty} \mathbb{E}_{N',\beta} \bigl[ g_j \circ \tau(\sigma ) \bigr].
    \end{equation}
    By combining~\eqref{eq: translation invariance 1} and~\eqref{eq: translation invariance 2}, we obtain
    \begin{equation*} 
        \lim_{N \to \infty} \mathbb{E}_{N,\beta} \bigl[ g_j(\sigma ) \bigr] = \lim_{N' \to \infty} \mathbb{E}_{N',\beta} \bigl[  g_j \circ \tau (\sigma) \bigr].
    \end{equation*}
    Since \( f = \sum_{j=1}^m a_j g_j \) is a finite sum, this shows that (ii) holds. This concludes the proof.
\end{proof}

\section{Vortices}\label{sec: vortices}

\subsection{Definition}
 
We say that a plaquette configuration \( \omega \in \Omega^2_0(B_N,G) \) is \emph{irreducible} if there is no non-empty set \( P \subsetneq \support \omega \) such that \( \omega|_P \in \Omega^2_0(B_N,G) \). Equivalently, \( \omega \in \Omega^2_0(B_N,G) \) is irreducible if there is no non-empty set \( P \subsetneq \support \omega \) such that \( d(\omega|_P)=0 \).
The following definition will play a crucial role in this paper. 
\begin{definition}[Vortex]\label{def: vortex}
    Let \( \sigma \in \Omega^1(B_N,G) \) be a spin configuration. A non-trivial and irreducible plaquette configuration \( \nu \in \Omega^2_0(B_N,G) \) is said to be a \emph{vortex} in \( \sigma \) if \( (d\sigma)|_{\support \nu} = \nu.\) 
\end{definition}%
In Figure~\ref{fig: vortex}, we give an illustration of this definition.

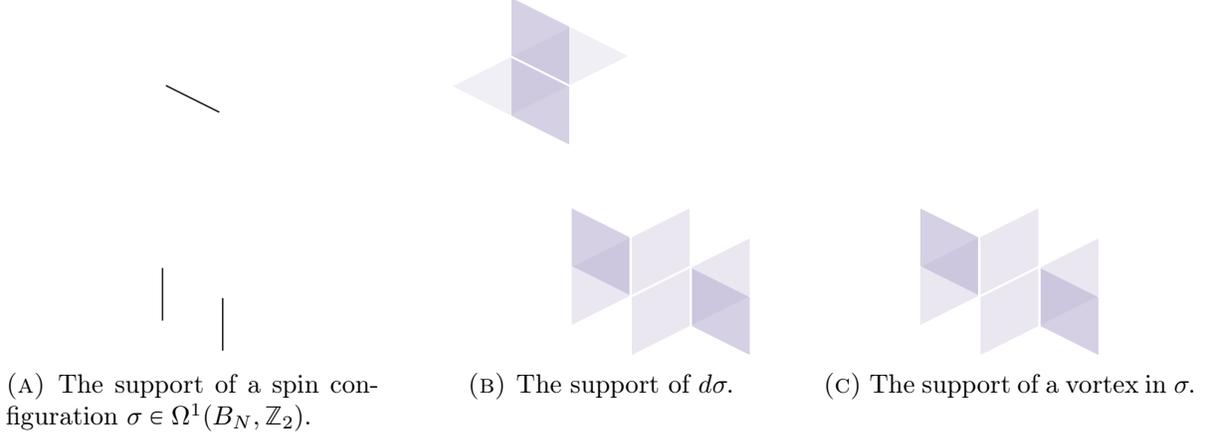
\begin{figure}[htp]
    \centering 
    
    \begin{subfigure}[t]{0.3\textwidth}
        \centering
        \begin{tikzpicture}[scale=0.8,every node/.style={minimum size=1cm},on grid] 
            
            \begin{scope}[every node/.append     
            style={yslant=-0.5}, yslant=-0.5,shorten <=.5mm,shorten >=.5mm, semithick, draw=detailcolor07!20!black]
                \draw (0,0) -- (0,1);
                \draw (1,0) -- (1,1); 
            \end{scope}
        
        \begin{scope}[every node/.append style={yslant=0.5,xslant=-1},yslant=0.5,xslant=-1,shorten <=.5mm,shorten >=.5mm, semithick, draw=detailcolor07!20!black] 
            \draw (4,3) -- (4,4);
        \end{scope} 
    \end{tikzpicture}
    \caption{The support of a spin configuration \( \sigma \in \Omega^1(B_N,\mathbb{Z}_2). \)}
    \end{subfigure}
    \hfil
    \begin{subfigure}[t]{0.3\textwidth}
        \centering
        \begin{tikzpicture}[scale=0.8, every node/.style={minimum size=1cm},on grid]
            
            \begin{scope}[every node/.append style={yslant=-0.5}, yslant=-0.5,color=detailcolor07, fill opacity=0.6]
            
                \fill (-1+0.02,0.02) rectangle +(0.96,0.96);
                \fill (1.02,0.02) rectangle +(0.96,0.96);
            \end{scope}
        
            \begin{scope}[shift={(0,0)}, every node/.append style={yslant=0.5},yslant=0.5,color=detailcolor07, fill opacity=0.3] 
            
                \fill (0+0.02,0+0.02) rectangle +(0.96,0.96); 
                \fill (-1+0.02,0+0.02) rectangle +(0.96,0.96);
            \end{scope}
        
            \begin{scope}[shift={(1,-0.5)}, every node/.append style={yslant=0.5},yslant=0.5,color=detailcolor07, fill opacity=0.3] 
            
                \fill (0+0.02,0+0.02) rectangle +(0.96,0.96);
                \fill (-1+0.02,0+0.02) rectangle +(0.96,0.96);
            \end{scope} 
            
            \begin{scope}[shift={(-2,4)}, every node/.append style={yslant=-0.5}, yslant=-0.5,color=detailcolor07, fill opacity=0.6]

                \fill (0.02,0.02) rectangle +(0.96,0.96);
                \fill (0.02,-1+0.02) rectangle +(0.96,0.96);
            \end{scope}
            
            \begin{scope}[shift={(-1,3.5)},every node/.append style={yslant=0.5,xslant=-1},yslant=0.5,xslant=-1,color=detailcolor07, fill opacity=0.2] 
            
            		\fill (0+0.02,0+0.02) rectangle +(0.96,0.96);
            		\fill (-1+0.02,0+0.02) rectangle +(0.96,0.96);
			\end{scope}

    \end{tikzpicture}
    \caption{The support of \( d\sigma. \)}
    \end{subfigure}
    \hfil
    \begin{subfigure}[t]{0.3\textwidth}
        \centering
        \begin{tikzpicture}[scale=0.8, every node/.style={minimum size=1cm},on grid]
            
            \begin{scope}[every node/.append style={yslant=-0.5}, yslant=-0.5,color=detailcolor07, fill opacity=0.6]
            
                \fill (-1+0.02,0.02) rectangle +(0.96,0.96);
                \fill (1.02,0.02) rectangle +(0.96,0.96);
            \end{scope}
        
            \begin{scope}[shift={(0,0)}, every node/.append style={yslant=0.5},yslant=0.5,color=detailcolor07, fill opacity=0.3] 
            
                \fill (0+0.02,0+0.02) rectangle +(0.96,0.96); 
                \fill (-1+0.02,0+0.02) rectangle +(0.96,0.96);
            \end{scope}
        
            \begin{scope}[shift={(1,-0.5)}, every node/.append style={yslant=0.5},yslant=0.5,color=detailcolor07, fill opacity=0.3] 
            
                \fill (0+0.02,0+0.02) rectangle +(0.96,0.96);
                \fill (-1+0.02,0+0.02) rectangle +(0.96,0.96);
            \end{scope}  

        \end{tikzpicture}
        \caption{The support of a vortex in~\( \sigma. \)}
    \end{subfigure}
    \caption{In the figures above, we draw an illustrations of an a spin configuration, the corresponding plaquette configuration, and a corresponding vortex.}
    \label{fig: vortex}
\end{figure}

When \( G = \mathbb{Z}_2 \), our definition of a vortex is similar but not the same as the definition of a vortex in~\cite{c2019}. In particular, the assumption of irreducibility in Definition~\ref{def: vortex} is a stronger assumption than the corresponding assumption in the definition of vortex in~\cite{c2019}.

\subsection{Vortex decompositions}

The next observation, although simple, will be important in the proofs of several later results.
\begin{lemma}\label{lemma: vortex flip}
    Let \( \omega \in \Omega^2_0(B_N,G) \), and let \( \nu \in \Omega^2_0(B_N,G) \) be such that \( \omega|_{\support \nu} = \nu \). 
    Then \( \omega|_{C_2(B_N) \smallsetminus \support \nu}  \in \Omega^2_0(B_N,G) \).
\end{lemma}

\begin{proof}
    Note first that \( \omega|_{C_2(B_N) \smallsetminus \support \nu} = \omega - \nu\). 
    Since \( \omega \in \Omega^2_0(B_N,G) \), we have \( d\omega = 0 \). Similarly, since \( \nu \in \Omega^2_0(B_N,G) \), we have \( d\nu = 0 \). Combining these observations, we obtain
    \begin{equation*}
        d \bigl( \omega|_{C_2(B_N) \smallsetminus  \support \nu}\bigr)  = d(\omega - \nu) = d\omega - d\nu = 0-0 = 0
    \end{equation*}
    and hence \( \omega|_{C_2(B_N) \smallsetminus  \support \nu} \in \Omega^2_0(B_N,G) \) as desired.
\end{proof}

\begin{lemma}\label{lemma: decomposition}
    Let \( \omega \in \Omega^2_0(B_N,G)\). Then either \( \omega \) is irreducible, or there are non-trivial \( \nu_1, \nu_2 \in \Omega^2_0(B_N,G) \) with disjoint supports contained in \( \support \nu \), such that \( \omega = \nu_1 + \nu_2. \)
\end{lemma}

\begin{proof}
    Assume that \( \omega \in \Omega^2_0(B_N,G) \) is not irreducible. Then there is a  non-empty set \( P \subsetneq \support \omega \) such that \( \nu_1 \coloneqq \omega|_P \in \Omega^2_0(B_N,G) \). By definition, \( P = \support \nu_1 \), and hence, by Lemma~\ref{lemma: vortex flip}, we also have \( \nu_2 \coloneqq \omega|_{C_2(B_N) \smallsetminus P} \in \Omega^2_0(B_N,G) \). 
    Since \( P \subsetneq \support \omega \) is non-trivial, both \( \nu_1 \) and \( \nu_2 \) must be non-trivial. 
    Next, since \( \support \nu_1 = P \) and \( \support \nu_2 \subseteq C_2(B_N) \smallsetminus P \), we clearly have \( \support \nu_1 \cap \support \nu_2 = \emptyset \).
    Finally, 
    \begin{equation*}
        \nu_1 + \nu_2 = \omega|_P + \omega|_{C_2(B_N)\smallsetminus P} = \omega.
    \end{equation*}
    This concludes the proof.
\end{proof}

\begin{lemma}[Compare with Lemma~3.3 in~\cite{c2019}] \label{lemma: sum of vortices}
    Let \( \omega \in \Omega^2_0(B_N,G) \). Then \(\omega\) can be written as a sum of irreducible 2-forms in \( \Omega^2_0(B_N,G) \) with disjoint supports.
    In particular, if \( \sigma \in \Omega^1(B_N,G) \), then \( d\sigma \) can be written as a sum of vortices in \( \sigma \) with disjoint supports.
\end{lemma}

\begin{proof}
    Note first that if \( \omega = 0 \) or \(| \support \omega| =2 \), then \( \omega \) must be irreducible, and hence in this case, the lemma trivially holds.
    Next, note that by definition, we have \( \omega|_{\support \omega} = \omega \), and hence either \( \omega \) is irreducible, in which case we are done, or \( \omega \) is not irreducible. 
    If \( \omega \) is not irreducible, then by Lemma~\ref{lemma: decomposition}, there are non-trivial \( \omega_1 , \omega_2 \in \Omega^2_0(B_N,G) \) with disjoint supports contained in \( \support \omega \), such that \( \omega = \omega_1 + \omega_2 \).
    Since \( \omega_1 \) and \( \omega_2 \) have disjoint supports, we must have \( |\support \omega_1|,|\support \omega_2|< |\support \omega|. \)
    The lemma follows by induction on the size of the support of the plaquette configurations. 
\end{proof}

\subsection{Minimal vortices}

Let \( \sigma \in \Omega^1(B_N,G) \). We say that a vortex \( \nu \) in \( \sigma \) is a \emph{minimal vortex} if its support does not contain any boundary plaquettes of \( C_2(B_N) \) and \( |\support \nu| = 12\). 
The next lemma explains in what sense the minimal vortices are indeed minimal.
\begin{lemma}\label{lemma: minimal vortex I}
    If \( \nu \) is a vortex in \( \sigma \in \Omega^1(B_N,G) \) whose support does not contain any boundary plaquettes of \( C_2(B_N) \), then 
    \begin{enumerate}[label=\textnormal{(\roman*)}]
        \item \( | \support \nu| \geq 12\), and
        \item if \( |\support \nu| = 12 \), then there exists an edge \( e_0 \in C_1^+(B_N) \) such that \( \support \nu = \pm \support \hat \partial e_0 \).\label{item: minimal vortex}
    \end{enumerate}
\end{lemma}
For a proof of Lemma~\ref{lemma: minimal vortex I}, see e.g. the proof of Lemma~7.1~in~\cite{c2019} or Lemma~3.4.6~in~\cite{cao20}. 
 
\begin{lemma}\label{lemma: minimal vortex II}
    If \( \nu \) is a minimal vortex in \( \sigma \in \Omega^1(B_N,G) \) whose support does not contain any boundary plaquettes of \( C_2(B_N) \), then there is an edge \( dx_j \in C_1(B_N) \) and \( g \in G\smallsetminus \{ 0 \} \) such that 
    \begin{equation}\label{eq: minimal vortex as df}
        \nu = d\bigl(g \, dx_j \bigr).
    \end{equation}
\end{lemma}

\begin{proof}
    Assume that \( \sigma \in \Omega^1(B_N,G) \) and that \( \nu \) is a minimal vortex in \( \sigma \). Let \( e_0 = dx_j \in C_1^+(B_N) \) be such that  \( \support \nu = \pm \support \hat \partial e_0  \) (such \( e_0 \) is guaranteed to exist by~Lemma~\ref{lemma: minimal vortex I}\ref{item: minimal vortex}).  
    Assume for contradiction that there is no \( g \in G\smallsetminus \{ 0 \} \) such that~\eqref{eq: minimal vortex as df} holds, or equivalently, such that for all \( p \in C_2(B_N) \), 
    \begin{equation}\label{nupgminusg}
        \nu(p) = \bigl(d(g \, dx_j )\bigr)(p) . 
    \end{equation} 
    Since the support of the boundary of any \( 3 \)-cell contains at most two plaquettes in \(  \support  \hat \partial e_0  \), \( \support \nu  = \pm \support \hat \partial e_0 \), and \( \nu \) is a 2-form, it follows from~\eqref{eq: Bianchi} that there must then exist \(c \in C_3(B_N)\) and two distinct plaquettes \( p_1,p_2 \in \support \nu  \cap  \pm \support  \partial c \) with \( \partial p_1[e_1]=1 \) and \( \partial p_2[e_0]=-1 \) such that \( \nu(p_1) \neq -\nu(p_2). \) Again since the support of the boundary of any \( 3 \)-cell contains at most two plaquettes in \(  \support \hat \partial e_0  \), we must have
    \begin{equation*}
        \nu(\partial c) = \nu(p_1) + \nu(p_2) \neq 0.
    \end{equation*}
    But this contradicts~\eqref{eq: Bianchi}, so the assumption must be false, and hence the desired conclusion follows.
\end{proof}
If a vortex \( \nu \) in \( \sigma \in \Omega^1(B_N,G) \) can be written as \( d(g\, dx_j) \) for some \( g \in G\smallsetminus \{ 0 \} \) and some \( dx_j \in C_1(B_N) \), then we say that \( \nu \) is a \emph{minimal vortex centered at \( dx_j\).}

\subsection{Distribution of vortices}\label{subsec: distribution of vortices}
  
In this section we give upper bounds for a few probabilities related to the distribution of vortices. In particular, we give proofs of the natural analogues of Corollaries 6.2 and 6.3 in~\cite{c2019}. Even when the structure group is \(  \mathbb{Z}_2 \), our proofs are different from the proofs in~\cite{c2019}.

\begin{lemma} \label{lemma: agreement probability} 
	Let \( \nu \in \Omega^2_0(B_N,G) \).  %
	Then 
	\begin{equation*}
	    \mu_{N,\beta} \bigl( \{ \sigma \in \Omega^1(B_N,G) \colon (d\sigma)|_{\support \nu} = \nu \} \bigr) \leq \prod_{p \in \support \nu} \frac{\phi_\beta  \bigl(\nu(p)\bigr)}{\phi_\beta(0)},
	\end{equation*} 
	where $\phi_\beta$ is the function defined in \eqref{wilson-action-piece}.
\end{lemma}
\begin{proof} 
    Let \( P \coloneqq \support \nu \). Further, let \( \mathcal{E}_P^\nu \coloneqq \bigl\{ \omega \in \Omega^2_0(B_N,G) \colon \omega|_P = \nu \bigr\} \) and, similarly, let \( \mathcal{E}_P^0 \coloneqq \bigl\{ \omega \in \Omega^2_0(B_N,G)  \colon \omega|_P = 0 \bigr\}. \)
    By the definition of \(\mu_{N,\beta}\) (see Section \ref{sec: mubetaNsubsec}), 
    \begin{equation*}
      \mu_{N,\beta} \bigl( \{ \sigma \in \Omega^1(B_N,G) \colon (d\sigma)|_{\support \nu} = \nu \} \bigr)
      =
        \mu_{N,\beta}(\mathcal{E}_P^\nu) =
        \frac{\sum_{\omega \in \mathcal{E}_P^\nu}  \prod_{p \in C_2(B_N)} \phi_\beta  \bigl(\omega(p)\bigr)}{\sum_{\omega \in \Omega^2_0(B_N,G)}  \prod_{p \in C_2(B_N)} \phi_\beta  \bigl(\omega(p)\bigr)}.
    \end{equation*}
    Since \( \mathcal{E}_P^0 \subseteq \Omega^1_0(B_N,G)  \), we have
    \begin{equation}\label{eq: eq 2 in proof}
        \frac{\sum_{\omega \in \mathcal{E}_P^\nu}  \prod_{p \in C_2(B_N)} \phi_\beta  \bigl(\omega(p)\bigr)}{\sum_{\omega \in \Omega^2_0(B_N,G)}  \prod_{p \in C_2(B_N)} \phi_\beta  \bigl(\omega(p)\bigr)} \leq
        \frac{\sum_{\omega \in \mathcal{E}_P^\nu}  \prod_{p \in C_2(B_N)} \phi_\beta  \bigl(\omega(p)\bigr)}{\sum_{\omega \in \mathcal{E}_P^0}  \prod_{p \in C_2(B_N)} \phi_\beta  \bigl(\omega(p)\bigr)}.
    \end{equation}
    Since the mapping \( \omega \mapsto \omega-\nu \) is a bijection from \( \mathcal{E}_P^\nu \) to \( \mathcal{E}_P^0 \), the right-hand side of~\eqref{eq: eq 2 in proof} is equal to
    \begin{equation}\label{eq: eq 3 in proof}
        \frac{\sum_{\omega \in \mathcal{E}_P^\nu}  \prod_{p \in C_2(B_N)} \phi_\beta  \bigl(\omega(p)\bigr)}{\sum_{\omega \in \mathcal{E}_P^\nu}  \prod_{p \in C_2(B_N)} \phi_\beta  \bigl((\omega-\nu)(p)\bigr)}
        =
        \frac{\sum_{\omega \in \mathcal{E}_P^\nu}  \prod_{p \in P} \phi_\beta  \bigl(\omega(p)\bigr)\prod_{p \in C_2(B_N)\smallsetminus P} \phi_\beta  \bigl(\omega(p)\bigr) }{\sum_{\omega \in \mathcal{E}_P^\nu}  \prod_{p \in P} \phi_\beta  \bigl((\omega-\nu)(p)\bigr) \prod_{p \in C_2(B_N)\smallsetminus P} \phi_\beta  \bigl((\omega-\nu)(p)\bigr)}.
    \end{equation}
    By definition, if \( p \in C_2(B_N) \smallsetminus P \) then \( \nu(p) = 0 \). Moreover, if \( p \in P \) and \( \omega \in \mathcal{E}_P^\nu \), then \( \omega(p) = \nu(p) \), and consequently,  \( (\omega-\nu)(p) = \omega(p) - \nu(p) = \nu(p)-\nu(p) = 0 \). From this it immediately follows that the right-hand side of~\eqref{eq: eq 3 in proof} is equal to
    \begin{equation*} 
        \begin{split}
            &\frac{\sum_{\omega \in \mathcal{E}_P^\nu}  \prod_{p \in P} \phi_\beta  \bigl(\nu(p)\bigr)\prod_{p \in C_2(B_N)\smallsetminus P} \phi_\beta  \bigl(\omega(p)\bigr) }{\sum_{\omega \in \mathcal{E}_P^\nu}  \prod_{p \in P} \phi_\beta  (0) \prod_{p \in C_2(B_N)\smallsetminus P} \phi_\beta  \bigl(\omega(p)\bigr)}
            =
            \frac{  \prod_{p \in P} \phi_\beta  \bigl(\nu(p)\bigr) \sum_{\omega \in \mathcal{E}_P^\nu} \prod_{p \in C_2(B_N)\smallsetminus P} \phi_\beta  \bigl(\omega(p)\bigr) }{ \prod_{p \in P} \phi_\beta  (0) \sum_{\omega \in \mathcal{E}_P^\nu}  \prod_{p \in C_2(B_N)\smallsetminus P} \phi_\beta  \bigl(\omega(p)\bigr)}
            \\&\qquad = 
            \frac{ \prod_{p \in P} \phi_\beta  \bigl(\nu(p)\bigr) }{ \prod_{p \in P} \phi_\beta  (0)  }
            = 
            \prod_{p \in P} \frac{\phi_\beta  \bigl(\nu(p)\bigr)}{\phi_\beta(0)}.
        \end{split}
    \end{equation*}
    Combining the previous equations, the desired conclusion follows. 
\end{proof}

\begin{lemma}\label{lemma: counting vortex configurations ii}
    For each \(  p_0 \in C_2(B_N) \) and each \( m \geq 6 \), there are at most \( 5^{m-1}(|G|-1)^m \) irreducible  \( \nu \in \Omega^2_0(B_N,G) \) with \( p_0 \in \support \nu \) and \( |\support \nu| = 2m \).
\end{lemma}

\begin{proof}
    We will prove that the statement of the lemma holds by constructing an injective map from the set of  irreducible 2-forms in \( \Omega^2_0(B_N,G) \) with support in \( 2m \) plaquettes, including \(  p_0 \), to a set of certain sequences \( \nu^{(1)} ,\nu^{(2)}, \ldots, \nu^{(m)} \) of \( G \)-valued 2-forms in \( \Omega^2(B_N,G) \), and then give an upper bound for the number of such sequences.
To this end, assume that arbitrary total orderings of the \( C_2(B_N) \) plaquettes and \( C_3(B_N) \) are given.

    Let \( \nu \in \Omega^2_0(B_N,G) \) be irreducible. Assume that \( p_0 \in \support \nu \) and \( |\support \nu|  = 2m \).
    Let \(\nu^{(0)} \coloneqq 0 \in \Omega^2(B_N,G)\), \( p_1 \coloneqq p_0 \), and define, for \(p \in C_2(B_N)\), 
    \begin{equation*}
        \nu^{(1)}(p) \coloneqq \begin{cases}
            \nu(p) &\text{if } p = p_1, \cr 
            -\nu(p) &\text{if } p = -p_1, \cr 
            0 &\text{otherwise.}
        \end{cases}
    \end{equation*}

    Now assume that for some \( k\in \{ 1,2, \ldots, m\} \), we are given 2-forms \( \nu^{(1)}, \nu^{(2)} , \ldots, \nu^{(k)} \)  such that, for each \( j \in \{ 1,2, \ldots, k \} \), we have
    \begin{enumerate}[label=(\roman*)]
        \item \( \support \nu^{(j)} \smallsetminus \support \nu^{(j-1)} = \{ p_j,-p_j \} \) for some \( p_j \in C_2(B_N) \), and
        \item \( \nu|_{\support \nu^{(j)}} = \nu^{(j)} \).
    \end{enumerate}
    Clearly this holds when \( k = 1 \).

    If \( d\nu^{(k)} = 0 \), then \( \nu^{(k)} \in \Omega^2_0(B_N,G) \) and, by (ii), we have \( \nu|_{\support \nu^{(k)}} = \nu^{(k)} \). Since \( \nu \) is irreducible by assumption, and \( \support \nu^{(k)} \neq \emptyset \) by (i), it follows that \( \nu^{(k)} = \nu \). By (i), we have \( \support \nu^{(k)} = 2k \), and hence, by definition, \( k = m \).
    Consequently, if \( k <m \), then \( d\nu^{(k)} \not \equiv 0 \), and there is at least one oriented 3-cell \( c \in C_3(B_N)\) for which \( d \nu^{(k)}(c) \neq 0 \). Let \( c_{k+1} \) be the first oriented 3-cell (with respect to the ordering of the 3-cells) for which \( d\nu^{(k)}(c_{k+1}) \neq 0\). Since \( \nu \in \Omega^2_0(B_N,G)\), we have \( d\nu(c_{k+1}) = 0, \) and consequently there must be at least one plaquette \( p\in \support \partial c_{k+1} \smallsetminus  \support \nu^{(k)} \). Let \( p_{k+1}  \) be the first such plaquette (with respect to the ordering of the plaquettes) and define
        \begin{equation*}
            \nu^{(k+1)}(p) \coloneqq \begin{cases}
            \nu(p) &\text{if } p = p_{k+1}, \cr 
            -\nu(p) &\text{if } p = -p_{k+1}, \cr 
            \nu^{(k)}(p) &\text{otherwise.}
            \end{cases}
        \end{equation*}
    Note that if \( \nu^{(1)},\nu^{(2)}, \ldots, \nu^{(k)} \)   satisfy (i) and (ii), then so does \( \nu^{(k+1)} \).

    We now show that \( \nu^{(m)} = \nu \). To this end, note that by (i), \( |\support \nu^{(m)} | = 2m \) and by (ii), \( \nu|_{\support \nu^{(m)}} = \nu^{(m)} \). Since \( |\support \nu| = 2m \), it follows that \( \nu^{(m)} = \nu \).

    We now derive an upper bound for the total number of sequences \( (\nu^{(1)}, \nu^{(2)}, \ldots, \nu^{(m)}) \) which correspond, as above, to some irreducible \( \nu \in \Omega^2_0(B_N,G)  \) with  \( p_0\in \support \nu \) and \( |\support \nu| = 2m \). To this end,  note first that for each \( k \in \{ 1,2, \ldots, m \} \), we have \( \nu(p_k) \in G \smallsetminus \{ 0 \}\). 
    Next, recall that for each  \( k \in \{ 1,2, \ldots, m-1 \} \), we have \( d\nu^{(k)}(c_{k+1}) \neq 0\), and hence \( \support \partial c_{k+1} \cap \support \nu^{(k)} \neq \emptyset \). Since \( |\support \partial c_{k+1}| = 6 \), it follows that  \( | \support \partial c_{k+1} \smallsetminus \support \nu^{(k)} | \leq 5. \) As \(p_{k+1} \in \support \partial c_{k+1} \smallsetminus \support \nu^{(k)} \), there are thus at most five possible ways to choose \( p_{k+1} \). 
    Consequently, the total number sequences \( (\nu^{(1)}, \nu^{(2)}, \ldots, \nu^{(m)}) \) which can correspond to some \( \nu \) as above is at most \( 5^{m-1}(|G|-1)^m \). Since the mapping \( \nu \mapsto (\nu^{(1)}, \nu^{(2)}, \ldots, \nu^{(m)}) \) is injective,  this completes the proof.
\end{proof}

We next use the previous two lemmas to provide a proof of the following result, which extends Corollary 6.2 in~\cite{c2019}. In contrast to the proof of the corresponding result in~\cite{c2019}, we use neither any duality of the model, nor the rate of decay of correlations.   

\begin{proposition}[Compare with Corollary 6.2 in~\cite{c2019}] \label{corollary 6.2}\label{proposition: 6.2 II v2}
    Fix any \( \beta_0 >0 \) such that \( 5(|G|-1) \lambda(\beta)^2< 1 \)
    for all \( \beta > \beta_0 \), where $\lambda(\beta)$ is defined by \eqref{lambdadef}.  Fix \( p_0 \in C_2(B_N)\) and \( M \geq 6 \).
    Then
    \begin{equation*}
        \mu_{N,\beta}\bigl( \{ \sigma \in\Omega^1(B_N,G) \colon \exists \text{ a vortex } \nu \text{ in } \sigma \text{ with } p_0 \in \support \nu \text{ and } |\support \nu | \geq 2M \bigr) \leq K_0^{(M)}  \lambda(\beta)^{2M}
    \end{equation*}
    for all $\beta > \beta_0$, where
    \begin{equation}\label{eq: vortex constant}
    K_0^{(M)}  \coloneqq  
     \frac{5^{M} (|G|-1)^{M} }{1 - 5(|G|-1)\lambda(\beta)^2 } .
\end{equation}
\end{proposition}

\begin{proof}
    Let \( m \geq 1 \) be an integer and let \( p_0 \in C_2(B_N)\). By Lemma~\ref{lemma: counting vortex configurations ii}, there are at most \( \bigl(5(|G|-1)\bigr)^m \) irreducible plaquette configurations \( \nu \in \Omega^2_0(B_N,G) \) with \( p_0 \in \support \nu \) and \( |\support \nu| = 2m \). 
    By Lemma~\ref{lemma: agreement probability}, for any such plaquette configuration \( \nu \in \Omega^2_0(B_N,G) \), we have
	\begin{equation*} 
        \mu_{N,\beta} \bigl( \{ \sigma \in \Omega^1(B_N,G) \colon (d\sigma)|_{\support \nu} = \nu \} \bigr) 
	        \leq \prod_{p \in \support \nu} \frac{\phi_\beta  \bigl(\nu(p)\bigr)}{\phi_\beta(0)} \leq \lambda(\beta)^{2m}. 
	\end{equation*} 
	Summing over all \( m \geq M \), we thus obtain
    \begin{equation*}
        \begin{split}
            &\mu_{N,\beta}\bigl( \{ \sigma \in \Omega^1(B_N,G) \colon \exists \text{ a vortex } \nu \text{ in } \sigma \text{ with } p_0 \in \support \nu \text{ and } |\support \nu | \geq M \bigr)
            \\&\qquad \leq \sum_{m=M}^\infty   \bigl( 5(|G|-1) \bigr)^{m}  \lambda(\beta)^{2m}. 
        \end{split}
    \end{equation*}
    The right-hand side in the previous equation is a geometric sum, which converges if and only if \({5(|G|-1) \lambda(\beta)^2< 1.} \)
    In this case, we obtain the upper bound
    \begin{align*}
        \sum_{m=M}^\infty \bigl( 5(|G|-1) \bigr)^{m}  \lambda(\beta)^{2m} 
        \leq
        \frac{\bigl( 5(|G|-1) \bigr)^{M}  \lambda(\beta)^{2M}}{1 - 5(|G|-1)\lambda(\beta)^2 } .
    \end{align*}
    From this  the desired conclusion  follows.
\end{proof}

 \subsection{Vortices and oriented surfaces}

 We now state and prove the following lemma, which is analogous to Lemma 3.2 in~\cite{c2019}.

\begin{lemma}\label{lemma: 3.2}
    Let \( \sigma \in \Omega^1(B_N,G) \) and let \( \nu \in \Omega^2_0(B_N,G) \) be a vortex in \( \sigma .\)
    Let \( q \) be an oriented surface. 
    If there is a box \( B \) with \( \support \nu \subseteq C_2(B) \) and \( (B^*)^* \subseteq B_N, \) such that \( C_2\bigl((B^*)^*,G\bigr) \cap \support q \) consists of only internal plaquettes of \( q \), then
	\begin{equation}\label{eq: vortex surface equation}
	    \nu(q) = 0.
    \end{equation} 
\end{lemma}

\begin{proof} 

    Since \( \nu \in \Omega^2_0(B_N,G)\), we have \( d\nu = 0 \), and hence, by \eqref{hodgehodge} and Lemma~\ref{lemma: lemma 2.3} (which we can apply because \( (B^*)^* \subseteq B_N \)), we have
    \begin{equation*}
        \delta ({* \nu}) = (-1)^{4(2+1)+1} {*(d({*({*\nu})}))} = (-1)^{4(2+1)+1} {*}d((-1)^{2(4-2)}\nu)
        = (-1)^{4(2+1)+1 + 2(4-2)} {*(d \nu)} = 0.
    \end{equation*} 
    Since the support of \( \nu \) is contained in \( C_2(B) \) by assumption, it follows from Lemma~\ref{lemma: lemma 2.4} that \( *\nu \) has no support outside \( C_2(B^*) \), which by Lemma~\ref{lemma: lemma 2.7} implies that \( *\nu = -\delta g \) for some 3-form \( g \) that is zero outside \( C_3(B^*) \). Utilizing \eqref{hodgehodge} and Lemma~\ref{lemma: lemma 2.3} again, we conclude that $\nu = *\! * \! \nu = d (*  g)$.
    Let 
    \begin{equation*}
        S \coloneqq \nu(q)
    \end{equation*}
    Since \( \support q \) is finite, this sum is well-defined and \( S \in G \).
    We need to show that \( S = 0 \). To that end, let \( \gamma \) be the boundary of \( q \) and note that by Stokes' theorem (see~\eqref{eq: stokes}), we have
    \begin{equation*}
        S = (d ({*}g))(q) = (*g)(\gamma).
    \end{equation*}

    Let us say that a $k$-form \( f_0 \) is \emph{elementary} if there is a $k$-cell \( c \) such that \( \support f_0 = \{ c,-c\} \). Since the 3-form \( g \) has no support outside \( C_3(B^*) \), it has finite support, and hence \( g \) can be written as the sum of finitely many elementary 3-forms that are zero outside \( B^* \). Since the Hodge-star operator is additive, it will follow that \( S = 0 \) if we show that
    \begin{equation}\label{eq: elementary function}
        (*g_0)(\gamma) = 0
    \end{equation}
    for any elementary 3-form \( g_0 \) that is zero outside \( C_3(B^*) \). Take any such \( g_0 \). Then \( *g_0 \) is an elementary 1-form on the primal lattice. 
    Let \( c_0 \) be a 3-cell with \( \support g_0 = \{ c_0,-c_0 \} \), and let \( e_0 = *c_0 \) so that \( \support *g_0 = \{ e_0,-e_0\} \).
    Recall that  \( \support \hat \partial e_0 \) is the set of all plaquettes in \( C_2^+(B_N)\) with \( \partial p[e_0]\neq 0 \), and note that \( \support \hat \partial e_0= *\support (\partial c_0) \). In particular, the support of \( \hat \partial e_0 \) are plaquettes in \( C_2\bigl((B^*)^*\bigr) \).
    Let \( Q_0 =   \support \hat \partial e_0 \cap \pm \support q \). Then \( Q_0 \subseteq C_2\bigl((B^*)^*\bigr) \cap \support q \), and hence, by assumption, all elements of \( Q_0 \) are internal plaquettes of \( q \). 
    By definition, this is equivalent to \( \gamma(e) = 0\) for all \( e \in p \) and \( p \in \support \hat \partial e_0 \cap  \support q \). Consequently, \( \gamma(e_0) = \gamma(-e_0)=0 \). This implies that~\eqref{eq: elementary function} holds, and hence the desired conclusion follows.
\end{proof}

\section{The expected value of a Wilson loop}\label{section: proof of main result}

The purpose of this section is to prove Theorem~\ref{theorem: Chatterjee's main theorem}. 
To this end, recall that we have assumed that \( G = \mathbb{Z}_n \) for some \( n \geq 2 \), and that \( \rho \) is a faithful one-dimensional representation of \( G \).
Our proof closely follows the proof of Theorem 1.1 in~\cite{c2019}.

\subsection{From \texorpdfstring{\( W_\gamma \)}{Wgamma}  to \texorpdfstring{\( W_\gamma' \)}{Wgamma'}}\label{sec:WtoWprime}
In this section, we prove Proposition~\ref{prop: first part of proposition proof}, which shows that the expected value $\mathbb{E}_{N,\beta} [W_\gamma]$ of the Wilson loop observable $W_\gamma$ can be well approximated by the expected value of a simpler observable $W_\gamma'$. 
The proof is based on the same idea as the proof of Theorem 1.1 in~\cite{c2019}, namely, to show that the main contribution to $\mathbb{E}_{N,\beta} [W_\gamma]$ stems from the minimal vortices centered on edges of the loop $\gamma$.

Given a generalized loop \( \gamma \), recalling the definition of \( \gamma_c \) from~\eqref{def: gammac},  we let \( \gamma_1 \coloneqq \gamma- \gamma_c \).
Further, given \( \sigma \in \Omega^1_0(B_N,G) \), we let \( \gamma' \in C_1(B_N,\mathbb{Z}) \) be defined by
\begin{equation*}
    \gamma'[e] \coloneqq \gamma_1[e] \cdot \mathbb{1}\bigl(\exists p,p' \in \hat \partial e\colon  d\sigma(p) \neq  d\sigma(p') \bigr),\quad e \in C_1^+(B_N).
\end{equation*}

\begin{proposition}\label{prop: first part of proposition proof}
    Let \( \beta \geq 0 \) be such that~\( 5(|G|-1) \lambda(\beta)^2< 1 \), and let \( \gamma \) be a generalized loop in \( C_1(B_N) \) such that \( \support \hat \partial \gamma \subseteq C_2(B_N) \). 
    For each \( e \in \gamma \), fix \( p_e \in  \hat \partial e \) and define
    \begin{equation*}
          W_\gamma' \coloneqq    \rho \Bigl( \,\sum_{e \in  \gamma_1 - \gamma'}  d\sigma(p_e)   \Bigr).
    \end{equation*}
    Then
    \begin{equation*} 
        \Bigl|\mathbb{E}_{N,\beta} \bigl[W_\gamma\bigr]- \mathbb{E}_{N,\beta}\bigl[W_\gamma' \bigr]\Bigr|
        \leq 
         2 K_0^{(25)} \ell^4 \lambda(\beta)^{50}
        +
         4K_1 K_0^{(7)} \ell \lambda(\beta)^{14}
         +
          2 C^{(6)} \ell_c \lambda(\beta)^{12} .
    \end{equation*}
\end{proposition}

In the proof of Proposition~\ref{prop: first part of proposition proof},  we use the following additional notation.   
\begin{description}

    \item[\( q \)] An oriented surface such that \( \gamma \) is the boundary of \( q \), and such that the support of \( q \) is contained in a cube of side length \( |\support \gamma|/2 \) which also contains \( \gamma \). Since \( \gamma \) has length \( |\support \gamma| \), such a cube exists, and the existence of such a surface is then guaranteed by Lemma~\ref{lemma: oriented loops}.
    
     \item[\( Q \)] The support of the oriented surface \( q \).
    \item[\( b \)] The smallest number such that any irreducible \( \nu \in \Omega^2_0(B_N,G) \) with \( |\support \nu| \leq 2 \cdot 24 \) is contained in a box of width \( b \). By the definition of irreducible plaquette configurations, \( b \) is a finite universal constant.
    \item[\( Q' \)] The set of plaquettes   \( p \in Q \) that are so far away from \( \gamma \) that any cube of width \( b+2 \) containing \( p \) does not intersect \( \gamma \).
\end{description}
 
We will consider the following ``good'' events:
\begin{description}
    \item[\( \mathcal{A}_1 \)] There is no vortex \( \nu \) in \( \sigma \) with \( |\support \nu | \geq 2 \cdot 25 \) whose support intersects \( Q \).
    \item[\( \mathcal{A}_2 \)] There is no vortex \( \nu \) in \( \sigma \) with \( |\support \nu| \geq 2\cdot 7 \) whose support  intersects \( Q \smallsetminus Q' \).
    \item[\( \mathcal{A}_3\)] There is no minimal vortex \( \nu \) in \( \sigma \) with \( \support \nu = \pm \support \hat\partial e  \) for some \( e \in  \gamma_c \).
\end{description}

Before we give a proof of Proposition~\ref{prop: first part of proposition proof}, we  state and prove a few short lemmas which bound the probabilities of the events \( \neg  \mathcal{A}_1 \), \(\neg \mathcal{A}_2 \), and \( \neg \mathcal{A}_3 \).

\begin{lemma}\label{lemma: A1}
    Let \( \beta \geq 0 \) be such that~\( 5(|G|-1) \lambda(\beta)^2< 1 \), and let \( \gamma \) be a generalized loop in \( C_1(B_N) \). Then
    \begin{equation*}
        \mu_{N,\beta}(\neg \mathcal{A}_1) \leq  K_0^{(25)} |\support \gamma|^4 \lambda(\beta)^{50}.
    \end{equation*}
\end{lemma}
 
\begin{proof}
   By Proposition~\ref{proposition: 6.2 II v2} and a union bound, we have
    \begin{equation*}
        \mu_{N,\beta}(\neg A_1) \leq  K_0^{(25)} |Q| \lambda(\beta)^{50}.
    \end{equation*}
    Since \(Q\) is contained in a cube of side \(|\support \gamma|/2\), we have \( |Q| \leq |\support \gamma|^4\) and the desired conclusion follows.
\end{proof}

\begin{lemma}\label{lemma: A2}
    Let \( \beta \geq 0 \) be such that~\( 5(|G|-1) \lambda(\beta)^2< 1 \), and let \( \gamma \) be a generalized loop in \( C_1(B_N) \). Then
    \begin{equation}\label{eq: 7.4}
        \mu_{N,\beta}(\neg \mathcal{A}_2) \leq K_1 K_0^{(7)} |\support \gamma| \lambda(\beta)^{14}.
    \end{equation}
\end{lemma}
 
\begin{proof}
    Since each plaquette of \( Q \smallsetminus Q' \) is contained in a cube of width \( b+2 \) which intersects \( \gamma \), it follows that 
    \begin{equation}\label{eq: C1}
        | Q \smallsetminus Q' | \leq K_1 |\support \gamma|,
    \end{equation}
    where \( K_1 \) is a universal constant which depends on \( b \).
    Therefore, by Proposition~\ref{proposition: 6.2 II v2} and a union bound,~\eqref{eq: 7.4} follows.
\end{proof}

\begin{lemma}\label{lemma: A3}
    Let \( \beta \geq 0 \) be such that~\( 5(|G|-1) \lambda(\beta)^2< 1 \), and let \( \gamma \) be a generalized loop in \( C_1(B_N) \) such that \( \support \hat \partial \gamma \subseteq C_2(B_N) \). Then
    \begin{equation*}
        \mu_{N,\beta}(\neg \mathcal{A}_3) \leq  C^{6}|\support \gamma_c|  \lambda(\beta)^{12}.
    \end{equation*} 
\end{lemma}

\begin{proof}
    By Proposition~\ref{proposition: 6.2 II v2} and a union bound, the desired conclusion follows.
\end{proof}

\begin{proof}[Proof of Proposition~\ref{prop: first part of proposition proof}] 

    Let \( B \) be a cube of width \( \ell \) which contains \( \gamma \) and is contained in \( B_N \). By Lemma~\ref{lemma: oriented loops}, there is an oriented surface \( q \) with support contained in \( B \) such that \( \gamma \) is the boundary of \( q \).
    Let \( \sigma \sim \mu_{N,\beta}\).
 
    By Lemma~\ref{lemma: sum of vortices}, \( d\sigma \) can be written as a sum of vortices \( \nu_1, \nu_2, \ldots \) in $\sigma$ with disjoint supports.  Fix such a decomposition, and let \( V \) be the set of vortices in this decomposition whose support intersects \( Q  \).
    By Stokes' theorem (see~\eqref{eq: stokes}),
    \begin{equation*}
        W_\gamma  
        = \rho \bigl(  d\sigma(q) \bigr)
        = \rho \Bigl( \sum_{\nu \in V}  \nu(q)  \Bigr).
    \end{equation*}
    
    Let 
    \begin{equation*}
        V_0 \coloneqq \bigl\{ \nu \in V \colon |\support \nu | \leq 2 \cdot 24 \bigr\},
    \end{equation*}
    and define
    \begin{equation*}
        W_\gamma^0 \coloneqq   \rho  \Bigl( \sum_{\nu \in V_0} \nu(q)\Bigr).
    \end{equation*}
    If the event \( \mathcal{A}_1 \) occurs, then \( W_\gamma = W_\gamma^0 \). 
    Since $|\rho(g)|= 1$ for all $g \in G$, we always have $| W_\gamma - W_\gamma^0| \leq 2$. Consequently, by Lemma~\ref{lemma: A1}, we have
    \begin{equation}\label{eq: comb eq 1}
        \mathbb{E}_{N,\beta} \bigl[ | W_\gamma - W_\gamma^0 | \bigr] \leq 2 \bigl(1 - \mu_{N,\beta}(A_1)\bigr) \leq 2 K_0^{(25)} \ell^4 \lambda(\beta)^{50}.
    \end{equation}

    Let 
    \begin{equation*}
        V_1 \coloneqq \{ \nu \in V_0 \colon \support \nu \cap Q' \neq \emptyset \} .
    \end{equation*}
    Take any \( \nu \in V_1 \). Then, by the definitions of \( Q' \) and \( V_1 \), it follows that any cube \( B \) of width \( b \) that contains \( \support \nu \) has the property that \( C_2\bigl((B^*)^*\bigr) \cap Q\) only contains internal plaquettes of \( q \). Therefore, by Lemma~\ref{lemma: 3.2},
    \begin{equation*}
	     \nu(q)= 0,\quad \nu \in V_1.
    \end{equation*}
    This implies that if we let 
    \begin{equation*}
         V_2 \coloneqq V_0 \smallsetminus V_1,
    \end{equation*} 
    then
    \begin{equation}\label{eq: 7.3}
        W_\gamma^0 =   \rho  \Bigl( \,  \sum_{\nu \in V_2} \nu(q)\Bigr).
    \end{equation}

    By Lemma~\ref{lemma: minimal vortex I}, if \( \nu \in \Omega^2_0(B_N,G) \) is a vortex in $\sigma$ whose support does not contain any boundary plaquette of $C_2(B_N)$, then \( |\support \nu| \geq 2 \cdot 6 \), and if \( |\support \nu| = 2 \cdot 6 \), then \( \support \nu = \pm \support \hat \partial e  \) for some edge \( e \in C_1(B_N)\). Recall that if a vortex \( \nu \) in \( \sigma \) satisfies \( |\support \nu| = 2 \cdot 6\), then it is said to be a minimal vortex. 
    With this in mind, let 
    \begin{equation*}
        V_3 \coloneqq \bigl\{ \nu \in V_2 \colon |\support \nu| = 2 \cdot 6 \bigr\} = \bigl\{ \nu \in V \colon |\support \nu| = 2 \cdot 6 \text{ and } \support \nu \cap Q' = \emptyset \bigr\},
    \end{equation*}
    and define 
    \begin{equation*}
        W_\gamma^3 \coloneqq   \rho  \Bigl(  \sum_{\nu \in V_3} \nu(q)  \Bigr) .
    \end{equation*}
    If the event \( \mathcal{A}_2 \) occurs, then \( V_3 = V_2 \), and hence \( W_\gamma^0 = W_\gamma^3 \). 
    Consequently, by~Lemma~\ref{lemma: A2},
    \begin{equation}\label{eq: comb eq 2}
        \mathbb{E}_{N,\beta}\bigl[ |W_\gamma^0-W_\gamma^3|\bigr] \leq 2\bigl(1 - \mu_{N,\beta}(A_2)\bigr)  \leq 2K_1 K_0^{(7)} \ell \lambda(\beta)^{14}.
    \end{equation}

    Next, let 
    \begin{equation*}
        V_4 \coloneqq \bigl\{ \nu \in V_3 \colon \exists e \in  \gamma \text{ such that } \support \nu = \pm \support \hat \partial e  \bigr\}
    \end{equation*}
    and define 
    \begin{equation*}
        W_\gamma^4 \coloneqq  \rho  \Bigl( \, \sum_{\nu \in V_4} \nu(q) \Bigr).
    \end{equation*} 
    If \(\nu \in V_3 \smallsetminus V_4\), then \(\nu\) satisfies \( \support \nu = \pm \support \hat \partial e \) for some edge \( e \) which is an internal edge of \(q.\)
    If \( e \) in an internal edge of \( q \) and \( \nu \in V_3 \) is such that  \( \support \nu = \pm \support \hat \partial e \), then, by Lemma~\ref{lemma: 3.2},
    \begin{equation*}
        \nu(q)=0.
    \end{equation*}
    It follows that
    \begin{equation}\label{eq: comb eq 3}
        W_\gamma^3 =  \rho  \Bigl( \sum_{\nu \in V_3}  \nu(q) \Bigr) 
        =   \rho   \Bigl(  \sum_{\nu \in V_4} \nu(q)   \Bigr)= W_\gamma^4.
    \end{equation} 
    
    Now recall the definition of \( \gamma_c \) from~\eqref{def: gammac}.
    Let
    \begin{equation*}
        V_5 \coloneqq \bigl\{ \nu \in V_4 \colon \exists e \in  (\gamma - \gamma_c) \text{ such that } \support \nu = \pm \support \hat \partial e  \bigr\}
    \end{equation*}
    and define
    \begin{equation*}
        W_\gamma^5 \coloneqq  \rho  \Bigl( \sum_{\nu \in V_5} \nu(q) \Bigr)  .
    \end{equation*}
    If \( \mathcal{A}_3 \) occurs, then \( W_\gamma^4 = W_\gamma^5 \), and hence, by Lemma~\ref{lemma: A3}, we have
    \begin{equation}\label{eq: comb eq 4}
        \mathbb{E}_{N,\beta}\bigl[|W_\gamma^4 - W_\gamma^5|\bigr] \leq 2\bigl(1 - \mu_{N,\beta}(A_3)\bigr)
        \leq 2 C^{(6)} \ell_c \lambda(\beta)^{12}.
    \end{equation} 
    If \( \nu \in V_5 \), then \( \support \nu = \pm \support \hat \partial e  \) for some edge \( e \in  (\gamma- \gamma_c) \). 
    Define \( \gamma_5 \in C_1(B_N,\mathbb{Z}) \) by
    \begin{equation*}
        E_5 \coloneqq \bigl\{ e \in  \gamma-\gamma_c \colon \exists \nu \in V_5 \text{ such that } \support \nu = \pm \support \hat \partial e \bigr\}.
    \end{equation*}
    By the definition of \( V_5 \), we then have
    \begin{equation}\label{eq: 7.8}
         W_\gamma^5 =   \rho  \Bigl( \sum_{\nu \in V_5} \nu(q)  \Bigr)  
         =   \rho  \Bigl( \,  \sum_{e \in E_5} \sum_{p \in  \hat \partial e } q[p] d\sigma(p)   \Bigr).
    \end{equation} 
    
    For distinct edges \( e,e' \in  (\gamma - \gamma_c) \), the 2-chains \( \hat \partial e \) and \( \hat \partial e' \) have disjoint supports.
    Moreover, if \( \nu \in V_5 \) has support \( \pm \support \hat \partial e  \), then is follows from Lemma~\ref{lemma: minimal vortex II} that the value of \( d\sigma(p) \) is independent of the choice of \( p \in  \hat \partial e \).
    With this in mind, let \( E_6 \) denote the set of all  edges \( e \in  (\gamma - \gamma_c) \) which are such that 
    \begin{equation*}
        d\sigma(p)
        =
        d\sigma(p') \quad \text{for all } p,p' \in  \hat \partial e,
    \end{equation*}
    and define
    \begin{equation*}
          W_\gamma^6 \coloneqq     
           \rho \Bigl( \,  \sum_{e \in E_6} \sum_{p \in  \hat \partial e  } q[p] d\sigma(p)  \Bigr).
    \end{equation*}
    For each \( e \in E_6 \), recall that we have fixed a plaquette \( p_e \in  \hat \partial e \). Since \( q \) is an oriented surface with boundary \( \gamma \), 
    for any \( e \in E_6 \) we then have
    \begin{equation*}
        \begin{split}
            &
            \sum_{p \in \hat \partial e} q[p] d\sigma(p)  
            = \sum_{p \in  \hat \partial e} q[p]  d\sigma(p_e)
            =
            d\sigma(p_e),
        \end{split}
    \end{equation*}  
    where the last equality holds since \( e \in E_6 \) implies that \( e \in \gamma \), and, by definition,
    \begin{equation*}
        \sum_{p \in \hat \partial e} q[p] 
        =
        \partial q[e] = \gamma[e] = 1.
    \end{equation*}
    Thus, we can write
    \begin{equation*}
          W_\gamma^6 = \rho \Bigl( \,  \sum_{e \in E_6} d\sigma(p_e)  \Bigr).
    \end{equation*}
    On the event \( \mathcal{A}_2 \), we have \( E_5 = E_6 \), and hence on this event 
    \begin{equation*}
          W_\gamma^6   = W_\gamma^5. 
    \end{equation*}
    Again using~Lemma~\ref{lemma: A2}, we thus have
    \begin{equation}\label{eq: 7.9}
        \mathbb{E}_{N,\beta} \bigl[|W_\gamma^5-W_\gamma^6| \bigr] \leq 2\bigl(1 - \mu_{N,\beta}(\mathcal{A}_2)\bigr) \leq 2K_1K_0^{(7)} \ell \lambda(\beta)^{14}.
    \end{equation}
    Now recall that \( \gamma_1 = \gamma - \gamma_c \) is the restriction of \( \gamma \) to its non-corner edges, and that \( \gamma' \) is the restriction of \( \gamma_1 \) to the set of edges \( e \in  \gamma_1 \) such that there exist plaquettes \( p,p' \in  \hat \partial e  \) with \(  d\sigma(p) \neq d\sigma(p'). \) Consequently, we have \( E_6 = \support (\gamma_1 - \gamma') \).  
    Combining~\eqref{eq: comb eq 1},~\eqref{eq: comb eq 2},~\eqref{eq: comb eq 3},~\eqref{eq: comb eq 4}, and~\eqref{eq: 7.9}, and using the definition of \( W_\gamma' \), we thus obtain the desired conclusion.
\end{proof}

\subsection{A resampling trick}

Recall from Section~\ref{sec:WtoWprime}, that given a generalized loop \( \gamma \), we let \( \gamma_c \) denote the restriction of \( \gamma \) to its set of corner edges,  let \( \gamma_1 = \gamma- \gamma_c \), and let \( \gamma' \in C_1(B_N,\mathbb{Z}) \) be given by
\begin{equation*}
    \gamma'[e] \coloneqq \gamma_1[e] \cdot \mathbb{1}\bigl(\exists p,p' \in  \hat \partial e\colon  d\sigma(p) \neq  d\sigma(p') \bigr),\quad e \in C_1^+(B_N).
\end{equation*}
Furthermore, for each \( e \in \gamma\) we have fixed some (arbitrary) \( p_e \in  \hat \partial e \), and defined
\begin{equation*}
    W_\gamma' =    \rho \Bigl( \sum_{e \in  \gamma_1 - \gamma'} d\sigma(p_e)  \Bigr). 
\end{equation*} 
In this section, we use a resampling trick, first introduced (in a different setting) in~\cite{c2019}, to calculate \( \mathbb{E}_{N,\beta}[W_\gamma'] \).
The use of this trick is the main reason that the corner edges of the loop \( \gamma \) become relevant in the argument, these are namely exactly the edges where this trick cannot be applied.

\begin{proposition}\label{prop: resampling in main proof}
    Let \( \beta \geq 0 \), and let \( \gamma \) be a generalized loop in \( C_1(B_N) \) such that \( \support \partial \hat \partial \gamma \subseteq C_1(B_N)\). Then 
    \begin{equation*}
        \mathbb{E}_{N,\beta}\bigl[W_{\gamma  }'  \bigr] 
        =  \theta(\beta)^{|\support \gamma_1|} \mathbb{E}_{N,\beta}\bigl[\theta(\beta)^{-|\support \gamma'|}\bigr].
    \end{equation*}   
\end{proposition}

\begin{proof}
    Note first that for any  \( e \in C_1(B_N) \) and \( p \in  \hat \partial e \), we have
    \begin{equation*}
        d\sigma(p)=\sigma(\partial p) = \sum_{e' \in  \partial p} \sigma(e') = \sigma(e) + \sum_{e' \in  \partial p,\, e' \neq e} \sigma(e').
    \end{equation*}
    Thus, whether or not \( e \in  \gamma_1 - \gamma' \) does not depend on \( \sigma(e) \). 
    Consequently,  if the spins of all edges which are not in \( \pm \support \gamma_1  \) are known, then \( \gamma' \) is determined.

    Let \( \mu_{N,\beta}' \) denote conditional probability and \( \mathbb{E}_{N,\beta}' \) denote conditional expectation given \( \bigl( \sigma(e) \bigr)_{e \not \in \pm \gamma_1} \). Since no two non-corner edges belong to the same plaquette, it can be verified by direct computation that under this conditioning,  \( \bigl( \sigma(e)\bigr)_{e \in \support \gamma_1} \) are independent spins. Moreover, by the argument above, conditioning on the spins outside \( \pm \support \gamma_1 \) determines \( \gamma' \). %
    Thus we find
    \begin{equation*}
        \begin{split}
            &\mathbb{E}_{N,\beta}'\bigl[W_\gamma' \bigr]
            =
            \mathbb{E}_{N,\beta}' \biggl[    \rho \Bigl( \sum_{e \in  \gamma_1 - \gamma'}  d\sigma(p_e)   \Bigr)\biggr]
            =
            \mathbb{E}_{N,\beta}' \biggl[  \prod_{e \in  \gamma_1 - \gamma'} \rho  \bigl(   d\sigma(p_e)  \bigr) \biggr]
            = \prod_{e \in \gamma_1 - \gamma'} \mathbb{E}_{N,\beta}' \Bigl[ \rho \bigl(  d\sigma(p_e)  \bigr) \Bigr]. 
        \end{split}
    \end{equation*} 
    For \(  e \in  \gamma_1 - \gamma' \), let 
    \begin{equation*}
         \sigma^{ e} \coloneqq \sum_{ e' \in  \partial p_{ e} ,\, e' \neq e}  \sigma(e'),
    \end{equation*}
    so that for each \( p \in \hat{\partial} e \), we have
    \begin{equation*}
        d\sigma(p) = \sigma^e + \sigma(e).
    \end{equation*} 
    Then, for each \( e \in \gamma_1 - \gamma' \), we have
    \begin{equation*}
        \mathbb{E}_{N,\beta}' \Bigl[ \rho \bigl(  d\sigma(p_e)  \bigr) \Bigr]
        =
        \frac{ \sum_{g \in  G} \rho  \bigl( \sigma^e + g \bigr)\, \phi_\beta(\sigma^e + g)^{2 \cdot 6}}{\sum_{g \in  G}   \phi_\beta(\sigma^e + g)^{2 \cdot 6}} 
        =
        \frac{\sum_{g \in  G} \rho  \bigl( g \bigr) \phi_\beta (  g)^{2 \cdot 6}}{\sum_{g \in G}   \phi_\beta(  g)^{2 \cdot 6}}  
        =  \theta(\beta),
    \end{equation*}
     where the last identity holds by definition.
    Combining the above equations, we obtain
     \begin{equation*}
        \begin{split}   
            &\mathbb{E}_{N,\beta}\bigl[W_\gamma' \bigr]
            =
            \mathbb{E}_{N,\beta} \Bigl[\mathbb{E}_{N,\beta}'\bigl[W_\gamma' \bigr]\Bigr]
            = \mathbb{E}_{N,\beta} \bigl[ \theta(\beta)^{|\gamma_1 - \gamma'|} \bigr] = \mathbb{E}_{N,\beta}\bigl[  \theta(\beta)^{|\gamma_1 - \gamma'|} \bigr]
            \\&\qquad 
            =   \theta(\beta)^{|\support \gamma_1|} \mathbb{E}_{N,\beta}\bigl[\theta(\beta)^{-|\support \gamma'|}\bigr]. 
        \end{split}
     \end{equation*}  
     as desired. This concludes the proof.
\end{proof}

\subsection{Properties of \texorpdfstring{$\theta(\beta)$}{theta}}

The main result in this section is the following proposition.
\begin{proposition}\label{proposition: rewriting theta}
    Let \( \beta \geq 0 \) be such that~\( 5(|G|-1) \lambda(\beta)^2< 1 \), and let \( \gamma \) be a generalized loop in \( C_1(B_N) \) such that \( \support\hat \partial \gamma \subseteq C_2(B_N) \). Then
    \begin{equation}\label{eq: 7.15}
        \begin{split}
            &\Bigl|\theta^{|\support \gamma_1|} \mathbb{E}_{N,\beta}[\theta^{-|\support \gamma'|}] -  \theta^\ell \Bigr|
            \leq 
            2
            \Bigl[
             \frac{  K_0^{(6)}+ K^*   }{  \sqrt{\ell}} 
            +  \frac{  K^* \ell_c }{\ell} \Bigr]  \ell \lambda(\beta)^{12} e^{K^* \ell \lambda(\beta)^{12}} +  
            K_0^{(6)} \sqrt{\ell} \lambda(\beta)^{12} .
        \end{split}
     \end{equation}
\end{proposition} 

Before we give a proof of Proposition~\ref{proposition: rewriting theta} we will state and prove two useful lemmas.
\begin{lemma}\label{lemma: theta bounded from above}
    Let \( \beta \geq 0 \). Then \( \theta(\beta) \leq 1 \).
\end{lemma}

\begin{proof}
    Recall from~\eqref{Thetadef} that
    \begin{equation*}
        \theta(\beta) = \frac{\sum_{g \in G}\rho(g) e^{12\beta \Re\rho(g)}}{\sum_{g \in G} e^{12\beta \Re \rho(g)}}.
    \end{equation*}
    Since \( \rho \) is a representation of \( G \), we have, for any \( g \in G \),  
    \begin{equation*}
        \begin{split}
            &\phi_\beta(-g ) = e^{\beta \Re (\rho(-g))} 
            = e^{\beta \Re ( \rho(g)^{-1})}
             = e^{\beta \Re ( \rho(g) )}
            = \phi_\beta(g).
        \end{split}
    \end{equation*}
    Also, we have \( |\rho(g)| = 1 \) for all \( g \in G \). Combining these observations, the desired conclusion follows.
\end{proof}

\begin{lemma}\label{lemma: theta observations}
    Let \( \beta > \beta_0 > 0 \), and define 
    \begin{equation}\label{eq: C2}
        \begin{split}
            &K^* \coloneqq  \sup_{\beta > \beta_0} \, \Bigl[  \bigl(1-\theta(\beta) \bigr)  \lambda(\beta)^{-12} \Bigr].
        \end{split}
     \end{equation}
     Then
     \begin{enumerate}[label=(\alph*)]
         \item for any \( 0 \leq j \leq \ell \), we have \( \theta^{-j}  \leq 2e^{K^* \ell \lambda(\beta)^{12}}, \) \label{item 7.12}
         
         \item for any \( 1 \leq j \leq \ell \), we have \( \theta^{-j}-1  \leq 2 j K^*  \lambda(\beta)^{12} e^{K^* \ell \lambda(\beta)^{12}}. \) \label{item: 7.13}
     \end{enumerate}   
\end{lemma}

\begin{proof}
    By definition, using Lemma~\ref{lemma: theta bounded from above}, we have
     \begin{equation}\label{eq: 7.11}
         1 - K^* \lambda(\beta)^{12} \leq \theta(\beta) \leq 1.
     \end{equation}
     Moreover, using the definition of \( \theta(\beta) \) (see~\eqref{Thetadef}), it follows that  
     \begin{equation*} 
        \begin{split}
            &K^* 
            =
            \sup_{\beta > \beta_0} \, \biggl[ \frac{\sum_{g \in  G} \bigl(1 -  \rho(g)\bigr) \phi_\beta(  g)^{12}}{\sum_{g \in  G}   \phi_\beta( g)^{12}}  \lambda(\beta)^{-12} \biggr] 
            =
            \sup_{\beta > \beta_0} \, \biggl[ \frac{\sum_{g \in  G\smallsetminus \{ 0 \}} \bigl(1 -  \Re \rho(g)\bigr) \phi_\beta(  g)^{12}}{\sum_{g \in  G}   \phi_\beta( g)^{12}}  \lambda(\beta)^{-12} \biggr].
        \end{split}
     \end{equation*}
     For any \( g \in G \smallsetminus \{ 0 \} \), we have \( 1- \Re \rho(G) \in (0,2]\), and hence it follows from the definition of \( \lambda(\beta) \) (see~\eqref{lambdadef}) that 
     \begin{equation*}
         0<K^* \leq 2\bigl(|G|-1 \bigr).
     \end{equation*}
     Since \( 5(|G|-1) \lambda(\beta)^2 < 1 \) by assumption, it follows that
     \begin{equation}\label{eq: 7.11iii}
          K^* \lambda(\beta)^{12} 
          \leq
          2\bigl( |G| -1 \bigr) \lambda(\beta)^{12}
          =
          \frac{2}{5^6\bigl( |G| -1 \bigr)^5} \Bigl( 5 \bigl( |G|-1 \bigr) \lambda(\beta)^{2} \Bigr)^6
          <
          \frac{2}{5^6\bigl( |G| -1 \bigr)^5}.
     \end{equation}
     Combining~\eqref{eq: 7.11iii} with~\eqref{eq: 7.11}, we obtain
     \begin{equation}\label{eq: 7.11ii}
          0
          < 
          1 - \frac{2}{5^6\bigl( |G| -1 \bigr)^5}
          <
          1 - K^* \lambda(\beta)^{12} 
          \leq \theta(\beta) \leq 1.
     \end{equation}
    Using~\eqref{eq: 7.11ii}, we conclude that, for any \( 0 \leq j \leq \ell \),
    \begin{equation*}
        \theta^{-j}  \leq \theta^{-\ell}  \leq (1 - K^*\lambda(\beta)^{12})^{-\ell} \leq 2e^{K^* \ell \lambda(\beta)^{12}},
    \end{equation*}
    and hence~\ref{item 7.12} holds.
    (Here the third inequality holds since \( (1-x)^{-1} \leq 2e^x \) whenever \( x \in [0,1/2]\), and in our case we have \( x = K^* \lambda(\beta)^{12} < 2/5^6\).)
    Consequently, for any \( 1 \leq j \leq \ell \),
    \begin{equation*}
        \theta^{-j}-1 \leq j(1-\theta) \theta^{-j} \leq 2 j K^*  \lambda(\beta)^{12} e^{K^* \ell \lambda(\beta)^{12}}.
    \end{equation*}     
\end{proof}

\begin{proof}[Proof of Proposition~\ref{proposition: rewriting theta}]
    Note first that
    \begin{equation*}
        \begin{split}
            &\bigl|  \theta^{|\support \gamma_1|} \mathbb{E}_{N,\beta}[\theta^{-|\support \gamma'|}] -  \theta^\ell \bigr| 
            =  \Bigl| \theta^{|\support \gamma_1|} \bigl( \mathbb{E}_{N,\beta}[\theta^{-|\support \gamma'|}] -1 \bigr)-  \theta^{|\support \gamma_1|} \bigl( \theta^{\ell_c}-1\bigr) \Bigr|
            \\&\qquad
            \leq  
              \bigl|\theta^{|\support \gamma_1|} \bigl(\mathbb{E}_{N,\beta}[\theta^{-|\support \gamma'|}]-1\bigr) \bigr| + \theta^{|\support \gamma_1|}(\theta^{-\ell_c} - 1).
        \end{split}
    \end{equation*}

    Using~Lemma~\ref{lemma: theta observations}\ref{item: 7.13} and the inequality \( \theta(\beta) \leq 1 \) from Lemma~\ref{lemma: theta bounded from above}, we obtain 
     \begin{equation}\label{eq: 7.14}
         \bigl|\theta^{|\support \gamma_1|} \mathbb{E}_{N,\beta}[\theta^{-|\support \gamma'|}] -   \theta^\ell \bigr|  \leq  \bigl| \mathbb{E}_{N,\beta}[\theta^{-|\support \gamma'|}]-1\bigr| +    2K^* \ell_c  \lambda(\beta)^{12} e^{K^* \ell  \lambda(\beta)^{12}}.
     \end{equation} 
    
    Now recall that any plaquette \( p \in C_2(B_N)\) such that \(  d\sigma(p) \neq 0 \) is contained in the support of a vortex, and that any vortex whose support does not contain any boundary plaquettes of \( B_N \) has support of size at least \(12\). Therefore, by Corollary~\ref{proposition: 6.2 II v2}, the probability that any given plaquette \( p \) is such that \( d \sigma(p) \neq 0 \) is bounded from above by \( K_0^{(6)} \lambda(\beta)^{12} \).
    Consequently,
     \begin{equation}\label{eq: expectation of gamma}
         \mathbb{E}_{N,\beta} \bigl[ |\support \gamma'| \bigr] \leq K_0^{(6)} \ell \lambda(\beta)^{12}.
     \end{equation}
     Thus, for any \( j > 0 \),
     \begin{equation*}
         \begin{split}
            &\bigl|\mathbb{E}_{N,\beta} [\theta^{-|\support \gamma'|}] - \mathbb{E}_{N,\beta} [\theta^{-|\support \gamma'|} \, 1_{\{|\support \gamma'|\leq j\}}] \bigr| 
            = \mathbb{E}_{N,\beta} [\theta^{-|\support \gamma'|} \, \mathbb{1}_{\{|\support \gamma'|> j\}}]
             \\&\qquad \leq \theta^{-\ell} \mu_{N,\beta} \bigl( \{ \sigma \in \Omega^1(B_N,G) \colon |\support \gamma'|>j\} \bigr)
         \leq
         \frac{\theta^{-\ell} \mathbb{E}_{N,\beta} \bigl[|\support \gamma'|\bigr]}{j}
         \leq
         \frac{\theta^{-\ell}  K_0^{(6)} \ell \lambda(\beta)^{12} }{j} .
         \end{split}
    \end{equation*}
    By~Lemma~\ref{lemma: theta observations}\ref{item 7.12}, we thus have
    \begin{equation}\label{eq: first part}
         \begin{split}
            &\bigl|\mathbb{E}_{N,\beta} [\theta^{-|\support \gamma'|}] - \mathbb{E}_{N,\beta} [\theta^{-|\support \gamma'|} \, 1_{\{|\support \gamma'|\leq j\}}] \bigr| 
            \leq
            \frac{  K_0^{(6)}   \ell \lambda(\beta)^{12} e^{K^* \ell \lambda(\beta)^{12}}}{j}.
         \end{split}
    \end{equation}
    On the other hand, we have 
    \begin{equation*}
    \begin{split}
        &\Bigl| \mathbb{E}_{N,\beta} \bigl[\theta^{-|\support \gamma'|} \, \mathbb{1}_{\{|\support \gamma'|\leq j\}} \bigr]-1 \Bigr| 
        \leq 
        \Bigl| \mathbb{E}_{N,\beta} \bigl[ (\theta^{-|\support \gamma'|}-1) \, \mathbb{1}_{\{|\support \gamma'|\leq j\}} \bigr] \Bigr| + 
        \mathbb{E}_{N,\beta} \bigl[  \mathbb{1}_{\{|\support \gamma'|\leq j\}} \bigr] 
        \\&\qquad \leq 
        (\theta^{-j}-1)
        + 
        \frac{\mathbb{E}_{N,\beta} \bigl[ |\support \gamma'| \bigr] }{j}.
    \end{split}
    \end{equation*}
    Using~\eqref{eq: expectation of gamma} and Lemma~\ref{lemma: theta observations}\ref{item: 7.13}, we obtain
    \begin{equation}\label{eq: second part}
    \begin{split}
        &\Bigl| \mathbb{E}_{N,\beta} \bigl[\theta^{-|\support \gamma'|} \, \mathbb{1}_{\{|\support \gamma'|\leq j\}} \bigr]-1 \Bigr| 
        \leq 
        2K^* j \lambda(\beta)^{12} e^{K^* \ell \lambda(\beta)^{12}}
        + 
        \frac{K_0^{(6)} \ell \lambda(\beta)^{12}}{j}.
    \end{split}
    \end{equation}

     Combining~\eqref{eq: first part} and~\eqref{eq: second part}, using the triangle inequality and choosing \( j = \sqrt{\ell} \), we obtain
     \begin{equation*}
         \bigl|\mathbb{E}_{N,\beta} [\theta^{-|\support \gamma'|}]-1\bigr|
         \leq
           2(K_0^{(6)}+ K^*) \,   \sqrt \ell \lambda(\beta)^{12} e^{K^* \ell \lambda(\beta)^{12}} + 
        K_0^{(6)} \sqrt{\ell} \lambda(\beta)^{12} .
     \end{equation*}
     Combining this equation with~\eqref{eq: 7.14}, we infer that
     \begin{equation*} 
        \begin{split}
            &\bigl|\theta^{|\support \gamma_1|} \mathbb{E}_{N,\beta}[\theta^{-|\support \gamma'|}] -  \theta^\ell \bigr|
            \leq 
            2
            \Bigl[
             \frac{  K_0^{(6)}+ K^*   }{  \sqrt{\ell}} 
            +  \frac{  K^* \ell_c }{\ell} \Bigr]  \ell \lambda(\beta)^{12} e^{K^* \ell \lambda(\beta)^{12}} +  
            K_0^{(6)} \sqrt{\ell} \lambda(\beta)^{12} ,
        \end{split}
     \end{equation*}
     which is the desired conclusion.
\end{proof}

\subsection{\texorpdfstring{Wilson loop expectations when \(\ell \lambda(\beta)^{12} \) is small}{Short loops}}
 The main part of the proof of Theorem~\ref{theorem: Chatterjee's main theorem} is divided into two propositions, Proposition~\ref{proposition: 7.1} and Proposition~\ref{proposition: 7.2}, the first of which we state now. This result is useful whenever \( \ell \lambda(\beta)^{12} \) is bounded from above.

\begin{proposition}[Compare with Lemma 7.1 in \cite{c2019}]\label{proposition: 7.1} 
    Consider lattice gauge theory with structure group $G = \mathbb{Z}_n$ and a faithful one-dimensional representation \( \rho \) of \( G \). Let \( \gamma \) be a generalized loop in \( \mathbb{Z}^4 \), let $\ell = |\support \gamma|$, and let \( \ell_c \) be the number of corner edges in \( \gamma \).
    Let \( N \) be sufficiently large so that \( \support \gamma \subseteq C_1(B_N)\) and so that there is a cube \( B \) of width \( |\support \gamma| \) inside \( B_N \) which contains \( \gamma \).
    Finally, let \( \beta_0 >0 \) be such that \( 5(|G|-1) \lambda(\beta)^2< 1 \) for all \( \beta > \beta_0 \). Then, for all \( \beta > \beta_0 \), we have 
    \begin{equation*}
        \bigl|\mathbb{E}_{N,\beta} [W_\gamma]-   e^{-\ell (1-\theta(\beta)) } \bigr| 
        \leq 
        C_A  e^{2K^* \ell \lambda(\beta)^{12}} 
        \biggl[ 
            \sqrt\frac{\ell_c}{\ell} + \lambda(\beta)^2 
        \biggr],
        \end{equation*}
        where $\theta(\beta)$ is defined by \eqref{Thetadef},
        \begin{align}
            C_A \coloneqq   
            \frac{ 7K_0^{(6)}}{2K^*} +   \frac{2K_1K_0^{(7)} }{K^*} 
            + \frac{5K_0^{(25)} }{2^3(K^*)^4}+   \frac{9}{2} ,
    \end{align}
    and \( K_1 \), \( K^*\), \( K_0^{(6)} \), \( K_0^{(7)} \), and \( K_0^{(25)}\) are defined by~\eqref{eq: C1},~\eqref{eq: C2}, and~\eqref{eq: vortex constant} and depend only on \( \beta_0 \), \( G \) and \( \rho \).
\end{proposition}

\begin{proof} 
    By combining Proposition~\ref{prop: first part of proposition proof}, Proposition~\ref{prop: resampling in main proof}, and~Proposition~\ref{proposition: rewriting theta}, we obtain
    \begin{equation}\label{eq: error combination I}
        \begin{split}          
            &\bigl|\mathbb{E}_{N,\beta}[W_\gamma]-  \theta^\ell \bigr| 
            \\&\qquad\leq  \bigl|\mathbb{E}_{N,\beta}[W_\gamma]-  \mathbb{E}_{N,\beta}[W_\gamma']\bigr|
            +
            \Bigl|\mathbb{E}_{N,\beta}\bigl[W_{\gamma  }'  \bigr] 
            -  \theta(\beta)^{|\support \gamma_1|} \mathbb{E}_{N,\beta}\bigl[\theta(\beta)^{-|\support \gamma'|}\bigr] \Bigr|
            \\&\qquad\qquad+
            \Bigl|\theta^{|\support \gamma_1|} \mathbb{E}_{N,\beta}[\theta^{-|\support \gamma'|}] -  \theta^\ell \Bigr|
            \\&\qquad\leq
            2\Bigl[ \frac{  K_0^{(6)}+ K^*   }{  \sqrt{\ell}} +  \frac{  K^* \ell_c }{\ell} \Bigr] \ell \lambda(\beta)^{12} e^{K^* \ell \lambda(\beta)^{12}} 
            + K_0^{(6)} \sqrt{\ell} \lambda(\beta)^{12} 
            + \Bigl[  2K_1K_0^{(7)} \ell \lambda(\beta)^{14} \Bigr] 
            \\&\qquad\qquad
            + \Bigl[ 2K_0^{(6)}\ell_c \lambda(\beta)^{12} \Bigr]
            + \Bigl[ 0 \Bigr]
            + \Bigl[ 2K_1 K_0^{(7)} \ell \lambda(\beta)^{14} \Bigr]  
            + \Bigl[ 2K_0^{(25)} \ell^4 \lambda(\beta)^{50}\Bigr]
            \\&\qquad=
            2\Bigl[ \frac{  K_0^{(6)}  + K^*  }{ K^* \sqrt{\ell}}   
            +  \frac{   \ell_c }{\ell} \Bigr] K^* \ell \lambda(\beta)^{12} e^{K^* \ell \lambda(\beta)^{12}} 
            \\&\qquad\qquad
            + \Bigl\{ \frac{K_0^{(6)}}{2K^* \sqrt{\ell}}
            + \frac{2K_1K_0^{(7)}\lambda(\beta)^2}{K^*} + \frac{K_0^{(6)}\ell_c}{K^* \ell}  \Bigr\} \Bigl[  2K^* \ell \lambda(\beta)^{12} \Bigr]
            \\&\qquad\qquad + \frac{K_0^{(25)} \lambda(\beta)^2}{2^3 (K^*)^4 }\Bigl[ \bigl( 2K^*\ell \lambda(\beta)^{12} \bigr)^4 \Bigr]  .
        \end{split}
    \end{equation}
    Using that \( x \leq e^x \) and \( x^4 \leq 5 e^x \) for \( x>0 \) as well as the fact that $ 1 \leq \ell_c \leq \ell$, it follows that
    \begin{equation}\label{eq: combined bounds}
        \begin{split}
            &\bigl|\mathbb{E}_{N,\beta}[W_\gamma]-  \theta^\ell\bigr| 
            \leq  
            \biggl\{
            \frac{  5K_0^{(6)}+4K^* }{2K^*\sqrt{\ell}} 
            +   \frac{2K_1K_0^{(7)}\lambda(\beta)^2}{K^*} + \frac{K_0^{(6)}\ell_c}{K^* \ell} + \frac{2\ell_c}{\ell} + \frac{5K_0^{(25)} \lambda(\beta)^2}{2^3 (K^*)^4}  
            \biggr\}  e^{2K^* \ell \lambda(\beta)^{12}}
            \\& \qquad \leq 
            \biggl\{
            \frac{  5K_0^{(6)}+4K^* }{2K^*} 
            +   \frac{2K_1K_0^{(7)}}{K^*} + \frac{K_0^{(6)} }{K^*} + 2 + \frac{5K_0^{(25)} }{2^3 (K^*)^4}  
            \biggr\} 
            \biggl[ \sqrt\frac{\ell_c}{\ell} + \lambda(\beta)^2 \biggr]  e^{2K^* \ell \lambda(\beta)^{12}}.
        \end{split}
    \end{equation}
    Finally, using the inequality \( |a^\ell - b^\ell| \leq \ell |a - b| \) and Taylor's theorem, it follows that
    \begin{equation*}
        \begin{split}   &\bigl|\theta^\ell - e^{-\ell  (1 - \theta)} \bigr|
            \leq \ell |\theta - e^{-(1-\theta)}|
            \leq \frac{\ell  (1 - \theta)^2 }{2}
            \leq  \frac{ \ell (K^*)^2 \bigl( \lambda(\beta)^{ 12}\bigr)^2 }{2}
            \leq  \frac{ e^{2K^*  \ell\lambda(\beta)^{12}}}{2\ell}.
        \end{split}
    \end{equation*} 
    Combining this with~\eqref{eq: combined bounds}, the desired conclusion follows.
\end{proof}

\begin{remark}\label{remark: other bounds}
The estimate in Proposition~\ref{proposition: 7.1} plays well with the estimate in Proposition~\ref{proposition: 7.2} below. However, some of the inequalities used in the proof of Proposition~\ref{proposition: 7.1} are far from optimal, especially when \( \ell \lambda(\beta)^{12} \) is small. In particular, we several times use the estimate \( x \leq e^x \), applied with \( x = K^* \ell \lambda(\beta)^{12} \), which when \( \ell \lambda(\beta)^{12}  \) is small replaces something very small with something close to one.  Alternative upper bounds can easily be obtained by replacing the estimates in the last equations of the proof of Proposition~\ref{proposition: 7.1}.
 An example of such an upper bound on \( \bigl|\mathbb{E}_{N,\beta} [ W_\gamma]-  \theta^\ell \bigr|  \), which is valid when \( \ell \lambda(\beta)^{12} < 1 \), is the following:
 \begin{equation*} 
        \begin{split}
            \bigl| \mathbb{E}_{N,\beta} [W_\gamma]- \theta^\ell \bigr| 
            & \leq  
         \bigl\{ (K_0^{(6)}+ 2K^* ) e^{K^*}    + 4K_1K_0^{(7)} + 2K_0^{(6)} +2 K_0^{(25)}  \bigr\} \Bigl[ \sqrt\frac{\ell_c}{\ell} +  \lambda(\beta)^2 \Bigr] \, \ell \lambda(\beta)^{12}  .
         \end{split}
     \end{equation*} 
      When \( \ell \lambda(\beta)^{12} \) is small, the right-hand side of this inequality is much smaller than the right-hand side of~\eqref{eq: Chatterjee's main equation}, enabling us describe the behaviour of \( \mathbb{E}_{N,\beta} [W_\gamma] \) in greater precision than what follows from Theorem~\ref{theorem: Chatterjee's main theorem}.
\end{remark}

\subsection{\texorpdfstring{Wilson loop expectations when \(\ell \lambda(\beta)^{12} \) is large}{Long loops}}

\label{sec: long loops}
The main purpose of this section is to state and prove Proposition~\ref{proposition: 7.2} below, which generalizes Lemma~7.12 in \cite{c2019}, and is useful when \( \ell \lambda(\beta)^{12} \) is large. 

Throughout this section, we will use the following notation.
Let \( K \) be a finite and non-empty index set.
Given an element \( g_k \in G \) for each \( k \in K \),  we define a set $G_0[(  g_k )_{k \in K}] $ by 
\begin{equation*}
    G_0[( g_k )_{k \in K}] \coloneqq \argmax_{g \in G}  \prod_{k \in K} \phi_\beta(g +  g_k).
\end{equation*} 
In the lemmas in this section, we will need the following key assumption on \( \beta_0 = \beta_0(|K|) \).
\begin{enumerate}[label=($\star$)]
    \item For all \( \beta \geq \beta_0 \), all choices of \(  g_k \in G \) for \( k \in K \), and any \(  g' \in G_0[(  g_k )_{k \in K}]\), we have
    \begin{equation}\label{eq: key assumption}
        \sum_{g\in  G \smallsetminus G_0[(  g_k )_{k \in K}]} \prod_{k\in K} \frac{\phi_\beta( g+ g_k)^2}{ \phi_\beta(  g'+ g_k )^2 }  \leq \frac{1 - \cos (2 \pi/n)}{8},
    \end{equation}\label{item: key assumption}
\end{enumerate}
We now show that for any finite and non-empty set \( K \), \ref{item: key assumption} is satisfied for all sufficiently large \( \beta_0 \geq 0\).
To this end, let \( \beta \geq 0 \), \( g \in G \), and let \( K \) be a finite and non-empty set. Further, for each \(k \in K \) let \( g_k \in G \). Then, by definition, we have
\[
    \prod_{k \in K} \phi_\beta(g +  g_k) 
    = 
    \exp \Bigl( \beta \sum_{k \in K} \Re \bigl(\rho(g+ g_k)\bigr) \Bigr).
\]
and hence the set \( G_0[(  g_k )_{k \in K}] \) is independent of the choice of \( \beta \).
Since for any \( g' \in G_0[( g_k )_{k \in K}] \) and  \( g \in G \smallsetminus G_{0}[( g_k )_{k \in K}] \) we have
\begin{equation*}
    \sum_{k \in K} \Re \bigl(\rho(g'+g_k)\bigr) 
    > \sum_{k \in K} \Re \bigl(\rho(g+g_k)\bigr),
\end{equation*}
it follows that, as \( \beta \to \infty \),
\begin{equation}
    \lim_{n \to \infty} \prod_{k \in K} \phi_\beta(g' + g_k)^2  \bigg/ \prod_{k \in K} \phi_\beta(g + g_k)^2 = \infty
\end{equation} 
This implies in particular that the inequality in~\eqref{eq: key assumption} is satisfied for all sufficiently large \( \beta \), and hence~\ref{item: key assumption} holds for all sufficiently large \( \beta_0 \).

We are now ready to state the main result of this section.
\begin{proposition}\label{proposition: 7.2}
    Consider lattice gauge theory with structure group $G = \mathbb{Z}_n$ and a faithful one-dimensional representation \( \rho \) of \( G \). Let \( \gamma \) be a generalized loop in \( \mathbb{Z}^4 \), let $\ell = |\support \gamma|$, and let \( \ell_c = |\support \gamma_c|\).
    Let \( N \) be sufficiently large so that the edges of \( \gamma \) are all internal edges of \( B_N \).
    Finally, let \( \beta_0 >0 \)  satisfy~\ref{item: key assumption} when applied with sets \( K \) with \( |K| = 6 \), and be such that \(2\lambda(\beta_0)^{2|K|} \leq 1 \). Then, for all \( \beta > \beta_0 \), we have 
    \begin{equation}\label{mubetaCstar}
        \bigl|\mathbb{E}_{N,\beta}[ W_\gamma ] \bigr| \leq  e^{-K_*(\ell-\ell_c) \lambda(\beta)^{12}}, \quad \text{where} \quad 
        K_* \coloneqq 
        \frac{1-\cos \bigl(2 \pi /n \bigr)}{4}.
    \end{equation}
\end{proposition}

Before we give a proof of Proposition~\ref{proposition: 7.2}, we will state and prove a few technical lemmas, which will be important in the proof of Proposition~\ref{proposition: 7.2}. 

In order to simplify the notation in the lemmas and proofs below, for a finite and non-empty set \( K \), \( g_k \in G \) for \( k \in K \), and \( \beta \geq 0 \),  we define
\begin{equation}\label{eq: Sbeta def}
    S_\beta\bigl(( g_k)_{k \in K} \bigr) \coloneqq \frac{\sum_{g \in  G} \rho(g)\prod_{k\in K}\phi_\beta( g+g_k)^2}{\sum_{g \in  G}    \prod_{k\in K}\phi_\beta( g+g_k)^2    }.
\end{equation}

\begin{lemma}\label{lemma: Sbeta symmetry}
    Let \( \beta \geq 0, \) and let \( K \) be a finite and non-empty index set. 
    For each \( k \in K \), let \( g_k \in G \). Then, for any \( g' \in G \), we have
    \begin{equation*}
        \Bigl| S_\beta\bigl(( g'+g_k)_{k \in K} \bigr)\Bigr|
        =
        \Bigl| S_\beta\bigl((g_k)_{k \in K} \bigr)\Bigr|
    \end{equation*}
\end{lemma}

\begin{proof}
    Fix some \( g' \in G. \) Then
    \begin{equation*}
        \begin{split}
            &   
            \Bigl| S_\beta\bigl(( g'+g_k)_{k \in K} \bigr)\Bigr|
            =
            \Biggl| \frac{  \sum_{g \in  G} \rho(g)  \prod_{k \in K}\phi_\beta \bigl( g+(g'+g_k)\bigr)^2}{\sum_{g \in  G}    \prod_{k \in K}\phi_\beta\bigl( g+(g'+g_k)\bigr)^2    } \Biggr| 
            \\&\qquad=
            \bigl| \rho(g')\bigr| \Biggl| \frac{  \sum_{g \in  G} \rho(g)  \prod_{k \in K}\phi_\beta \bigl( g+(g'+g_k)\bigr)^2}{\sum_{g \in  G}    \prod_{k \in K}\phi_\beta \bigl( g+(g'+g_k)\bigr)^2    } \Biggr| 
            =
            \Biggl| \frac{\sum_{g \in  G} \rho(g) \rho(g') \prod_{k \in K}\phi_\beta \bigl( (g+g')+g_k \bigr)^2}{\sum_{g \in  G}    \prod_{k \in K}\phi_\beta \bigl( (g+g')+g_k \bigr)^2    } \Biggr| 
            \\&\qquad =
            \Biggl| \frac{\sum_{g \in  G} \rho(g+g') \prod_{k \in K}\phi_\beta\bigl( (g+g')+g_k\bigr)^2}{\sum_{g \in  G} \prod_{k \in K}\phi_\beta \bigl( (g+g')+g_k\bigr)^2} \Biggr| 
            =
            \Biggl| \frac{\sum_{g \in  G} \rho(g )\prod_{k \in K}\phi_\beta( g+g_k)^2}{\sum_{g \in  G}    \prod_{k \in K}\phi_\beta( g+g_k)^2    } \Biggr|  
            =\Bigl| S_\beta\bigl(( g_k)_{k \in K} \bigr) \Bigr|
            .
        \end{split}
    \end{equation*}
    This completes the proof.
\end{proof}

\begin{lemma}\label{lemma: very crude estimate}
    Let \( \beta \geq 0, \) and let \( K \) be a finite and non-empty index set. 
    For each \( k \in K \), let \( g_k \in G \). Assume that \( G_0[(g_k)_{k \in K}] = \{ 0 \}\) and that~\ref{item: key assumption} holds.  
    Then
    \begin{equation}\label{eq: first term estimate i}
        \frac{1}{2} \leq \frac{\prod_{k\in K}\phi_\beta( 0+g_k)^2  }{\sum_{g \in  G}    \prod_{k\in K}\phi_\beta( g+g_k)^2    } \leq  1.
    \end{equation}  
\end{lemma}

\begin{proof}
    By~\ref{item: key assumption}, using that \( G_0[(g_k)_{k \in K}] = \{ 0 \}\), we have
    \begin{equation}\label{eq: assumption of claim i}
        \frac{\sum_{g\in  G \smallsetminus \{ 0 \}}  \prod_{k\in K}\phi_\beta( g+g_k)^2}{ \prod_{k\in K} \phi_\beta( 0+g_k )^2 }  \leq \frac{1 - \cos (2 \pi/n)}{8}.
    \end{equation}
    Since
    \begin{equation*} 
        \frac{\prod_{k\in K}\phi_\beta( 0+g_k)^2  }{\sum_{g \in  G}    \prod_{k\in K}\phi_\beta( g+g_k)^2    } = 1 - \frac{\sum_{g \in  G \smallsetminus \{ 0 \} }    \prod_{k\in K}\phi_\beta( g+g_k)^2    }{\sum_{g \in  G }    \prod_{k\in K}\phi_\beta( g+g_k)^2    }
        > 
        1 - \frac{\sum_{g \in  G \smallsetminus \{ 0 \} }    \prod_{k\in K}\phi_\beta( g+g_k)^2    }{  \prod_{k\in K}\phi_\beta( 0+g_k)^2    },
    \end{equation*}
    it follows from~\eqref{eq: assumption of claim i} that
    \begin{equation*}
        1 \geq \frac{\prod_{k\in K}\phi_\beta( 0+g_k)^2  }{\sum_{g \in  G}    \prod_{k\in K}\phi_\beta( g+g_k)^2    } \geq \frac{3}{4} \geq \frac{1}{2}
    \end{equation*}  
    as desired.
\end{proof}

\begin{lemma}\label{lemma: assumption holds if} 
    Let \( \beta \geq 0, \) and let \( K \) be a finite and non-empty index set. 
    For each \( k \in K \), let \( g_k \in G \). Assume that \( G_0[(g_k)_{k \in K}] = \{ 0 \}\), and that~\ref{item: key assumption} holds. Then
    \begin{equation}\label{eq: inequality trick}
        \bigl| S_\beta \bigl(( g_k)_{k \in K} \bigr) \bigr| \leq  \frac{1 + \Re S_\beta \bigl(( g_k)_{k \in K} \bigr)}{2}.
    \end{equation}
\end{lemma}

\begin{proof}
    To simplify notation, let
    \begin{equation}\label{eq: decomposition 1}
        w \coloneqq
        1 - S_\beta \bigl(( g_k)_{k \in K} \bigr),
    \end{equation}

    Recall that by assumption,~\ref{item: key assumption} holds, and hence
    \begin{equation}\label{eq: assumption of claim}
        \frac{\sum_{g\in  G \smallsetminus \{ 0 \}}  \prod_{k\in K}\phi_\beta( g+g_k)^2}{ \prod_{k\in K} \phi_\beta( 0+g_k )^2 }  \leq \frac{1 - \cos (2 \pi/n)}{8}.
    \end{equation}
    Also, by Lemma~\ref{lemma: very crude estimate}, we have
    \begin{equation}\label{eq: first term estimate ii}
        1 \geq \frac{\prod_{k\in K}\phi_\beta( 0+g_k)^2  }{\sum_{g \in  G}    \prod_{k\in K}\phi_\beta( g+g_k)^2    } \geq \frac{1}{2}.
    \end{equation}   
    Next, since \(\rho(G) = \{e^{2 \pi ij/n} \}_{j =1}^n\) consists of the \(n^{\textrm{th}}\) roots of unity and \(\rho(0) =1\), we have 
    \begin{equation}\label{eq: eq 1i}
        \min_{g \in G \smallsetminus \{ 0 \}}\bigl(1-\Re \rho (g )\bigr) = 1 - \cos(2 \pi/n) > 0
    \end{equation}  
    and 
    \begin{equation}\label{eq: eq 2i}
        \max_{g \in G  }\bigl|1- \rho (g )\bigr| \leq 2.
    \end{equation}   
    Since
    \begin{equation*}
        w = \biggl( \sum_{g\in  G} \bigl(1-\rho (g )\bigr)\prod_{k\in K} \frac{\phi_\beta( g+g_k)^2}{ \phi_\beta( 0+g_k )^2 } \biggr)
        \cdot
        \frac{\prod_{k\in K}\phi_\beta( 0+g_k)^2  }{\sum_{g \in  G}    \prod_{k\in K}\phi_\beta( g+g_k)^2    },
    \end{equation*}
    combining these observations, we obtain
    \begin{align*}
        &\Re w =  
        \biggl( \sum_{g\in  G\smallsetminus \{ 0 \}} \bigl(1-\Re \rho (g )\bigr)\prod_{k\in K} \frac{\phi_\beta( g+g_k)^2}{ \phi_\beta( 0+g_k )^2 } \biggr)
        \cdot
        \frac{\prod_{k\in K}\phi_\beta( 0+g_k)^2  }{\sum_{g \in  G}    \prod_{k\in K}\phi_\beta( g+g_k)^2    }
        \\&\qquad \overset{\eqref{eq: first term estimate ii},\eqref{eq: eq 1i}}{\geq}
        \bigl( 1 - \cos(2 \pi/n)\bigr)  \sum_{g\in  G\smallsetminus \{ 0 \}} \prod_{k\in K} \frac{\phi_\beta( g+g_k)^2}{ \phi_\beta( 0+g_k )^2 } 
        \cdot
        \frac{1}{2},
    \end{align*}
    and
    \begin{equation*}
        \begin{split}
            &|w|=  
            \biggl| \biggl( \sum_{g\in  G \smallsetminus \{ 0 \}} \bigl(1- \rho (g )\bigr)\prod_{k\in K} \frac{\phi_\beta( g+g_k)^2}{ \phi_\beta( 0+g_k )^2 } \biggr)
            \cdot
            \frac{\prod_{k\in K}\phi_\beta( 0+g_k)^2  }{\sum_{g \in  G}    \prod_{k\in K}\phi_\beta( g+g_k)^2    } \biggr|
            \\&\qquad\leq
            \max_{g \in G} \bigl| 1- \rho (g )\bigr|  \sum_{g\in  G \smallsetminus \{ 0 \}} \prod_{k\in K} \frac{\phi_\beta( g+g_k)^2}{ \phi_\beta( 0+g_k )^2 }  
            \cdot
            \frac{\prod_{k\in K}\phi_\beta( 0+g_k)^2  }{\sum_{g \in  G}    \prod_{k\in K}\phi_\beta( g+g_k)^2    }  
            \\&\qquad\overset{\eqref{eq: first term estimate ii},\eqref{eq: eq 2i}}{\leq}
            2  \sum_{g\in G \smallsetminus \{ 0 \}} \prod_{k\in K} \frac{\phi_\beta( g+g_k)^2}{ \phi_\beta( 0+g_k )^2 } 
            \cdot
            1 .
        \end{split}
    \end{equation*}
    Using~\eqref{eq: assumption of claim}, we obtain \( |w|^2 \leq \Re w \). 
    Finally, since 
    \begin{equation*}
        |1-w| = \sqrt{1 - 2  \Re w + |w|^2} \leq  1 -   \Re w + |w|^2/2.
    \end{equation*}
    it follows that 
    \begin{equation*} 
        |1-w| \leq 1 - \frac{\Re w}{2}.
    \end{equation*}
    Recalling the definition of \( w \), the desired conclusion now immediately follows.
\end{proof}

\begin{lemma}\label{lemma: old sublemma}
    Let \( \beta \geq 0, \) and let \( K \) be a finite and non-empty index set. 
    For each \( k \in K \), let \( g_k \in G \), and assume that \(  G_0\bigl[ (g_k)_{k \in K}\bigr] = \{ 0 \} \). Then 
    \begin{equation}\label{eq: irritating equation in new lemma i}
        \max_{g \in G \smallsetminus \{0\}}\prod_{k\in K} \frac{\phi_\beta( g+0)}{ \phi_\beta( 0+0 )}
        \leq 
        \max_{g \in G \smallsetminus \{0\}}\prod_{k\in K} \frac{\phi_\beta( g+g_k)}{ \phi_\beta( 0+g_k )}.
    \end{equation}
\end{lemma}
    
\begin{proof}
    By definition, for any \( g \in G \), we have
    \begin{equation}\label{eq: blaha 2}
        \prod_{k \in K}\phi_\beta(  g+g_k)
        =
        \prod_{k \in K} \exp \bigl( \beta \Re  \rho( g + g_k) \bigr)
        =
        \exp \Bigl( \beta \Re  \bigl(  \rho ( g)\sum_{k \in K}\rho(g_k)   \bigr) \Bigr).
    \end{equation}
    It follows that~\eqref{eq: irritating equation in new lemma i} is equivalent to 
    \begin{equation}\label{eq: new goal i}
        \begin{split}
            &\max_{g \in G \smallsetminus \{0\}}
            \Bigl[ \Re \bigl( \rho(g) \sum_{k\in K}\rho(g_k) \bigr)- \Re \bigl( \rho(0) \sum_{k\in K}\rho(g_k) \bigr) \Bigr] 
            \geq 
            \max_{g \in G \smallsetminus \{0\}} \Bigl[  \Re \bigl( \rho(g) \sum_{k\in K}\rho(0) \bigr)- \Re \bigl( \rho(0) \sum_{k\in K}\rho(0) \bigr) \Bigr].
        \end{split}
    \end{equation} 
    Let us prove \eqref{eq: new goal i}.
    Recall that \(  G_0\bigl[ (g_k)_{k \in K}\bigr] = \{ 0 \} \) by assumption, and note that this implies that
    \begin{equation*} 
        \argmax_{g \in G\smallsetminus \{ 0 \}}  \prod_{k \in K} \phi_\beta(g +  g_k) \subseteq \rho^{-1} \bigl( \{ e^{2 \pi i/n}, e^{-2 \pi i/n}\}\bigr) = \argmax_{g \in G\smallsetminus \{ 0 \}}  \prod_{k \in K} \phi_\beta(g +  0).
    \end{equation*} 
    For two non-zero complex numbers \( z_1 \) and \( z_2 \), let \( \angle (z_1 , z_2) \)  denote the absolute value of the smallest angle between  \( z_1 \) and \( z_2 \).
    Using this notation, for any \( g \in G \), we have
    \begin{equation}\label{eq: angle equation ii} 
        \begin{split}
            &
            \Re  \Bigl( \rho(g) \sum_{k\in K}\rho( g_k) \Bigr)
            =
            \Re  \Bigl( \sum_{k\in K} \overline{\rho(- g_k)}  \rho(g) \Bigr)
            =
            \Bigl| \sum_{k\in K} {\rho(g_k)}   \Bigr| \cos \biggl( \angle \Bigl(  \sum_{k\in K}\rho(- g_k) ,  \rho(g)  \Bigr) \biggr).
        \end{split}
    \end{equation}   
    Now fix some \( \hat g \in G \smallsetminus \{ 0 \} \). 
    Since \( \hat g \notin G_0[(g_k)_{k \in K}] = \{ 0 \} \), we have
    \begin{equation*}
        \prod_{k \in K} \phi_\beta(\hat g + g_k) < \prod_{k \in K} \phi_\beta(0 + g_k).
    \end{equation*}
    Using~\eqref{eq: blaha 2} and~\eqref{eq: angle equation ii}, we see that this is equivalent to
    \begin{equation*}\label{eq: bad 1}
        \angle \biggl( \sum_{k \in K} \rho(-g_k), \rho(0) \biggr) < \angle \biggl( \sum_{k \in K} \rho(-g_k), \rho(\hat g) \biggr),
    \end{equation*}
    and hence, since \( \rho(0)=1 \), we obtain
    \begin{equation}\label{eq: negative i}
        \cos \Bigl( \angle \bigl(  \sum_{k\in K}\rho(-g_k), \rho(\hat g) \bigr) \Bigr)  -\cos \Bigl( \angle \bigl( \sum_{k\in K}\rho(-g_k) , 1 \bigr) \Bigr) < 0.
    \end{equation} 
    On the other hand, using~\eqref{eq: angle equation ii}, the definition of \( G_0[( g_k )_{k \in K}] \), and the assumption that \( G_0[(g_k)_{k \in K}] = \{ 0 \} \), we have
    \begin{align*}
        &\{ 0 \} = G_0[( g_k )_{k \in K}] = \argmax_{g \in G}  \prod_{k \in K} \phi_\beta(g +  g_k) 
        =
        \argmax_{g \in G} \prod_{k \in K}e^{\beta \Re \rho(g+g_k)}
        \\&\qquad =
        \argmax_{g \in G}  e^{\beta \Re \rho(g) \sum_{k \in K} \rho(g_k)}
        \overset{\eqref{eq: angle equation ii}}{=} \argmax_{g \in G} \cos \biggl( \angle \Bigl( \, \sum_{k\in K}\rho(- g_k) ,  \rho(g)  \Bigr)
        \biggr)
        \\&\qquad=\argmin_{g \in G}   \angle \Bigl( \,  \sum_{k\in K}\rho(- g_k) ,  \rho(g)  \Bigr) .
    \end{align*}
    Since by assumption, we have \( G_0[( g_k )_{k \in K}] = \{ 0 \} \), we analogously have
    \begin{equation*}
        \hat g \in   \argmax_{g \in G \smallsetminus G_0[( g_k )_{k \in K}]}  \prod_{k \in K} \phi_\beta(g +  g_k) 
        =\argmin_{g \in G \smallsetminus \{ 0 \}}   \angle \Bigl(  \sum_{k\in K}\rho(- g_k) ,  \rho(g)  \Bigr) .
    \end{equation*}
    Consequently, we must be in one of the situations displayed in Figure~\ref{fig: angles}.
    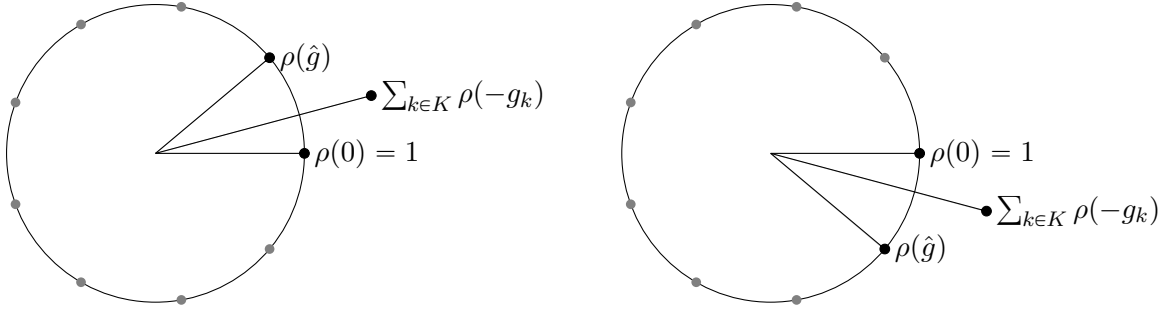
\begin{figure}[htp]
        \centering 
        \hspace{-2cm}
        \begin{subfigure}[b]{0.4\textwidth}
            \centering
            \begin{tikzpicture}[scale=0.99] 
                \draw[white] (-3,0) -- (4,0);
                  \draw[black] (0,0) circle (2); 
                
                \foreach \phi in {0,40,...,360}{
                    \fill[gray]  (\phi:2)  circle (0.065);
                }
                
                \filldraw[black] (0,0) -- (2,0) circle (0.065) node[anchor=west] {\( \rho(0)=1\)};
                \filldraw[black] (0,0) -- (40:2) circle (0.065) node[anchor=west] {\( \rho(\hat g)\)};
                \filldraw[black] (0,0) -- (15:3) circle (0.065) node[anchor=west] {\( \sum_{k \in K} \rho(-g_k)\)};
            \end{tikzpicture}
        \end{subfigure}
        \hspace{0.08\textwidth}
        \begin{subfigure}[b]{0.4\textwidth}
            \centering
            \begin{tikzpicture}[scale=0.99]
                \draw[white] (-3,0) -- (4,0);
                \draw[black] (0,0) circle (2); 
                
                \foreach \phi in {0,40,...,360}{
                    \fill[gray]  (\phi:2)  circle (0.065);
                }
                
                \filldraw[black] (0,0) -- (2,0) circle (0.065) node[anchor=west] {\( \rho(0)=1\)};
                \filldraw[black] (0,0) -- (-40:2) circle (0.065) node[anchor=west] {\( \rho(\hat g)\)};
                \filldraw[black] (0,0) -- (-15:3) circle (0.065) node[anchor=west] {\( \sum_{k \in K} \rho(-g_k)\)};
            \end{tikzpicture}
        \end{subfigure}
        \caption{The two figures show the points in \( \rho(G) = \{ e^{2 \pi i j/n} \} _{j \in \{ 1,2, \ldots, n \}} \) together with the two possible options for the point \( \sum_{k \in K} \rho(-g_k) \) under the assumption that \( G_0[(g_k)_{k \in K}] = \{ 0 \} \).}
        \label{fig: angles}
    \end{figure}%
    This implies in particular that 
    \begin{equation*}
        \angle \bigl(  \sum_{k\in K}\rho(-g_k), 1  \bigr) +  \angle \bigl(  \sum_{k\in K}\rho(-g_k), \rho(\hat g)  \bigr)  =  \angle \bigl( 1, \rho(\hat g)  \bigr).
    \end{equation*}
    Since \( \cos(x) \) is decreasing in \( x \) for \( x \in (0,\pi) \), it follows that
    \begin{align*}
        & 
        \cos\Bigl( \angle \bigl( 1, 1 \bigr) \Bigr)
        +
        \cos\Bigl( \angle \bigl(  \sum_{k\in K}\rho(-g_k), \rho(\hat g) \bigr) \Bigr)
        \geq 
        \cos\Bigl( \angle \bigl(  \sum_{k\in K}\rho(-g_k)  ,1\bigr) \Bigr)
        +
        \cos \Bigl( \angle \bigl( 1, \rho(\hat g) \bigr) \Bigr),
    \end{align*} 
    or, equivalently, that
    \begin{equation}\label{eq: the previous equation i}
        \begin{split} 
            &\cos \Bigl( \angle \bigl(  \sum_{k\in K}\rho(-g_k), \rho(\hat g) \bigr) \Bigr)  -\cos \Bigl( \angle \bigl( \sum_{k\in K}\rho(-g_k) , 1 \bigr) \Bigr) 
            \geq 
            \cos \Bigl( \angle \bigl( 1, \rho(\hat g) \bigr) \Bigr)  -\cos \Bigl( \angle \bigl( 1, 1 \bigr) \Bigr).
        \end{split} 
    \end{equation}
    Next, since \( \rho \) is unitary, we have 
    \begin{equation}\label{eq: unitary consequnce}
        \bigl| \sum_{k\in K}\rho(g_k) \bigr| \leq |K|  .
    \end{equation}
    Combining the above equations, we finally obtain
    \begin{equation*} 
        \begin{split}
            & \Re \bigl( \rho(\hat g) \sum_{k\in K}\rho(g_k) \bigr)- \Re  \bigl( \rho(0) \sum_{k\in K} \rho (g_k)  \bigr)
            \\&\qquad \overset{\eqref{eq: angle equation ii}}{=}
            \Bigl| \sum_{k\in K}\rho(g_k) \Bigr|  \biggl( \cos \Bigl( \angle \bigl( \sum_{k\in K}\rho(-g_k), \rho(\hat g) \bigr) \Bigr)  -\cos \Bigl( \angle \bigl( \sum_{k\in K}\rho(-g_k) , 1 \bigr) \Bigr)   \biggr) 
            \\&\qquad\overset{\eqref{eq: negative i},\eqref{eq: unitary consequnce}}{\geq}
            |K| \biggl( \cos \Bigl( \angle \bigl( \sum_{k\in K}\rho(-g_k), \rho(\hat g) \bigr) \Bigr)  -\cos \Bigl( \angle \bigl( \sum_{k\in K}\rho(-g_k) , 1 \bigr) \Bigr)   \biggr) 
            \\&\qquad \overset{\eqref{eq: the previous equation i}}{\geq}
            |K|  \biggl( \cos \Bigl( \angle \bigl( 1, \rho(\hat g) \bigr) \Bigr)  -\cos \Bigl( \angle \bigl( 1, 1 \bigr) \Bigr)   \biggr)
            \\&\qquad=
            \Bigl| \sum_{k\in K}\rho(0) \Bigr| \biggl( \cos \Bigl( \angle \bigl(  \sum_{k\in K}\rho(0), \rho(\hat g) \bigr) \Bigr)  -\cos \Bigl( \angle \bigl( \sum_{k\in K}\rho(0) , \rho(0) \bigr) \Bigr)   \biggr)
            \\&\qquad\overset{\eqref{eq: angle equation ii}}{=}
            \Re \bigl( \rho(\hat g) \sum_{k\in K}\rho(0) \bigr)- \Re \bigl( \rho(0) \sum_{k\in K} \rho(0) \bigr).
        \end{split}
    \end{equation*}   
    Since \( \hat g \in G \smallsetminus \{ 0 \}\) was arbitrary, this implies \eqref{eq: new goal i} and the desired conclusion follows.
\end{proof}

\begin{lemma}\label{lemma: the last but crucial inequality}
    Let \( K \) be a finite and non-empty index set.
        Assume that \( \beta_0 = \beta_0(|K|) > 0\)  satisfies~\ref{item: key assumption} applied with \( K \), and that
    \begin{equation}\label{eq: second assumption of lemma}
         2\lambda(\beta_0)^{2|K|} \leq 1.
    \end{equation}
    Then, for all choices of \( g_k \in G \) for \( k \in K \), in the setting of Proposition~\ref{proposition: 7.2}, we have  
    \begin{equation} \label{eq: the last but crucial inequality}
        \Bigl| S_\beta\bigl(( g_k)_{k \in K} \bigr)\Bigr| 
        \leq  
        1-K_*   \lambda(\beta)^{2|K|},
    \end{equation} 
    where $K_*$ is the constant defined in (\ref{mubetaCstar}).
\end{lemma}

\begin{proof}  
    Since \( \rho(G) = \{ e^{2 \pi ij/n} \}_{j =1}^n \) and 
    \begin{equation}\label{eq: blaha}
        \prod_{k \in K} \phi_\beta(g +  g_k) 
        = 
        \exp \Bigl( \beta \sum_{k \in K} \Re \bigl(\rho(g+ g_k)\bigr) \Bigr)
        = 
        \exp \Bigl( \beta  \Re \bigl(\rho(g)  \sum_{k \in K}\rho(g_k)\bigr) \Bigr), 
    \end{equation}
    we have either (1) \(  G_0[(  g_k )_{k \in K}] = G, \)
    (2) \( \bigl|G_0[(  g_k )_{k \in K}] \bigr| = 1 ,\) or (3) \( \bigl|G_0[(  g_k )_{k \in K}] \bigr| = 2. \) We now divide into three cases accordingly.

    \vfil\noindent \textbf{Case 1.} 
    Assume that \(  G_0[(  g_k )_{k \in K}] = G. \) Then
    \begin{equation*}
        \sum_{g\in  G \smallsetminus G_0[( \hat g_k )_{k \in K}]} \prod_{k\in K} \frac{\phi_\beta( g+\hat g_k)^2}{ \phi_\beta( 0+\hat g_k )^2 } = 0,
    \end{equation*}
        and hence~\eqref{eq: the last but crucial inequality} trivially holds in this case.

    \vfil\noindent \textbf{Case 2.} 
    Assume that \( \bigl|G_0[ ( g_k )_{k \in K}] \bigr| = 2 \). Let \( g' \in G_0[ ( g_k )_{k \in K}]  \). Since \( |\rho(g)|\leq 1 \) for all \( g \in G \), we have
    \begin{align*} 
        &\bigl| S_\beta\bigl( ( g_k )_{k \in K} \bigr) \bigr| 
        = 
        \biggl| \frac{\sum_{g \in  G} \rho(g)\prod_{k\in K}\phi_\beta( g+g_k)^2}{\sum_{g \in  G}    \prod_{k\in K}\phi_\beta( g+g_k)^2    }
        \biggr|
        \leq
        \biggl| \frac{\sum_{g \in  G} \rho(g)\prod_{k\in K}\phi_\beta( g+g_k)^2}{\sum_{g \in  G_0[(g_k)_{k\in K}]}    \prod_{k\in K}\phi_\beta( g+g_k)^2    }
        \biggr|
        \\&\qquad\leq 
        \biggl| \frac{\sum_{g \in  G_0[(g_k)_{k\in K}]} \rho(g)\prod_{k\in K}\phi_\beta( g+g_k)^2}{\sum_{g \in G_0[ (  g_k )_{k \in K}]}    \prod_{k\in K}\phi_\beta( g+g_k)^2    }
        \biggr|
        +
        \biggl| \frac{\sum_{g \in  G\smallsetminus G_0[(g_k)_{k\in K}]} \rho(g)\prod_{k\in K}\phi_\beta( g+g_k)^2}{\sum_{g \in G_0[ (  g_k )_{k \in K}]}    \prod_{k\in K}\phi_\beta( g+g_k)^2    }
        \biggr|
        \\&\qquad\leq 
        \frac{\bigl| \sum_{g \in G_0[ (  g_k )_{k \in K}]} \rho(g) \bigr|}{|G_0[ ( g_k )_{k \in K}]|}+ \sum_{g\in  G \smallsetminus G_0[( g_k )_{k \in K}]} \prod_{k\in K} \frac{\phi_\beta( g+g_k)^2}{ \phi_\beta( g'+g_k )^2 }.
    \end{align*} 
    Since \( \rho(G) = \{e^{2 \pi ij/n} \}_{j =1}^n \)  and \( \bigl|G_0[ (  g_k )_{k \in K}] \bigr| = 2 \), there must be some \( j \in \mathbb{Z}_n \) such that \( G_0[(g_k)_{k \in K}] = \{ j,j+1\} \). From this it follows that
    \begin{equation*} 
        \begin{split}
            &\frac{\bigl| \sum_{g \in G_0[ ( g_k )_{k \in K}]} \rho(g) \bigr|}{|G_0[ ( g_k )_{k \in K}]|} \leq \frac{\sqrt{(1 + \cos(2 \pi /n))^2 + \sin^2(2 \pi /n)}}{2}
            =
            \sqrt{\frac{1 + \cos(2 \pi /n)}{2}}
            \\&\qquad =
            \sqrt{1 - \frac{1 - \cos(2 \pi /n)}{2}}
            \leq
            1 -  \frac{1 - \cos(2 \pi /n)}{4}.
        \end{split}
    \end{equation*}
    Combining the previous equations and using~\ref{item: key assumption}, we get 
    \begin{equation*} 
        \bigl| S_\beta\bigl( ( g_k )_{k \in K} \bigr) \bigr| \leq 
        1 -  \frac{1 - \cos(2 \pi /n)}{8}.
    \end{equation*} 
    Using~\eqref{eq: second assumption of lemma}, and recalling the definition of \( K_* \), we obtain~\eqref{eq: the last but crucial inequality} in the case \( |G_0[ (  g_k )_{k \in K}] \bigr| = 2 \). 
    
    \vfil\noindent \textbf{Case 3.} 
    Now assume that \( \bigl|G_0[ ( g_k )_{k \in K}]\bigr| =1  \).
    In this case, both sides of~\eqref{eq: the last but crucial inequality} tend to \( 1 \) as \( \beta \to \infty \). Since \( \rho(G) = \{e^{2 \pi ij/n} \}_{j =1}^n \), for any fixed \( \beta > 0 \) both sides of~\eqref{eq: the last but crucial inequality} are strictly smaller than \( 1 \), and hence the desired conclusion will follow if we can show that the  convergence of the right-hand side is faster than that of the left-hand side.
    To this end, note first that, for any \( g' \in G \), by Lemma~\ref{lemma: Sbeta symmetry}, we have
    \begin{equation*}
        \bigl| S_\beta \bigl( (g_k + g')_{k \in K}\bigr)\bigr|
        =
        \bigl|S_\beta \bigl(( g_k )_{k \in K} \bigr) \bigr|
        .
    \end{equation*}
    This implies in particular that we can assume that \( 0 \in G_0[( g_k )_{k \in K}] \). Since \( \bigl| G_0[( g_k )_{k \in K}]  \bigr| = 1 \) by assumption, this implies that \( G_0[( g_k)_{k \in K}] = \{ 0 \} \). 
    By Lemma~\ref{lemma: assumption holds if}, we have 
      \begin{equation*} 
        \bigl|S_\beta \bigl(( g_k)_{k \in K} \bigr)\bigr| \leq \frac{1 + \Re S_\beta \bigl(( g_k)_{k \in K} \bigr)}{2},
    \end{equation*}
     or, equivalently, 
    \begin{align*} 
        &\bigl| S_\beta \bigl(  (g_k)_{k \in K} \bigr) \bigr|
        \leq 
        1 - \frac{1}{2} \biggl( \sum_{g\in  G\smallsetminus \{ 0 \}} \Re \bigl(1-\rho (g )\bigr)\prod_{k\in K} \frac{\phi_\beta( g+g_k)^2}{ \phi_\beta( 0+g_k )^2 } \biggr)
        \cdot
        \frac{\prod_{k\in K}\phi_\beta( 0+g_k)^2  }{\sum_{g \in  G}    \prod_{k\in K}\phi_\beta( g+g_k)^2    }.
    \end{align*} 
    Noting that \( \Re \bigl(1 - \rho(g) \bigr) \geq 1 - \cos(2 \pi/n) \) for all \( g \in G \smallsetminus \{ 0 \} \) and using~Lemma~\ref{lemma: very crude estimate}, we obtain
    \begin{equation*} 
        \bigl| S_\beta \bigl(  (g_k)_{k \in K} \bigr) \bigr|
        \leq 
        1 - \frac{1}{2} \bigl(1 - \cos(2 \pi/n) \bigr)  \max_{g \in G \smallsetminus \{0\}}
        \prod_{k\in K} \frac{\phi_\beta(  g+g_k)^2}{ \phi_\beta( 0+g_k )^2 } \cdot \frac{1}{2}. 
    \end{equation*} 
    Applying Lemma~\ref{lemma: old sublemma}, we thus find that
    \begin{equation*} 
        \begin{split}
            &\bigl| S_\beta \bigl(  (g_k)_{k \in K} \bigr) \bigr|
            \leq 
            1 - \frac{1}{2} \bigl(1 - \cos(2 \pi/n) \bigr)  \max_{g \in G \smallsetminus \{0\}}
            \prod_{k\in K} \frac{\phi_\beta(  g+0)^2}{ \phi_\beta( 0+0 )^2 } \cdot \frac{1}{2} 
            \\&\qquad=
            1 - \frac{1}{2} \bigl(1 - \cos(2 \pi/n) \bigr)  \lambda(\beta)^{2|K|}\cdot \frac{1}{2}.
        \end{split}
    \end{equation*} 
    This concludes the proof in the case \( \bigl|G_0[ ( g_k )_{k \in K}]\bigr| =1. \)
\end{proof}

We are now ready to give a proof of Proposition~\ref{proposition: 7.2}. The main difference between this proof and the proof of the corresponding result for $G = \mathbb{Z}_2$ in~\cite{c2019} is the use of Lemma~\ref{lemma: the last but crucial inequality}, which extends what in the case of \( G = \mathbb{Z}_2 \) was a very simple inequality.
 
\begin{proof}[Proof of Proposition~\ref{proposition: 7.2}]
    As in the proof of Proposition~\ref{proposition: 7.1}, let \( \gamma_1 \coloneqq \gamma-\gamma_c \), and let \( \mu_{N,\beta}' \) and \( \mathbb{E}_{N,\beta}'  \) denote conditional probability and conditional expectation given \( \bigl(\sigma(e)\bigr)_{e \not \in \pm \support \gamma_1 } \) respectively. As observed earlier, the spins \( \bigl(\sigma(e)\bigr)_{e \in  \gamma_1} \) are independent  under this conditioning.
 
    Take any \( e \in  \gamma_1 \). For \( p \in  \hat \partial e  \), let
    \begin{equation*}
        \sigma_p^e \coloneqq   \sum_{e' \in  \partial p\smallsetminus \{ e \}}  \sigma(e')
    \end{equation*}
    Then, for any \( g \in  G \), 
    \begin{equation*}
        \mu_{N,\beta}'\bigl( \sigma(e) = g\bigr) =  \frac{\prod_{p \in  \hat \partial e }\phi_\beta( \sigma_p^e + g)^2}{\sum_{g' \in  G}   \prod_{p \in \hat \partial e }\phi_\beta( \sigma_p^e + g')^2   }.
    \end{equation*} 
    It follows that the expected value of \( \rho\bigl(\sigma(e)\bigr) \) under \( \mu_{N,\beta}' \) is given by
    \begin{equation*}
        \begin{split}
            &\mathbb{E}_{N,\beta}'\bigl[\rho\bigl(\gamma_1[e]\sigma(e) \bigr) \bigr] =    \frac{\sum_{g \in  G} \rho(g )\prod_{p \in  \hat \partial e }\phi_\beta( \sigma_p^e + g)^2}{\sum_{g \in  G}    \prod_{p \in  \hat \partial e }\phi_\beta( \sigma_p^e + g)^2    },\quad e\in  \gamma_1.
        \end{split}
    \end{equation*}
    Since the edge spins \( \bigl(\sigma(e)\bigr)_{ e \in  \gamma_1} \) are independent given this conditioning (see the proof of Proposition~\ref{proposition: 7.1}), we have
    \begin{equation*}
        \mathbb{E}_{N,\beta}' \Bigl[\rho\bigl( \sigma(\gamma_1) \bigr)\Bigr] 
        =  
        \mathbb{E}_{N,\beta}' \Bigl[\rho\bigl(\, \sum_{e \in  \gamma_1} \sigma(e) \bigr)\Bigr] 
        =  
        \prod_{e \in  \gamma_1} \mathbb{E}_{N,\beta}' \bigl[ \rho \bigl( \sigma(e) \bigr)\bigr].
    \end{equation*}
    This implies in particular that
    \begin{equation}\label{eq: many products}
        \begin{split}
            &\bigl|\mathbb{E}_{N,\beta} [W_\gamma]\bigr| 
            = \biggl| \mathbb{E}_{N,\beta}\Bigl[ \rho \bigl(\sigma(\gamma)\bigr)\Bigr]\biggr|
            = \biggl| \mathbb{E}_{N,\beta}\Bigl[ \rho \bigl(\sigma(\gamma-\gamma_1+\gamma_1)\bigr)\Bigr]\biggr|
            = \biggl| \mathbb{E}_{N,\beta}\Bigl[ \rho \bigl(\sigma(\gamma-\gamma_1)\bigr) \rho \bigl(\sigma(\gamma_1)\bigr)\Bigr]\biggr|
            \\&\qquad 
            = \biggl|\mathbb{E}_{N,\beta}\Bigl[  \rho \bigl(\sigma ( \gamma- \gamma_1)\bigr)\,  \mathbb{E}_{N,\beta}' \bigl[  \rho \bigl( \sigma(\gamma_1)\bigr)\bigr]\Bigr]\biggr|
            \leq 
            \mathbb{E}_{N,\beta} \biggl[ \Bigl|    \rho \bigl( \sigma(\gamma-\gamma_1)\bigr)\,  \mathbb{E}_{N,\beta}' \bigl[  \rho \bigl( \sigma(\gamma_1)\bigr)\bigr]\Bigr|\biggr]
            \\&\qquad=
            \mathbb{E}_{N,\beta} \biggl[ \Bigl| \rho\bigl(  \sigma (\gamma- \gamma_1) \bigr) \Bigr| \cdot \Bigl| \mathbb{E}_{N,\beta}' \bigl[  \rho\bigl(  \sigma(\gamma_1) \bigr)\bigr]\Bigr| \biggr]
            = 
            \mathbb{E}_{N,\beta} \biggl[  \Bigl|  \mathbb{E}_{N,\beta}' \bigl[  \rho\bigl(\sigma(\gamma_1)\bigr)\bigr]\Bigr| \biggr] 
            \\&\qquad
            =
            \mathbb{E}_{N,\beta} \biggl[   \prod_{e \in \gamma_1}  \Bigl|\mathbb{E}_{N,\beta}' \bigl[  \rho\bigl(\sigma(e)\bigr)\bigr]\Bigr| \biggr].
        \end{split}
    \end{equation}
    For \( e \in  \gamma_1\), let \( K =  \hat \partial e  \), and for \( p \in  \hat \partial e \) define \( g_p \coloneqq \sigma_p^e \).
	Applying Lemma~\ref{lemma: the last but crucial inequality}, with \( K  \) and \( ( g_p )_{p \in K}  \), we obtain
    \begin{equation*}
        \begin{split} &\bigl|\mathbb{E}_{N,\beta}'[\rho\bigl(\sigma(e) \bigr) ] \bigr| 
            \leq  1 - K_* \lambda(\beta)^{12} .
        \end{split}
    \end{equation*}
    This implies in particular that 
    \begin{equation*}
        \begin{split}
            &\Bigl|\mathbb{E}_{N,\beta}' \bigl[ \rho\bigl(\sigma(e)\bigr) \bigr] \Bigr|  = 1 - K_* \lambda(\beta)^{12} \leq e^{-K_*\lambda(\beta)^{12}}.
        \end{split}
    \end{equation*}
    Inserting this into~\eqref{eq: many products}, we finally obtain
    \begin{equation*}
        \begin{split}
            &\bigl|\mathbb{E}_{N,\beta}  [W_\gamma]\bigr| 
            \leq  e^{-K_*(\ell-\ell_c) \lambda(\beta)^{12}}.
        \end{split}
    \end{equation*}
    This concludes the proof.
\end{proof}

 \subsection{Proof of the main result}
 We are now ready to give a proof of Theorem~\ref{theorem: Chatterjee's main theorem}.  
 
\begin{proof}[Proof of Theorem~\ref{theorem: Chatterjee's main theorem}]
    Define $K'$ and $K''$ by
    \begin{equation}\label{CpCppexpressions}
        K' \coloneqq \sqrt{2}\bigl( C_A  2^{4K^*/K_*}\bigr)^{1/(1 + 4K^*/K_*  )}, \qquad
        K'' \coloneqq 1/(1 + 4K^*/K_*).
    \end{equation}

    Assume first that \( \ell_c \geq \ell/2 \). 
    Since \( C_A \geq 9/2 \), $K_* >0$, and $K^*>0$, we have \( K' \geq 2\sqrt{2} \) and \(K'' \in ( 0,1) \). From this it follows that
    $$K' \Bigl[            \sqrt\frac{\ell_c}{\ell}
         +  \lambda(\beta)^2 \Bigr]^{K''} \geq 2\sqrt{2} \Bigl[  \frac{1}{\sqrt{2}} + 0
         \Bigr]^{K''} \geq 2^{3/2 - \frac{K''}{2}} \geq 2.$$
    Since \( |\rho(g)|=1 \) for each $g \in G$, we always have 
 	\begin{equation}
             \bigl|\mathbb{E}_{N,\beta} [W_\gamma]-  e^{-\ell (1 - \theta)}\bigr| 
          \leq 
          2,
     \end{equation}
     and hence the desired conclusion follows automatically in this case.
Consequently, we only need to prove the theorem for \( \ell_c < \ell/2 \).
 
 Suppose $\ell_c < \ell/2$. From Proposition~\ref{proposition: 7.1} we are given the upper bound
 	\begin{equation} \label{eq: ineq 1} 
             \bigl|\mathbb{E}_{N,\beta} [W_\gamma]-  e^{-\ell (1 - \theta)}\bigr| 
          \leq 
         C_A   e^{2K^* \ell \lambda(\beta)^{12}} 
             \biggl[ \sqrt\frac{\ell_c}{\ell} + \lambda(\beta)^2 \biggr].
     \end{equation}
At the same time,  Proposition~\ref{proposition: 7.2} implies that
 \begin{equation*}
     \bigl|\mathbb{E}_{N,\beta} [ W_\gamma ]\bigr| \leq   e^{-K_*(\ell-\ell_c) \lambda(\beta)^{12}}.
 \end{equation*}
 From the second of the above inequalities, using the triangle inequality, it follows that
 \begin{equation}\label{eq: first bound}
 \begin{split}
     &\bigl|\mathbb{E}_{N,\beta} [W_\gamma] -  e^{-\ell (1 - \theta)}\bigr| 
     \leq  
      e^{-\ell(1 - \theta )} + e^{-K_*(\ell-\ell_c) \lambda(\beta)^{12}}.
 \end{split}
 \end{equation} 
    Since $\ell_c < \ell/2$ by assumption, we have 
    \begin{equation}\label{eq: first ineq of two}
        \begin{split} 
            e^{-K_*(\ell-\ell_c) \lambda(\beta)^{12}} 
            \leq  
            e^{-\frac{1}{2} K_* \ell  \lambda(\beta)^{12}}.
        \end{split}
    \end{equation} 
For the other term on the right-hand side of~\eqref{eq: first bound}, note that
\begin{align*}
    &1 - \theta(\beta) 
    =
    1 - \frac{\sum_{g \in G} \rho(g) e^{12 \beta \Re \rho(g)}}{\sum_{g \in G} e^{12 \beta \Re \rho(g)}} 
    = 
    \frac{\sum_{g \in G} (1-\rho(g)) e^{12 \beta \Re \rho(g)}}{\sum_{g \in G} e^{12 \beta \Re \rho(g)}}
    = 
    \frac{\sum_{g \in G} (1-\rho(g)) e^{12 \beta (\Re \rho(g) - 1)}}{\sum_{g \in G} e^{12 \beta (\Re \rho(g) - 1)}}
    \\&\qquad\geq
     \frac{ (1 - \cos(2 \pi/n)) \lambda(\beta)^{12}}{1 + (|G|-1)  \lambda(\beta)^{12}}
    \geq
      \bigl(1 - \cos(2 \pi/n) \bigr) \lambda(\beta)^{12} \Bigl(1 - \bigl(|G|-1 \bigr)  \lambda(\beta)^{12} \Bigr)
    \\&\qquad=
      4K_* \lambda(\beta)^{12} \Bigl(1 - \bigl(|G|-1 \bigr)  \lambda(\beta)^{12} \Bigr).
\end{align*}
Since \( 5 (|G|-1) \lambda(\beta)^2 < 1 \) by assumption, we have \( 1 - \bigl(|G|-1 \bigr)  \lambda(\beta)^{12} \geq 1/2  \), and thus we obtain
\begin{align*}
    &1 - \theta(\beta) 
    \geq
      2K_* \lambda(\beta)^{12},
\end{align*}
and hence
\begin{equation}\label{eq: second ineq of two}
    e^{-\ell(1 - \theta )} \leq e^{ -    2K_*  \ell  \lambda(\beta)^{12} }.
\end{equation}
Combining~\eqref{eq: first bound}, \eqref{eq: first ineq of two} and~\eqref{eq: second ineq of two}, we obtain
 \begin{equation*}
 \begin{split}
     &\bigl|\mathbb{E}_{N,\beta} [W_\gamma] -  e^{-\ell (1 - \theta)}\bigr| 
     \leq   
       e^{ - 2 K_*  \ell  \lambda(\beta)^{12} } + e^{-\frac{1}{2} K_* \ell  \lambda(\beta)^{12}}  
     \leq  2 e^{ -  \frac{1}{2}K_* \ell  \lambda(\beta)^{12} }  .
 \end{split}
 \end{equation*}
 Combining this with~\eqref{eq: ineq 1}, we get
 \begin{equation*}
     \begin{split}
         &\bigl|\mathbb{E}_{N,\beta}[ W_\gamma  ] -   e^{-\ell (1 - \theta)} \bigr|^{1 + 4K^*/K_* } 
         \leq
     C_A e^{2K^* \ell \lambda(\beta)^{12}} \Bigl[  
          \sqrt\frac{\ell_c}{\ell}  
         +  \lambda(\beta)^2 \Bigr] \Bigl[ 2 e^{-\frac{1}{2}K_*  \ell  \lambda(\beta)^{12}} \Bigr]^{4K^*/K_* }
         \\&\qquad =
     C_A  2^{4K^*/K_*}  \Bigl[  
          \sqrt\frac{\ell_c}{\ell}  
         +  \lambda(\beta)^2 \Bigr].
     \end{split}
 \end{equation*}
 Rearranging, we obtain
 \begin{equation}\label{finalestimate}
     \begin{split}
         &\bigl|\mathbb{E}_{N,\beta} [W_\gamma]-  e^{-\ell (1 - \theta)} \bigr|
         \leq
     K' \Bigl[            \sqrt\frac{\ell_c}{\ell}  
         +  \lambda(\beta)^2 \Bigr]^{K''}.
     \end{split}
 \end{equation}
  Since neither \( e^{-\ell(1-\theta)} \) nor the right-hand side of \eqref{finalestimate} depends on \( N \), 
   and, by Theorem~\ref{theorem: limit exists}, the limit $\langle W_\gamma \rangle_\beta = \lim_{N \to \infty} \mathbb{E}_{N,\beta}[W_\gamma]$ exists (see~\eqref{wilsonlimit}), the desired conclusion follows.
 \end{proof}

\end{document}